\documentclass[12pt]{amsart}
\usepackage{latexsym, amssymb, amsfonts, amscd}

\theoremstyle{plain}
\newtheorem{Thm}{Theorem}[subsection]
\newtheorem{Cor}[Thm]{Corollary}

\newtheorem{Prop}[Thm]{Proposition}
\newtheorem{Lem}[Thm]{Lemma}

\newtheorem{Thm'}{Theorem}[section]
\newtheorem{Prop'}[Thm']{Proposition}
\newtheorem{Lem'}[Thm']{Lemma}
\newtheorem{Cor'}[Thm']{Corollary}

\theoremstyle{definition}
\newtheorem{Def}[Thm]{Definition}
\newtheorem{Rem}[Thm]{Remark}

\newtheorem{Emp}[Thm]{}

\newtheorem{Not}[Thm]{Notation}
\newtheorem{Emp'}[Thm']{}

\numberwithin{equation}{section}

\newcommand{\fqbar}{\overline{\fq}}
\newcommand{\om}{\omega}
\newcommand{\un}[1]{\underline{#1}}

\newcommand{\ov}{\overline}

\newcommand{\fq}{\B{F}_q}

\newcommand{\B}[1]{\mathbb#1}
\newcommand{\cal}[1]{\mathcal{#1}}
\newcommand{\C}[1]{\cal#1}

\newcommand{\sr}{\operatorname{sr}}

\newcommand{\tu}{\operatorname{tu}}

\newcommand{\isom}{\overset {\thicksim}{\to}}

\newcommand{\Om}{\Omega}

\newcommand{\si}{\sigma}

\newcommand{\lra}{\longrightarrow}

\newcommand{\hra}{\hookrightarrow}
\newcommand{\wt}{\widetilde}
\newcommand{\wh}{\widehat}
\newcommand{\Gm}{\Gamma}

\newcommand{\gm}{\gamma}
\newcommand{\ka}{\kappa}
\newcommand{\dt}{\delta}
\newcommand{\Dt}{\Delta}
\newcommand{\bs}{\backslash}

\newcommand{\m}{^{\times}}

\newcommand{\al}{\alpha}

\newcommand{\la}{\lambda}

\newcommand{\rl}[1]{Lemma \ref{L:#1}}
\newcommand{\rn}[1]{Notation \ref{N:#1}}

\newcommand{\rp}[1]{Proposition \ref{P:#1}}
\newcommand{\rr}[1]{Remark \ref{R:#1}}

\newcommand{\re}[1]{\ref{E:#1}}
\newcommand{\rco}[1]{Corollary \ref{C:#1}}

\newcommand{\rt}[1] {Theorem \ref{T:#1}}

\newcommand{\sm}{\smallsetminus}
\newcommand{\be}{\infty}

\newcommand{\form}[1]{(\ref{Eq:#1})}

\newcommand{\inv}{\operatorname{inv}}

\newcommand{\Ker}{\operatorname{Ker}}

\newcommand{\Spec}{\operatorname{Spec}}
\newcommand{\Aut}{\operatorname{Aut}}

\newcommand{\Ad}{\operatorname{Ad}}

\newcommand{\Gal}{\operatorname{Gal}}
\newcommand{\Tr}{\operatorname{Tr}}
\newcommand{\Fr}{\operatorname{Fr}}

\newcommand{\Lie}{\operatorname{Lie}}
\newcommand{\Int}{\operatorname{Int}}
\newcommand{\ad}{\operatorname{ad}}
\newcommand{\Id}{\operatorname{Id}}

\newcommand{\rk}{\operatorname{rk}}
\newcommand{\sgn}{\operatorname{sgn}}
\newcommand{\Hom}{\operatorname{Hom}}

\newcommand{\lan}{\langle}
\newcommand{\ran}{\rangle}

\newcommand{\vk}{{\varkappa}}

\newcommand{\mcB}{\mathcal B}

\newcommand{\mcE}{\mathcal E}

\newcommand{\mcG}{\mathcal G}
\newcommand{\mcH}{\mathcal H}

\newcommand{\mcL}{\mathcal L}
\newcommand{\mcM}{\mathcal M}

\newcommand{\mcO}{\mathcal O}

\newcommand{\mcT}{\mathcal T}
\newcommand{\mcU}{\mathcal U}

\newcommand{\fa}{\mathfrak a}
\newcommand{\fb}{\mathfrak b}

\newcommand{\fm}{\mathfrak m}

\begin{document}

\title[On endoscopic transfer of Deligne--Lusztig functions]%
{On endoscopic transfer of Deligne--Lusztig functions}

\author{David Kazhdan}
\author{Yakov Varshavsky}
\address{Institute of Mathematics\\
Hebrew University\\
Givat-Ram, Jerusalem,  91904\\
Israel}
\email{kazhdan@math.huji.ac.il, vyakov@math.huji.ac.il }

\thanks{The work was partially supported by
THE ISRAEL SCIENCE FOUNDATION (Grants No. 555/04 and 1438/06)}
\date{\today}

\begin{abstract}
In this paper we prove the fundamental lemma for Deligne--Lusztig functions. Namely, for every
Deligne--Lusztig
function $\phi$ on a $p$-adic group $G$ we write down an explicit linear combination
$\phi^H$ of  Deligne--Lusztig functions on an endoscopic group $H$ such that $\phi$ and $\phi^H$
have ``matching orbital integrals''. In particular, we prove a conjecture of Kottwitz \cite{Ko4}.
More precisely, we do it under some mild restriction on $p$.
\end{abstract}
\maketitle

\section*{Introduction}
Let $G$ be a connected reductive group over a local non-archimedean field $F$ of characteristic
zero, and let $H$ be an endoscopic group for $G$.
It follows from a combination of results of Langlands--Shelstad \cite{LS,LS2}, Ngo \cite{Ngo} and
Waldspurger \cite{Wa2,Wa3} that for every locally compact function with compact support $\phi\in C^{\be}_c(G(F))$
on $G(F)$ there exists a function $\phi^H\in C^{\be}_c(H(F))$ on $H(F)$ such that $\phi$ and $\phi^H$ have
``matching orbital integrals''. In this case we say that $\phi^H$ is an {\em endoscopic transfer} of $\phi$.

Despite of the fact that an endoscopic transfer is not unique, we expect that for all ``interesting''
functions $\phi$ one can define explicitly ``the best'' endoscopic transfer.

For example, in the case when both $G$ and $H$ are unramified, the famous fundamental lemma
(which was conjectured by Langlands and proved recently by Ngo \cite{Ngo}) asserts that every
spherical function $\phi$ on $G(F)$  has  unique spherical endoscopic transfer $\phi^H$ on $H(F)$.
Moreover, the correspondence $\phi\mapsto \phi^H$ is an explicit algebra homomorphism.

In his Seattle lecture \cite{Ko4}, Kottwitz suggested  a conjectural candidate for $\phi^H$
in the case when $G$ is split adjoint (hence $H$ is split) and $\phi$ is an unipotent Deligne--Lusztig
function, that is, $\phi$ is supported on $G(\C{O})$ and is equal to the inflation of the character of an
unipotent  Deligne--Lusztig virtual representation of $G(\fq)$. Namely, Kottwitz conjectured that in this case,
$\phi^H$ is an explicit linear combination of unipotent Deligne--Lusztig functions on $H(F)$ supported on $H(\C{O})$.

In this paper we prove a generalization of Kottwitz' conjecture under some mild restriction on the residual
characteristic of $F$. Namely, we prove it in the case when $G$ splits over an unramified extension of $F$ and
$\phi$ is an arbitrary Deligne--Lusztig function supported on an arbitrary parahoric subgroup.
Notice that our assumption on the residual  characteristic of $F$ is essential for our method, while
the assumption that $G$ splits over an unramified extension of $F$ was made only to simplify the exposition.

 Our proof is based on a theorem of Waldspurger \cite{Wa1,Wa2} asserting that if
$f^H\in C_c^{\be}(\Lie H(F))$ is an endoscopic transfer of $f\in C_c^{\be}(\Lie G(F))$, then the
Fourier transfer $\C{F}(f^H)$  is an endoscopic transfer of $e_{\psi}(H,G)\C{F}(f)$ for certain
sign $e_{\psi}(H,G)$ defined by Weil \cite{We}. Notice that the theorem of Waldspurger is
based on the fundamental lemma for unit elements.

Besides of our use of Waldspurger's theorem (whose proof is global),
our argument is purely local. Using the topological Jordan decomposition and the reduction formula
of Deligne--Lusztig \cite{DL}, we reduce to the topologically unipotent case. Then using
quasi-logarithm maps we reduce the problem to the corresponding problem about Lie algebras. Finally
the assertion follows from a combination of the Springer hypothesis about Green functions
(see \cite{KV}), the theorem of Waldspurger described above, and the explicit calculation of  the sign $e_{\psi}(H,G)$
in the unramified case.


Our paper is organized as follows. In Section 1 we recall basic notation and results from the
theory of endoscopy, while in Section 2 we introduce Deligne--Lusztig functions, and formulate our
Main Theorem.

In Sections 3--6 we introduce basic ingredients of the proof.
Namely, in Section 3 we formulate and prove a Lie algebra analog of the Main theorem.
Next, in section 4 we describe the theorem of Waldspurger which was mentioned above and calculate explicitly
the sign  $e_{\psi}(G,H)$. Then, in Section 5 we study quasi-logarithms, which were introduced in \cite{KV}.
More precisely, we show how every quasi-logarithm for a group induces a quasi-logarithm for its
endoscopic group and prove that the quasi-logarithms preserve the transfer factors. Finally, in
Section 6 we study the topological Jordan decomposition and prove the reduction formula
to the topologically unipotent case.

Section 7 is devoted to the proof of the Main theorem. First we do it in the topologically
unipotent case, using results of Sections 3--5 and the Springer hypothesis about
Green functions. Then we deduce the general case using results of Section 6.

In final Section 8 we show the properties of the transfer factors, which were used
in the previous sections. For this we observe that the transfer
factor $\Dt_{III_2}$ vanishes on tamely ramified topologically unipotent elements.
To prove this fact we show first a compatibility of the local Langlands correspondence for tori 
with field extensions.

We thank Labesse, who explained to us how to deduce \rl{abllc} from results of \cite{Lab}.

\section*{Notation and conventions}
 $F$ is always a field, $F^{ser}$ a separably closure of $F$, $\ov{F}$ an algebraic closure of
$F$, and $\Gm=\Gm_F$ the absolute Galois group of $F$.

In most of the paper (except in 1.1, 5.1 and 6.1), $F$ is a local non-archimedean
field of characteristic zero, $F^{nr}$ the maximal unramified extension of $F$, $\C{O}=\C{O}_F$
the ring of integers of $F$, $\fm\subset\C{O}$ the maximal ideal,  $\C{O}^{nr}$ the ring of integers of $F^{nr}$,
$\B{F}_q:=\C{O}/\fm$ the residue field of $F$, and $p$  the characteristic of $\fq$.

Let $G$ be a connected reductive group over $F$, $G^{sc}$ the simply connected covering of
the derived group of $G$, $Z(G)$ the center of $G$, $G^{ad}:=G/Z(G)$ the adjoint group of $G$, 
$W_G$ the Weyl group of $G$, and $\mcG$ the Lie algebra of $G$.
In most of the paper (except in 1.1--2.1, 4.1, 5.1, 6.1--6.2 and 8.1--8.2)
we assume that $G$ splits over $F^{nr}$.

\section{Preliminaries on Endoscopy}

\subsection{Endoscopy and stable conjugacy}
\begin{Emp} \label{E:not}
{\bf Notation.}
(a) We denote by $G^{\sr}(F)\subset G(F)$ (resp.
$\mcG^{\sr}(F)\subset \mcG(F)$) the set of {\em strongly regular}
elements of $G(F)$ (resp. $\C{G}(F)$), that is, elements
$\gm\in G(F)$ (resp. $x\in \mcG(F)$) such that the centralizer
$G_{\gm}\subset G$ (resp. $G_{x}\subset G$) is a torus.

(b) We set $c_G:=\Spec F[G]^G$ (resp. $c_{\mcG}:=\Spec F[\mcG]^G$), where $G$ acts on $G$ (resp. $\mcG$)
by conjugation, and denote by $\chi_G:G\to c_G$ and $\chi_{\mcG}:\mcG\to c_{\mcG}$ the canonical quotient
morphisms.

(c) A {\em splitting} of a (quasi-split) group $G$ over $F$ is a triple
$Spl_G=(B,T,\{x_{\al}\}_{\al})$, where $G\supset B\supset T$ are a Borel subgroup and
a maximal torus of $G$ defined over $F$,  $x_{\al}$ is a non-zero element of the root space
$\mcG_{\al}(F^{sep})$ for each simple root $\al$ of $(G,B,T)$ such that the set
$\{x_{\al}\}\subset\C{G}(F^{sep})$ is $\Gm$-invariant.
\end{Emp}

\begin{Emp} \label{E:inner}
{\bf Inner twistings.}
(a) By an {\em inner twisting of $G$}, we mean an isomorphism $\varphi:G\to G'$ over $F^{sep}$,
where $G'$ is a reductive group over $F$ and for each
$\si\in\Gm$, the automorphism $\varphi^{-1}\circ{}^{\si}\varphi$ of $G_{F^{sep}}$ is inner.

(b) Two inner twistings $\varphi:G\to G'$ and  $\varphi':G\to G''$ are called
{\em equivalent}, if there exists an $F$-isomorphism $f:G'\isom G''$ such that
$\varphi'^{-1}\circ f\circ\varphi$ is an inner automorphism of $G_{F^{sep}}$.

(c) Each inner twisting $\varphi$ induces  isomorphisms $c_G\isom c_{G'}$ and
$c_{\mcG}\isom c_{\mcG'}$, which we consider as equalities
$c_G= c_{G'}$ and $c_{\mcG}=c_{\mcG'}$.

(d) Note that there exists unique up to an equivalence inner twisting $\varphi:G\to G^*$
such that $G^*$ is quasi-split. In this case we say that $\varphi$ is {\em quasi-split}.

(e) We say that an inner twisting $\varphi$ is {\em trivial}, if
$G'=G$ and $\varphi$ is the identity.
\end{Emp}

\begin{Emp} \label{E:stconj}
{\bf Stable conjugacy.} Let $\varphi:G\to G'$ be an inner twisting.

(a) Two embeddings of maximal tori
$\fa:T\hra G$ and $\fa':T\hra G'$ (resp. two elements
$\gm\in G^{\sr}(F)$ and $\gm\in G'^{\sr}(F)$, resp. $x\in \mcG^{\sr}(F)$ and $x'\in \mcG'^{\sr}(F)$)
are called {\em $\varphi$-stably conjugate} or simple {\em stably conjugate},
if there exists  $g\in G'(F^{sep})$ such that $g\varphi(\fa)g^{-1}=\fa'$
(resp. $g\varphi(\gm) g^{-1}=\gm'$, resp. $\Ad g(\varphi(x))=x'$). Note that
this happens if and only if $\chi_G\circ \fa=\chi_{G'}\circ \fa'$
(resp. $\chi_G(\gm)=\chi_{G'}(\gm')$, resp. $\chi_{\mcG}(x)=\chi_{\mcG'}(x')$).

(b) If $\varphi$ is trivial, then in the situation of (a) we consider an invariant  $\inv(\fa,\fa')\in H^1(F,T)$
(resp. $\inv(\gm,\gm')\in H^1(F,G_{\gm})$, resp. $\inv(x,x')\in H^1(F,G_{x})$) defined to be a class of the
cocycle $\si\mapsto g^{-1}{}^{\si}g$. Notice that this cohomology class is independent of
a choice of $g$ (compare \cite[$\S$3]{Ko1}).

(c) If $\varphi:G\to G^*$ is quasi-split, then each $\fa:T\hra G$ (resp.
$\gm\in G^{\sr}(F)$, resp. $x\in \mcG^{\sr}(F)$) has a stable conjugate $\fa^*:T\hra G^*$
(resp. $\gm^*\in (G^*)^{\sr}(F)$, resp. $x\in  (\mcG^*)^{\sr}(F)$) (use \cite[Cor. 2.2]{Ko1}).
\end{Emp}

\begin{Emp} \label{E:dual}
{\bf Langlands dual group.}  Let $\wh{G}$ be the complex connected Langlands dual group of
$G$. Then there exists a natural action $\Gm$ on $\wh{G}$, which is unique up to a conjugacy. In particular,
there is a canonical $\Gm$-action on $Z(\wh{G})$ (compare \cite[1.5]{Ko2}).

(a)  Each inner twisting $\varphi:G\to G'$ gives rise an isomorphism $\wh{\varphi}:\wh{G}'\isom\wh{G}$,
defined up to a conjugacy, and hence to a canonical $\Gm$-equivariant isomorphism
$Z_{\wh{\varphi}}:Z(\wh{G}')\isom Z(\wh{G})$ (see \cite[1.8]{Ko2}).

(b) Each embedding of a maximal torus $\fa:T\hra G$ gives rise to a $\Gm$-invariant conjugacy class of
embeddings of maximal tori  $\wh{\fa}:\wh{T}\hra \wh{G}$. In particular, it gives rise to a canonical
 $\Gm$-equivariant embedding $Z_{\wh{\fa}}:Z(\wh{G})\hra\wh{T}$.
Moreover, embeddings $\fa:T\hra G$ and $\fa':T\hra G$ are stably conjugate if and only if the corresponding embeddings of dual tori $\wh{T}\hra \wh{G}$ are conjugate (compare \cite[1.5]{Ko2}).

(c) Conversely, if $G$ is quasi-split, then every  $\Gm$-invariant conjugacy class of
embeddings of maximal tori  $\wh{\fa}:\wh{T}\hra \wh{G}$ comes from an embedding of a maximal torus $\fa:T\hra G$,
unique up to a stable conjugacy (use \cite[Cor. 2.2]{Ko1}).
\end{Emp}

\begin{Emp} \label{E:endosc}
{\bf Endoscopic triples.}
(a) By an {\em endoscopic triple} for $G$ we mean a triple $\C{E}=(H,\ka_0,\eta_0)$,
where $H$ is a quasi-split reductive group over $F$, $\ka_0$ is an element of $Z(\wh{H})^{\Gm}$,
and $\eta_0$ is an embedding $\wh{H}\hra\wh{G}$ such that
$\eta_0(\wh{H})=(\wh{G}_{\eta_0(\ka_0)})^0$ and the conjugacy class of $\eta_0$ is $\Gm$-equivariant
(compare \cite[7.4]{Ko2}). The group $H$ is also called an {\em endoscopic group for $G$}.

(b) We say that two endoscopic triples  $\C{E}=(H,\ka_0,\eta_0)$ and $\C{E}'=(H',\ka'_0,\eta'_0)$ for $G$
are {\em equivalent}, if there exists an $F$-isomorphism $f:H'\isom H$ such that
$Z_{\wh{f}}(\ka_0)=\ka'_0$ (compare \re{dual} (a)), and $\eta'_0\circ\wh{f}:\wh{H}\hra\wh{G}$
is conjugate to $\eta_0$.

(c) For each embedding of a maximal torus $\fa_H:T\hra H$,  we denote by $\ka_{\fa_H}\in\wh{T}^{\Gm}$
the image of $\ka_0$ under $Z_{\wh{\fa_H}}:Z(\wh{H})^{\Gm}\hra \wh{T}^{\Gm}$ (see \re{dual} (b)).

(d) For each inner twisting $\varphi:G\to G'$ an endoscopic triple
$\C{E}$ for $G$ gives rise to an endoscopic triple $\mcE_{\varphi}:=(H,\ka_0,\wh{\varphi}\circ\eta_0)$ for $G'$,
whose equivalence class depends only on the equivalence class of $\varphi$.
If $\varphi$ is quasi-split, we write $\mcE^*$ instead of $\mcE_{\varphi}$.
\end{Emp}

\begin{Emp} \label{E:estconj}
{\bf $\C{E}$-stable conjugacy.}  Let $\C{E}=(H,\ka_0,\eta_0)$ be an endoscopic triple for $G$, and let
 $\varphi:G\to G^*$ be a quasi-split inner twisting.

(a) Each embedding $\fa_H:T\hra H$ of a
maximal torus gives rise to a stable conjugacy of embeddings of a maximal torus $\fa^*:T\hra G^*$.
Namely,  $\fa^*$ is the stable conjugacy class of embeddings $T\hra G^*$ such that
$\wh{\fa}^*:\wh{T}\hra\wh{G}^*$ is conjugate to the composition
$\wh{T}\overset{\wh{a_H}}{\lra}\wh{H}\overset{\eta_0}{\lra}\wh{G}\overset{\wh{\varphi}^{-1}}{\lra}\wh{G}^*$
(see \re{dual} (b), (c)).

In this case, we say that embeddings $\fa_H$ and $\fa^*$ are {\em $(\C{E},\varphi)$-stably conjugate} or
simply {\em $\C{E}$-stably conjugate}.

(b) Two embeddings of  maximal tori $\fa_H:T\hra H$ and  $\fa:T\hra G$
are called {\em $\C{E}$-stably conjugate}, if the exists a stably conjugate  $\fa^*:T\hra G^*$ of $\fa:T\hra G$,
which is $\C{E}$-stably conjugate to $\fa_H$.

In this situation, we say that $\fa^{(3)}=(\fa,\fa^*,\fa_H)$ is a triple of $\C{E}$-stably conjugate
embeddings of maximal tori.

(c)  We say that an endoscopic triple $\C{E}$ for $G$ is
{\em consistent}, if  there exist  $\C{E}$-stably conjugate embeddings of maximal tori $\fa_H:T\hra H$ and
$\fa:T\hra G$ for some torus $T$. In particular, every endoscopic triple for $G$ is consistent, if
$G$ is quasi-split (by (a)).

(d) Recall that there exists unique finite morphism
$\nu=\nu_{\mcE}:c_H\to c_G$ (resp. $\nu=\nu_{\mcE}:c_{\mcH}\to c_{\mcG}$) such that
for each maximal torus $T_H\subset H$ and each embedding
$\iota:T_H\isom T\subset G^*$, $\C{E}$-stably conjugate to the inclusion $T_H\hra H$, the
following diagram is commutative:

\begin{equation} \label{Eq:nu}
\CD
        T_H  @>\iota>>                    T\\
        @V{\chi_H}VV                        @V{\chi_G}VV\\
        c_H  @>\nu>>                    c_G.
\endCD
\end{equation}
We also have a similar commutative diagram for Lie algebras.
\end{Emp}

\begin{Emp} \label{E:gstr}
{\bf $G$-strongly regular elements.}  Let $\varphi:G\to G'$ be an inner twisting.

(a)  In the situation of \re{endosc}, we say that two elements $\gm_H\in H^{\sr}(F)$ and $\gm\in G^{\sr}(F)$
(resp. $x_H\in \mcH^{\sr}(F)$ and $x\in\mcG^{\sr}(F)$) are {\em $\C{E}$-stably conjugate},
if there exists an isomorphism $\iota: H_{\gm_H}\isom G_{\gm}\subset G$
(resp.  $\iota: H_{x_H}\isom G_{x}\subset G$) which is $\C{E}$-stably conjugate to the inclusion
$\iota_H:H_{\gm_H}\hra H$ (resp. $\iota_H:H_{x_H}\hra H$) such that $\iota(\gm_H)=\gm$
(resp. $d\iota(x_H)=x$). Note that such an $\iota$ is automatically unique.

By \re{estconj} (d), this happens if and only if  $\nu\circ\chi_H(\gm_H)=\chi_G(\gm)$ (resp.
$\nu\circ\chi_{\mcH}(x_H)=\chi_{\mcG}(x)$).

(b) For each $\C{E}$-stably conjugate elements $\gm_H\in H^{\sr}(F)$ and $\gm\in G^{\sr}(F)$
(resp. $x_H\in \mcH^{\sr}(F)$ and $x\in\mcG^{\sr}(F)$), we denote by
$\ka_{\gm,\gm_H}\in \wh{G_{\gm}}^{\Gm}$ (resp. $\ka_{x,x_H}\in \wh{G_{x}}^{\Gm}$)
the image of  $\ka_0\in Z(\wh{H})^{\Gm}$ under the canonical embedding
$Z(\wh{H})^{\Gm}\hra\wh{H_{\gm_H}}^{\Gm}\overset{\wh{\iota}^{-1}}{\to}\wh{G_{\gm}}^{\Gm}$
(resp. $Z(\wh{H})^{\Gm}\hra\wh{H_{x_H}}^{\Gm}\overset{\wh{\iota}^{-1}}{\to}\wh{G_{x}}^{\Gm}$).

(c) We say that $\gm_H\in H^{\sr}(F)$ (resp. $x_H \in \mcH^{\sr}(F)$) is
{\em $G$-strongly regular} and write $\gm_H\in H^{G-\sr}(F)$ (resp. $x_H\in \mcH^{G-\sr}(F)$),
if it has an $\C{E}$-stable conjugate $\gm\in (G^*)^{\sr}(F)$ (resp. $x\in(\mcG^*)^{\sr}(F)$).
\end{Emp}

\subsection{Transfer factors}

\begin{Emp} \label{E:eet}
{\bf  Extended endoscopic triples.} Let $W_F\subset\Gm_F$ be the Weil group of $F$, and
${}^LG:=\wh{G}\rtimes W_F$ the Langlands dual group of $G$.

By an {\em extended endoscopic triple for $G$} we mean a data
$\wt{\mcE}=(\mcE,\eta,Spl_{G^*},\fa_0^{(3)})$, consisting of

$\bullet$ a consistent endoscopic triple $\C{E}=(H,\ka_0,\eta_0)$ for $G$ (see \re{estconj} (c));


$\bullet$ an embedding $\eta:{}^L H\hra{}^L G$ over $W_F$, whose restriction to $\wh{H}$ is
$\eta_0:\wh{H}\hra\wh{G}$;

$\bullet$ a splitting $Spl_{G^*}=(B^*,T^*, \{x_{\al}\}_{\al})$ of $G^*$ over $F$;

$\bullet$ a triple $\fa_0^{(3)}$ of $\C{E}$-stable conjugate embeddings of maximal tori
$\fa_0:T_0\hra G$, $\fa^*_0:T_0\hra G^*$ and $\fa_{H,0}:T_0\hra H$ for some torus $T_0$.
\end{Emp}

\begin{Rem} \label{R:eet}
(a) Note that a splitting $Spl_{G^*}$ always exists, because $G^*$ is quasi-split, and a triple $\fa_0^{(3)}$
always exists, because $\C{E}$ was assumed to be consistent.

(b) The embedding $\eta$ does not always exist. However, it exists, if the derived group of $G$ is simply connected
(see \cite[Prop. 1]{La1}), or if both $G$ and $H$ split over $F^{nr}$ (see \cite[Lem. 6.1]{Ha}).

(c) If $G$ is quasi-split, we always assume that $\fa_0^*=\fa_0$.
For an arbitrary $G$, we denote by $\wt{\mcE}^*$ an extended endoscopic triple
$(\mcE^*,\eta,Spl_{G^*},(\fa_0^*,\fa_0^*,\fa_{H,0}))$ for $G^*$.
\end{Rem}

\begin{Emp} \label{E:setup}
Let $\wt{\mcE}$ be an extended endoscopic triple for $G$, and let $\gm_H\in H^{G-\sr}(F)$ and $\gm\in G^{\sr}(F)$
(resp. $x_H\in \mcH^{G-\sr}(F)$ and $x\in \mcG^{\sr}(F)$) be a pair of $\C{E}$-stable conjugate elements.

We set $T:=G_{\gm}$ (resp. $T:=G_{x}$) and  $T_H:=H_{\gm_H}$ (resp. $T_H:=H_{x_H}$),
denote by $\iota$ the inclusion $T\hra G$, by $\iota_H$ the canonical isomorphism $T\isom T_H\subset H$ from \re{gstr} (a), and let $\ka_{\iota_H}\in \wh{T}^{\Gm}$ be as \re{endosc} (c).

Now we recall the definition of the transfer factors $\Dt^{\wt{\mcE}}(\gm_H,\gm)$
(resp.  $\Dt^{\wt{\mcE}}(x_H,x)$) of Langlands--Shelstad (\cite{LS}), which we will sometimes denote simply by
$\Dt(\cdot,\cdot)$.

We fix an embedding $\iota^*:T\hra G^*$, which is stably conjugate to $\iota$ (use \re{stconj} (c)), and set $\gm^*:=\iota^*(\gm)$
(resp. $x^*=d\iota^*(x)$).
\end{Emp}

\begin{Emp} \label{E:adata}
{\bf $a$-data.} By an {\em $a$-data} for $T\subset G$, we mean a collection
$\{a_{\al}\in \ov{F}\m\}_{\al\in R(G,T)}$ such that $a_{-\al}=-a_{\al}$ and $a_{\si(\al)}=\si(a_{\al})$ for all
$\al\in R(G,T)$ and $\si\in\Gm$, where $R(G,T)$ denotes the set of roots of $G$ relative to $T$.

Langlands and Shelstad associated to each $a$-data $\{a_{\al}\}$ for $T$
a cohomology class $\inv(a)\in H^1(F,T)$ (depending on an embedding $\iota^*$ and a splitting $Spl_{G^*}$ of $G^*$),
which is equal to the image of $\la(T_{sc})\in H^1(F,T_{sc})$ in the notation of \cite[(3.2)]{LS}.
\end{Emp}

\begin{Emp} \label{E:chidata}
{\bf $\chi$-data.}
For each $\al\in R(G,T)$, we denote by  $F_{\al}\subset F$ and $F_{\pm \al}\subset F$ the
fields of rationality of $\al$ and the set $\{\al,-\al\}$, respectively. In particular, we have
$F_{\al}\supset F_{\pm \al}$ and $[F_{\al}:F_{\pm \al}]\leq 2$.

By a {\em $\chi$-data} for $T\subset G$, we mean a collection
$\{\chi_{\al}:F_{\al}\m\to\B{C}\m\}_{\al\in R(G,T)}$ of continuous homomorphisms such that
$\chi_{-\al}=\chi_{\al}^{-1}$,  $\chi_{\si(\al)}=\chi_{\al}\circ\si^{-1}$ for all $\al\in R(G,T)$ and $\si\in\Gm$, and such that
$\chi_{\al}|_{ F_{\pm \al}\m}$ is the quadratic character corresponding to the
quadratic extension $F_{\al}/F_{\pm \al}$, if
$F_{\al}\neq F_{\pm \al}$.

Notice that $\chi$-data always exist. Moreover, if each quadratic extension
$F_{\al}/F_{\pm \al}$ is tamely ramified (resp. unramified), we can assume that
each $\chi_{\al}$ is tamely ramified (resp. unramified), that is, trivial on
$1+\frak{m}_{F_{\al}}$
(resp. $\C{O}_{F_{\al}}\m$). In this case, we say that the $\chi$-data are
tamely ramified
(resp. unramified). In particular, tamely ramified $\chi$-data always exist, if $p\neq 2$.

Langlands and Shelstad associated to each $\chi$-data $\{\chi_{\al}\}$ for $T$
a continuous cohomology class $\inv(\chi)\in H_c^1(W_F,\wh{T})$, which is denoted by $\bf{a}$ in \cite[(3.5)]{LS}.
 \end{Emp}

\begin{Emp} \label{E:inv}
{\bf The invariant.}
We fix an inner twisting $\varphi:G\to G'$ and two triples
$\fa_i^{(3)}=(\fa_i,\fa'_i,\fa_{H,i});\;i=1,2$ of
$\C{E}$-stably conjugate embeddings of maximal tori $\fa_i:T_i\hra G$, $\fa'_i:T_i\hra G'$ and $\fa_{H,i}:T_i\hra H$.

We denote by  $\inv(\fa_1^{(3)},\fa_2^{(3)})=\inv_{\C{E}}(\fa_1^{(3)},\fa_2^{(3)})\in\B{C}\m$ the invariant
$\lan\frac{\fa_2,\fa'_2;\fa_{H,2}}{\fa_1,\fa'_1;\fa_{H,1}}\ran$ defined in \cite[1.5.5]{KV}
(its definition was motivated by \cite[(4.3)]{LS}). This invariant
has the following properties (see \cite[Lem. 1.5.7 and Lem. 2.4.5]{KV}).

(i) $\inv(\fa_1^{(3)},\fa_2^{(3)})$ depends only on the conjugacy classes of the $\fa_i$'s and the $\fa'_i$'s and
the stable conjugacy classes of the $\fa_{H,i}$'s.

(ii) Assume that $T_1=T_2(=T)$ and $\fa_{H,1}=\fa_{H,2}(=\fa_H)$. Then $\fa_1,\fa_2:T\hra G$ and
$\fa'_1,\fa'_2:T\hra G'$ are stable conjugate, and we have an equality
\[
\inv(\fa_1^{(3)},\fa_2^{(3)})=\lan \inv(\fa'_1,\fa'_2),\ka_{\fa_H}\ran\lan \inv(\fa_1,\fa_2),\ka_{\fa_H}\ran^{-1},
\]
where  the pairing is induced by the Tate--Nakayama duality for tori
$H^1(F,T)^D\isom\pi_0(\wh{T}^{\Gm})$.

(iii) If $\varphi$ is trivial, then $\inv(\fa_1^{(3)},\fa_2^{(3)})=
\lan \inv(\fa_2,\fa'_2),\ka_{\fa_{H,2}}\ran\lan \inv(\fa_1,\fa'_1),\ka_{\fa_{H,1}}\ran^{-1}$.

(iv) (transitivity) Let $\fa_3^{(3)}=(\fa_3,\fa'_3,\fa_{H,3})$ be a third triple of $\C{E}$-stably conjugate
embeddings of maximal tori. Then
$\inv(\fa_1^{(3)},\fa_3^{(3)})=\inv(\fa_1^{(3)},\fa_2^{(3)})\inv(\fa_2^{(3)},\fa_3^{(3)})$.

(v) Let $\pi:G_0\to G$ be a quasi-isogeny, that is, $\pi$ induces an isomorphism of adjoint groups
$\pi^{ad}:G_0^{ad}\isom G^{ad}$. Then $\varphi$ lifts to an inner twisting $\varphi_0:G_0\to G'_0$,
$\C{E}$ lifts to an endoscopic triple $\C{E}_0$ for $G_0$, each $\fa_i^{(3)}$ lifts to a triple
 $\fa_{i,0}^{(3)}$, and we have an equality $\inv_{\C{E}_0}(\fa_{1,0}^{(3)},\fa_{2,0}^{(3)})=
\inv_{\C{E}}(\fa_1^{(3)},\fa_2^{(3)})$.

(vi) Let $(s,s',s_H)$ be a triple of $\C{E}$-stable
conjugate semi-simple elements $s\in G(F)$, $s'\in G'(F)$ and  $s_H\in H(F)$
such that $H_{s_H}^0:=(H_{s_H})^0$ is quasi-split, and
let $\fa_i^{(3)}=(\fa_i,\fa'_i,\fa_{H,i});\;i=1,2$ be two triples of $\C{E}_s$-stably conjugate embeddings of
maximal tori
$\fa_i:T_i\hra G^0_s$, $\fa'_i:T_i\hra G'^0_{s'}$, $\fa_{H,i}:T_i\hra H^0_{s_H}$.

Then triples $\fa_i^{(3)}$ give rise to triples of $\C{E}$-stably conjugate embeddings of maximal tori
$\fa_i:T_i\hra G^0_s\hra G$, $\fa'_i:T_i\hra G'^0_{s'}\hra G'$, $\fa_{H,i}:T_i\hra H^0_{s_H}\hra H$,
and we have an equality $\inv_{\C{E}_s}(\fa_1^{(3)},\fa_2^{(3)})=\inv_{\C{E}}(\fa_1^{(3)},\fa_2^{(3)})$.
\end{Emp}

\begin{Emp} \label{E:deftr}
{\bf Definition of the transfer factors.} We fix $a$-data and $\chi$-data for $T$.
Langlands and Shelstad defined the transfer factor $\Dt(\gm_H,\gm)\in\B{C}\m$
(resp. $\Dt(\gm_H,\gm)\in\B{C}\m$) to be the product
$\Dt=\Dt_I\Dt_{II}\Dt_{III_1}\Dt_{III_2}\Dt_{IV}$, where

$\bullet$ $\Dt_I(\gm_H,\gm)$ (resp.  $\Dt_I(x_H,x)$) is the Tate--Nakayama
pairing $\lan\inv(a),\ka\ran$, where $\ka=\ka_{\gm,\gm_H}$ (resp. $\ka=\ka_{x,x_H}$);


$\bullet$ $\Dt_{II}(\gm_H,\gm)=\prod_{\al}\chi_{\al}(\frac{\al(\gm)-1}{a_{\al}})$
(resp.   $\Dt_{II}(x_H,x)=\prod_{\al}\chi_{\al}(\frac{\al(x)}{a_{\al}})$), where the
product is taken over a set of representatives of $\Gm$-orbits on $R(G,T)\sm R(H,T_H)$.

$\bullet$  $\Dt_{III_1}(\gm_H,\gm)$ (resp.  $\Dt_{III_1}(x_H,x)$) equals
$\inv(\fa_0^{(3)},\iota^{(3)})$, there $\iota^{(3)}$ is the triple of $\C{E}$-stable conjugate embeddings
$(\iota,\iota^*,\iota_H)$, chosen in \re{setup}.

$\bullet$ $\Dt_{III_2}(x_H,x)=1$, while
$\Dt_{III_2}(\gm_H,\gm)$ equals the Langlands pairing $\lan\inv(\chi_{\al}),\gm\ran$, whose definition we will
recall in \ref{SS:abllc}.

$\bullet$  $\Dt_{IV}(\gm_H,\gm)= \prod_{\al}|\al(\gm)-1|^{1/2}$ (resp.
$\Dt_{IV}(x_H,x)= \prod_{\al}|\al(x)|^{1/2}$), where the product is taken over all
$\al\in R(G,T)\sm R(H,T_H)$.
\end{Emp}

\begin{Rem}  \label{R:tf}
(a) $\Dt$ is independent of a choice of $\iota^*$, $a$-data, and $\chi$-data.

(b) The $\chi$-data, and hence a choice of $\eta$ is not needed for the transfer factors for
Lie algebras. In particular, the transfer factors for Lie algebras are always defined.

(c) Another choice of a splitting $Spl_{G^*}$ results in a multiplication of $\Dt$
by a non-zero constant. In the case of groups, another choice of an embedding $\eta$ results
in a multiplication of  $\Dt$ by a character $H(F)\to \B{C}\m$.
\end{Rem}

\begin{Emp} \label{E:proptf}
{\bf Properties of the transfer factors.}
(a) $\Dt(\gm_H,\gm)$ (resp. $\Dt(x_H,x)$) depends only on the stable conjugacy class of
$\gm_H$  (resp. $x_H$) and on the $G(F)$-conjugacy class of $\gm$ (resp. $x$) (see \re{inv} (i)).

(b) For each stable conjugate $\gm'$ of  $\gm$ (resp. $x'$ of $x$) we have an equality \\
$\Dt(\gm_H,\gm')=\Dt(\gm_H,\gm)\lan\inv(\gm,\gm'),\ka_{\gm,\gm_H}\ran$
(resp. $\Dt(x_H,x')=\Dt(x_H,x)\lan\inv(x,x'),\ka_{x,x_H}\ran$ (see  \re{gstr} (b) and \re{inv} (ii),(iv)).

(c) If $\fa_0^{(3)}$ is replaced by another triple  $\fa_1^{(3)}$ (and $\wt{\mcE}$ by the corresponding extended
endoscopic triple), then $\Dt^{\wt{\mcE}}$ is multiplied by $\inv(\fa_1^{(3)},\fa_0^{(3)})$.

(d)  If $G$ is quasi-split and $\fa_0^*=\fa_0$ (compare \rr{eet} (c)), then the transfer factor $\Dt^{\wt{\mcE}}$
is independent of $\fa_0^{(3)}$ (see (c) and \re{inv} (iii)). Moreover, in this case we always assume that
$\iota^*=\iota$, hence $\Dt_{III_1}=1$ (by \re{inv} (iii)).

(e) For an arbitrary $G$, we have  $\Dt^{\wt{\mcE}}(\gm_H,\gm)=\inv(\fa_0^{(3)},\iota^{(3)})
\Dt^{\wt{\mcE}^*}(\gm_H,\gm^*)$, where $\wt{\mcE}^*$ was defined in \rr{eet} (c) (use(d)).
\end{Emp}

\subsection{Endoscopic transfer}

\begin{Emp}
{\bf Orbital integrals.}

(a) Denote by $C_c^{\be}(G(F))$ (resp. $C_c^{\be}(\mcG(F))$) the
space of complex locally constant functions with compact support.
For each $\gm\in G^{\sr}(F)$ (resp. $x\in\mcG^{\sr}(F)$), we
denote by $dg_{\gm}$ (resp. $dg_x$) the Haar measure on
$G_{\gm}(F)$ (resp. $G_x(F)$) such that the measure of the maximal
compact subgroup is $1$. We also fix a Haar measure $dg$ on
$G(F)$.

(b) For each $\phi\in C_c^{\be}(G(F))$ (resp. $\phi\in C_c^{\be}(\mcG(F))$) and
$\gm\in G^{\sr}(F)$ (resp. $x\in \mcG^{\sr}(F)$), we denote by
$O_{\gm}(\phi)$ (resp. $O_{x}(\phi)$) the orbital integral
\[
\int_{G(F)/G_{\gm}(F)}\phi(g\gm g^{-1})\frac{dg}{dg_{\gm}}
\left(\text{resp. } \int_{G(F)/G_{x}(F)}\phi(\Ad g(x))
\frac{dg}{dg_x}\right).
\]

(c) For each $\phi$ and $\gm$ (resp. $x$) as in (b) and every
$\kappa\in\wh{G_{\gm}}^{\Gm}$ (resp.
$\kappa\in\wh{G_{x}}^{\Gm}$), we denote by
$O^{\ka}_{\gm}(\phi)$ (resp. $O^{\ka}_{x}(\phi)$) the
$\ka$-orbital integral
\[
\sum_{\gm'}\lan\inv(\gm,\gm'),\ka\ran O_{\gm'}(\phi)\left(\text{
resp. } \sum_{x'}\lan\inv(x,x'),\ka\ran O_{x'}(\phi)\right),
\]
where the sum runs over a set of representatives of the
$G(F)$-conjugacy classes in $G(F)$ (resp. $\mcG(F)$), which are
stably conjugate to $\gm$ (resp. $x$).

 In the case $\ka=1$, we usually write $O^{st}$ instead of
 $O^{\ka}$ and call it {\em stable orbital integral}.

(d) We say that $\phi\in C_c^{\be}(G(F))$ (resp. $\phi\in C_c^{\be}(\mcG(F))$) is {\em $\C{E}$-unstable}, if
for every $\C{E}$-stably conjugate pair $\gm_H\in H^{\sr}(F)$, $\gm\in G^{\sr}(F)$ (resp.
$x_H\in \mcH^{\sr}(F)$, $x\in \mcG^{\sr}(F)$), we have
$O^{\ka_{\gm,\gm_H}}_{\gm}(\phi)=0$ (resp. $O^{\ka_{x,x_H}}_{x}(\phi)=0$).
\end{Emp}

%

\begin{Emp} \label{E:transf}
{\bf Endoscopic transfer.}
Fix an extended endoscopic triple $\wt{\mcE}$ for $G$ and a Haar measure $dh$ of $H(F)$.

(a)  For each $\phi\in C^{\be}_c(G(F))$ and $\gm_H\in H^{G-\sr}(F)$
(resp. $\phi\in C^{\be}_c(\mcG(F))$ and $x_H\in \mcH^{G-\sr}(F)$), we denote by
$O^{\wt{\mcE}}_{\gm_H}(\phi)$
(resp. $O^{\wt{\mcE}}_{x_H}(\phi)$) the sum
$\sum_{\gm}\Dt^{\wt{\mcE}}(\gm_H,\gm) O_{\gm}(\phi)$ (resp. $\sum_{x}\Dt^{\wt{\mcE}}(x_H,x) O_{x}(\phi)$),
 where $\gm$ (resp. $x$) runs over a set of representatives of the
$G(F)$-conjugacy classes  in $G(F)$ (resp. $\mcG(F)$) which are
$\C{E}$-stably conjugate to $\gm_H$ (resp. $x_H$).

(b) By \re{proptf} (b), we have an equality $O^{\wt{\mcE}}_{\gm_H}(\phi)=
\Dt^{\wt{\mcE}}(\gm_H,\gm)O^{\ka_{\gm,\gm_H}}_{\gm}(\phi)$
(resp. $O^{\wt{\mcE}}_{x_H}(\phi)=\Dt^{\wt{\mcE}}(x_H,x)O^{\ka_{x,x_H}}_{x}(\phi)$)
for each $\C{E}$-stably conjugate $\gm\in G^{\sr}(F)$ of $\gm_H$
(resp.  $x\in \mcG^{\sr}(F)$ of $x_H$).

(c) We say that $\phi^H\in C^{\be}_c(H(F))$ (resp. $\phi^H\in C^{\be}_c(\mcH(F))$) is an
{\em $\wt{\mcE}$-transfer} or an {\em endoscopic transfer} of $\phi\in C^{\be}_c(G(F))$
(resp. $\phi\in C^{\be}_c(\mcG(F))$), if for each element $\gm_H\in H^{G-\sr}(F)$
(resp.  $x_H\in \mcH^{G-\sr}(F)$) we have an
equality $O^{st}_{\gm_H}(\phi^H)=O^{\wt{\mcE}}_{\gm_H}(\phi)$
(resp. $O^{st}_{x_H}(\phi^H)=O^{\wt{\mcE}}_{x_H}(\phi)$).

\end{Emp}


\section{Main Theorem}
\subsection{Deligne--Lusztig functions}

\begin{Emp} \label{E:BT}
{\bf Bruhat--Tits building.}
Let  ${\vk}$ be a point of the (non-reduced) Bruhat--Tits building $\C{B}(G)$ of $G$.

(a)  We denote by $G_{\vk}\subset G(F)$ \label{gx} (resp. $\C{G}_{\vk}\subset\C{G}$)
the corresponding  parahoric subgroup (resp. subalgebra), and by
$G_{\vk^+}\subset G_{\vk}$ (resp. $\C{G}_{\vk^+}\subset \C{G}_{\vk}$)  the
pro-unipotent (resp. pro-nilpotent) radical of $G_x$ (resp. of $\C{G}_x$) (compare \cite{MP1}).

(b) We denote by $\un{G}_{\vk}$  the canonical
smooth connected group scheme over $\C{O}$ defined by Bruhat--Tits (see \cite{BT} or \cite{Land}),
whose generic fiber is $G$, and $\un{G}_{\vk}(\C{O})=G_{\vk}$,
and let $\ov{G}_{\vk}$ be the special fiber of $\un{G}_{\vk}$.

(c) We denote by $L_{\vk}$  \label{lx} the maximal reductive quotient $(\ov{G}_{\vk})_{red}=\ov{G}_{\vk}/R_u(\ov{G}_{\vk})$.
We have canonical identifications $L_{\vk}(\fq)=G_{\vk}/G_{{\vk}^+}$ and
$\C{L}_{\vk}:=\Lie L_{\vk}=\C{G}_{\vk}/\C{G}_{{\vk}^+}$ (see \cite{MP1}).
For every $g\in G_{\vk}$ and $x\in\C{G}_{\vk}$, we put
$\ov{g}:=gG_{{\vk}^+}\in L_{\vk}(\fq)$ \label{ovg} and $\ov{x}:=x+\C{G}_{{\vk}^+}\in\C{L}_{\vk}(\fq)$.

(d) If $G=T$ is an unramified torus, then the group scheme $\un{T}_{\vk}$ is independent of
${\vk}\in\C{B}(T)$ and coincides with the canonical $\C{O}$-structure of $T$.
We denote by $\ov{T}:=\ov{T}_{\vk}$ the corresponding torus over $\fq$.
The functor $T\mapsto\ov{T}$ defines an equivalence of categories between unramified tori over $F$ and
tori over $\fq$.

\end{Emp}
\begin{Emp} \label{E:DL}
{\bf Deligne--Lusztig functions.}
Consider a triple $(\vk,\ov{\fa},\theta)$, where $\vk\in\C{B}(G)$,
$\ov{\fa}:\ov{T}\hra L_{\vk}$ is an embedding of a maximal torus, and
$\theta:\ov{T}(\fq)\to\B{C}\m$ is a character.

We denote by $\ov{\phi}_{\ov{\fa},\theta}$ the character of the virtual
representation  Deligne--Lusztig representation $R^{\theta}_{\ov{\fa}(\ov{T})}$ of $L_{\vk}(\fq)$
(see \cite{DL}), and let $\phi^G_{\ov{\fa},\theta}\in C^{\be}_c(G(F))$ be the
function supported on $G_{\vk}$ such that
$\phi_{\ov{\fa},\theta}^G(g):=\ov{\phi}_{\ov{\fa},\theta}(\ov{g})$ for all $g\in G_{\vk}$.
\end{Emp}

\subsection{Formulation of the result}
\begin{Emp}
The goal of this paper is to find explicitly an endoscopic transfer of Deligne--Lusztig functions
$\phi^G_{\ov{\fa},\theta}$. For simplicity, we restrict ourself to the case when $G$ is split
over $F^{nr}$.

To formulate our Main Theorem, we need the following lemma, whose
proof will be given in \re{pfmaxtor}.
\end{Emp}

\begin{Lem} \label{L:maxtor}
Assume that $G$ splits over $F^{nr}$, let $T$ be an unramified torus over $F$, and
$\ov{T}$ be the corresponding torus over $\fq$.

(a) Every embedding of a maximal torus $\fa:T\hra G$ defines an embedding
$\C{B}(T)\hra\C{B}(G)$, which we also denote by $\fa$, and an embedding of a maximal torus
$\ov{\fa}:\ov{T}\hra L_{\vk}$ for every $\vk\in \fa(\C{B}(T))\subset\C{B}(G)$.

(b) Conversely, for every $\vk\in\C{B}(G)$ and every embedding of a maximal torus
$\ov{\fa}:\ov{T}\hra L_{\vk}$, there exists an embedding of a maximal torus $\fa:T\hra G$ such that
 $\vk\in \fa(\C{B}(T))$ and $\ov{\fa}$ comes from $\fa$ as in (a). Moreover, $\fa$ is unique up to a
$G_{\vk^+}$-conjugacy.

(c) Let $\vk^*\in \C{B}(G^*)$ be a hyperspecial vertex. Then for every embedding of a maximal torus
$\fa:T\hra G$ there exists a stable conjugate $\fa^*:T\hra G^*$ of $\fa$ such that $\vk^*\in \fa^*(\C{B}(T))$.
Moreover, $\fa^*$ is unique up to a $G^*_{\vk^*}$-conjugacy.
\end{Lem}

\begin{Not} \label{N:stconj}
(a) Let $\vk\in\C{B}(G)$ and $\vk^*\in\C{B}(G^*)$. Embeddings of maximal tori
$\ov{\fa}:\ov{T}\hra L_{\vk}$ and $\ov{\fa}^*:\ov{T}\hra L_{\vk^*}$ are called {\em stably conjugate}, if
the corresponding embeddings  $\fa:T\hra G$ and  $\fa^*:T\hra G^*$ from \rl{maxtor} (b) are
stably conjugate.

(b) We fix a hyperspecial vertex $\vk^*\in\mcB(G^*)$, set $\ov{G}:=L_{\vk^*}$, and choose an embedding of a
maximal torus $\ov{\fa}^*:\ov{T}\hra \ov{G}$ which is stably conjugate to $\ov{\fa}$.
Then $\ov{\fa}^*$ is unique up a conjugacy (by \rl{maxtor} (c)).

\end{Not}

\begin{Emp} \label{E:unr}
{\bf The unramified case.} Let $\C{E}=(H,\ka_0,\eta_0)$ be an endoscopic triple for $G$ such that
$H$ is unramified.

(a) We fix a hyperspecial vertex $\vk^H\in \C{B}(H)$, set $\ov{H}:=L_{\vk^H}$,  and normalize
Haar measures on $G(F)$ and $H(F)$ such that $dg(G_{\vk})=dh(H_{\vk^H})=1$.

(b) We fix an extended endoscopic triple $\C{E}^*=(\C{E},\eta, Spl_{G^*}, \fa_0^{(3)})$ for $G$, where

$\bullet$ the splitting  $Spl_{G^*}=(B^*,T^*,\{x_{\al}\})$  defines a splitting of $\un{G}_{\vk^*}^*$ over $\C{O}$;

$\bullet$ the embedding $\eta:{}^L H\hra{}^L G$ is {\em unramified}, that is,
$\eta$ descends to an embedding $\wh{H}\rtimes Gal(E/F)\hra\wh{G}\rtimes Gal(E/F)$
for some finite unramified extension $E/F$ (such an embedding always exists by
\cite[Lem. 6.1]{Ha}).

(c) We choose embeddings of maximal tori $\fa:T\hra G$ and  $\fa^*:T\hra G^*$ corresponding
to $\ov{\fa}$ and  $\ov{\fa}^*:\ov{T}\hra \ov{G}$ (defined in \rn{stconj} (b)) as in  \rl{maxtor} (b).
Note that $\fa$ and $\fa^*$ are defined uniquely up to a conjugacy.

(d) We say that  embeddings of maximal tori $\ov{\fa}_H:\ov{T}\hra \ov{H}$ and
$\ov{\fa}:\ov{T}\hra L_{\vk}$ are {\em $\C{E}$-stably conjugate}, if the
corresponding embeddings  $\fa_H:T\hra H$ and  $\fa:T\hra G$ from \rl{maxtor} (b) are
$\C{E}$-stably conjugate.

(e) We set
\[
\phi^H_{\ov{\fa},\theta}:=
\sum_{\ov{\fa}_H}\inv_{\mcE}(\fa_0^{(3)},\fa^{(3)})\phi_{\ov{\fa}_H,\theta}^H,
\]
where

$\bullet$ the sums runs over a set of representatives of the set of conjugacy classes of
embeddings $\ov{\fa}_H:\ov{T}\hra\ov{H}$ of maximal tori, which are
$\C{E}$-stable conjugate to $\ov{\fa}$;

$\bullet$ $\fa_H:T\hra H$ is an embedding of a maximal torus, which corresponds to $\ov{\fa}_H$ by
\rl{maxtor} (b), $\fa^{(3)}$ is a triple of $\C{E}$-stably conjugate embeddings $(\fa,\fa^*,\fa_H)$, and the
invariant $\inv_{\mcE}(\fa_0^{(3)},\fa^{(3)})\in\B{C}\m$ was defined in \re{inv}.

$\bullet$  $\phi^H_{\ov{\fa}_H,\theta}$ is the
Deligne--Lusztig function on $H(F)$ supported on $H_{\vk^H}$ (see \re{DL}).

In particular, we define $\phi^H_{\ov{\fa},\theta}$ to be zero,
if there are no embeddings  $\ov{\fa}_H:\ov{T}\hra\ov{H}$, which are
$\C{E}$-stably conjugate to $\ov{\fa}$.
\end{Emp}

Now we are ready to formulate our main result.

\begin{Thm} \label{T:main}
Assume that $G$ splits over $F^{nr}$, $p$ does not divide the order of the Weyl group $W_G$ of $G$,
and there exists $\ov{t}\in\ov{\mcT}(\fq)$ such that $d\ov{\fa}^*(\ov{t})\in\ov{\mcG}^{\sr}(\fq)$.

(a) If $H$ is ramified, then the function $\phi^G_{\ov{\fa},\theta}$ is $\C{E}$-unstable.

(b) If $H$ is unramified, then in the notation of \re{unr}, function $\phi^H_{\ov{\fa},\theta}$ is an
endoscopic transfer of $\phi^G_{\ov{\fa},\theta}$.
\end{Thm}

\begin{Rem}
(a) Our assumptions holds, if $p$ is greater than the Coxeter number of $G$.

(b) The assertion should be true without any restriction on $p$, but our method does
not work in the general case.

(c) In the case when $G$ is split adjoint, $\vk$ is a hyperspecial vertex and $\theta=1$, \rt{main}
was conjectured by Kottwitz (\cite{Ko4}).

\end{Rem}

\section{Lie algebra analog of the Main Theorem}

\subsection{Preparations.}

\begin{Emp} \label{E:cg}
Notice that since $G$ splits over $F^{nr}$, its Cartan torus $T_G$ is unramified,
hence $T_G$ has a natural $\C{O}$-structure  $T_{G,\C{O}}$. We equip $c_{G}$ and
$c_{\mcG}$ with the induced $\C{O}$-structures
$c_{G,\C{O}}=W_G\bs T_{G,\mcO}$ and $c_{\C{G},\C{O}}=W_G\bs\C{T}_{G,\mcO}$,
denote by  $\ov{c}_{G}$ and  $\ov{c}_{\C{G}}$ their special fibers.
Note that $\ov{c}_{G}$ (resp. $\ov{c}_{\C{G}}$) is canonically isomorphic to $c_{L_{\vk}}$
(resp. $c_{\C{L}_{\vk}}$), if $\vk$ is hyperspecial.
\end{Emp}

\begin{Lem} \label{L:surj}
(a) For every $\vk\in \C{B}(G)$, the Chevalley morphisms  $\chi_G$ and $\chi_{\C{G}}$ extend uniquely to
morphisms $\chi_{G,\vk}:\un{G}_{\vk}\to c_{G,\C{O}}$ and  $\chi_{\mcG,\vk}:\un{\mcG}_{\vk}\to c_{\mcG,\C{O}}$ over
$\C{O}$ and induce morphisms $\ov{\chi}_{G,\vk}:\mcL_{\vk}\to \ov{c}_G$ and
$\ov{\chi}_{\mcG,\vk}:L_{\vk}\to \ov{c}_{\mcG}$
over $\fq$. Moreover,  $\ov{\chi}_{G,\vk}$ (resp. $\ov{\chi}_{\mcG,\vk}$) coincides with
$\chi_{L_{\vk}}$ (resp. $\chi_{\mcL_{\vk}}$), if $\vk$ is hyperspecial.

(b) If $\vk$ is hyperspecial, then the maps $\chi_{G,\vk}(\mcO):G_{\vk}\to c_G(\C{O})$
and $\chi_{\mcG,\vk}(\mcO):\mcG_{\vk}\to c_{\mcG}(\C{O})$ are surjective.
\end{Lem}

\begin{proof}
We will prove the assertions for $\chi_G$, while the proof for $\chi_{\mcG}$ is similar.

(a) The proof is based on the following assertion.

\begin{Lem} \label{L:defo}
Let $X_{\C{O}}$ and $Y_{\C{O}}$ be two affine schemes over $\C{O}$ of finite type, where $X_{\C{O}}$ is
smooth over $\C{O}$. Then a morphism of generic fibers $f:X\to Y$ extends (uniquely) to a morphism
$f_{\C{O}}:X_{\C{O}}\to Y_{\C{O}}$ over $\C{O}$ if and only if we have an inclusion
$f(X(\C{O}^{nr}))\subset Y(\C{O}^{nr})$.

Furthermore, this happens if and only if we have an inclusion
$f(S)\subset Y(\C{O}^{nr})$ for some subset $S\subset X(\C{O}^{nr})$ whose reduction $\ov{S}\subset X(\fqbar)$ is Zariski dense.
\end{Lem}

\begin{proof}
The argument of \cite[Prop. 0.3]{Land} works without changes.
\end{proof}

Now we return to the proof of \rl{surj}. Extending scalars to $F^{nr}$ and using \rl{defo}, it will suffice to show that
we have an inclusion $\chi_G(G_{\vk})\subset c_{G}(\mcO^{nr})$, and that we have an equality
$\ov{\chi_G(g)}=\ov{\chi_G(gh)}\in\ov{c}_{G}(\fq)$ for every
$g\in G_{\vk}$ and every $h\in G_{\vk^+}$.

Assume first that $\vk$ is hyperspecial, and choose a maximal $F^{nr}$-split subtorus $T\subset G$ such that
$T(\mcO^{nr})\subset G_{\vk}$. Consider the subset $S:=\cup_{g\in G_{\vk}}gT(\mcO^{nr})g^{-1}\subset  G_{\vk}$.
Then $\ov{S}=\cup_{\ov{g}\in L_{\vk}(\fqbar)}\ov{g}\ov{T}(\fqbar)\ov{g}^{-1}$ is Zariski dense in
$\ov{G}_{\vk}$. By the definition of the $\mcO^{nr}$-structure of $c_G$, we have
$\chi_G(T(\mcO^{nr}))\subset c_{G}(\mcO^{nr})$. Hence $\chi_G(S)\subset c_{G}(\mcO^{nr})$,
because $\chi_G$ is $G^{ad}$-invariant. Then the inclusion
$\chi_G(G_{\vk})\subset c_{G}(\mcO^{nr})$ follows from \rl{defo}, while
the equality $\ov{\chi_G(g)}=\ov{\chi_G(gh)}\in\ov{c}_{G}(\fq)$ is clear.

Next we assume that $\vk$ is Iwahori, that is, $G_{\vk}\subset G(F)$ is an Iwahori subgroup.
Choose a hyperspecial
$\vk'\in\C{B}(G)$ such that  $G_{\vk}\subset G_{\vk'}$. Then we have an inclusion
$\chi_G(G_{\vk})\subset\chi_G(G_{\vk'})\subset c_{G}(\mcO^{nr})$, and for each  $g\in G_{\vk}$ and $h\in G_{\vk^+}$
we have $\ov{\chi_G(g)}=\chi_{L_{\vk'}}(\ov{g})$ and $\ov{\chi_G(gh)}=\chi_{L_{\vk'}}(\ov{gh})$. Hence
the equality $\ov{\chi_G(gh)}=\ov{\chi_G(g)}$ follows from the fact that
$G_{\vk}/G_{\vk'^+}\subset L_{\vk'}(\fqbar)$ is a Borel subgroup of  $L_{\vk'}(\fqbar)$, while
$G_{\vk^+}/G_{\vk'^+}\subset L_{\vk'}(\fqbar)$ is its unipotent radical.

For an arbitrary $\vk$, we choose $\vk'\in\C{B}(G)$ such that
$G_{\vk'}\subset G_{\vk}$ is an Iwahori subgroup. Then $ G_{\vk}=\cup_{g\in G_{\vk}} gG_{\vk'}g^{-1}$ and
$\chi_G$ is $G^{ad}$-invariant. Hence the inclusion  $\chi_G(G_{\vk})\subset c_{G}(\mcO^{nr})$ for
$\vk$ follows from that for $\vk'$, and it is enough to show the equality
$\ov{\chi_G(g)}=\ov{\chi_G(gh)}$ for $g\in G_{\vk'}$ and $h\in G_{\vk^+}$. Since  $G_{\vk^+}\subset G_{\vk'^+}$,
this follows from the corresponding equality for $\vk'$.

(b) We will show a stronger assertion which says that the restriction
$\chi^{reg}_{G,\vk}(\mcO):\un{G}^{reg}_{\vk}(\mcO)\to c_G(\C{O})$, where
$\un{G}^{reg}_{\vk}\subset\un{G}_{\vk}$
denotes the set of regular elements of $\un{G}_{\vk}$, is surjective.
Recall that the morphism $\chi^{reg}_{G,\vk}:\un{G}^{reg}_{\vk}\to c_G$ is smooth, hence
by Hensel's lemma, it will suffice to show the surjectivity of corresponding morphism
$\chi^{reg}_{{L_{\vk}}}(\fq):L_{\vk}^{reg}(\fq)\to c_{L_{\vk}}(\fq)$. However for each
$c\in  c_{{L_{\vk}}}(\fq)$, the preimage $(\chi^{reg}_{{L_{\vk}}})^{-1}(c)\subset L_{\vk}^{reg}$
is a single ${L_{\vk}}$-conjugacy class, defined over $\fq$. Thus $(\chi^{reg}_{{L_{\vk}}})^{-1}(c)$ is a homogeneous space over a connected group ${L_{\vk}}$ defined
over a finite field $\fq$. Hence it has an $\fq$-rational point by a version of
Lang's theorem.
\end{proof}

\begin{Emp} \label{E:pfmaxtor}
\begin{proof}[Proof of \rl{maxtor}]
(a) The existence of the embedding $\C{B}(T)\hra\C{B}(G)$ is standard (see, for example, \cite[2.2]{DB}).
Next for each $\vk\in \fa(\C{B}(T))\subset\C{B}(G)$, the embedding $\fa$ induces  an embedding
$\fa_{\vk}:T_{\C{O}}\hra \un{G}_{\vk}$ over $\C{O}$, whose special fiber gives the embedding
$\ov{\fa}:\ov{T}\hra \ov{G}_{\vk}\to L_{\vk}$.

(b) The uniqueness assertion is proven, for example, in \cite[Lem. 2.2.2]{DB}. Though the existence assertion
is proven in the course of the proof of \cite[Prop. 5.1.10]{BT} (compare \cite[Lem. 2.3.1]{DB}),
their argument is rather involved, so we give a much more elementary argument instead.

Assume first that $T$ (and hence also $G$) splits over $F$. Then there exists an embedding
$\fb:T\hra G$ such that $\vk\in \fb(\C{B}(T))\subset\C{B}(G)$, and we denote by $\ov{\fb}:\ov{T}\hra L_{\vk}$
the corresponding embedding over $\fq$. Now the assertion follows from the fact that $\ov{\fa}$ and
$\ov{\fb}$ are $L_{\vk}(\fq)$-conjugate and the projection $G_{\vk}\to L_{\vk}(\fq)$ is surjective.

Let now $T$ be general. By the proven above, there exists an embedding $\fa$ over some finite
unramified extension $E$ of $F$. Let $\si\in Gal(E/F)$ be the Frobenius element. Then
by the uniqueness, there exists $g\in G_{\vk^+}(\C{O}_E)$ such that
${}^{\si}\fa=g^{-1}\fa g$. By a profinite version of Lang's theorem, there exists
$h$ in the profinite completion $G_{\vk^+}(\wh{\C{O}^{nr}})$ such that $g=h^{-1}{}^{\si}h$.
Then the conjugate embedding $\fa':=h\fa h^{-1}:T\hra G$ is defined over $F$ and satisfies
$\ov{\fa}'=\ov{\fa}$.

(c) To show the uniqueness of $\fa^*$, we can assume that $G=G^*$ and $\vk^*=\vk$.
Let $\fa_1,\fa_2:T\hra G$ be a pair of stably
conjugate embeddings such that $\vk\in \fa_1(\C{B}(T))\cap  \fa_2(\C{B}(T))$, and we want to show that $\fa_1$ and
$\fa_2$ are $G_{\vk}$-conjugate.

By the stable conjugacy of $\fa_1$ and $\fa_2$ we get that $\chi_G\circ \fa_1=\chi_G\circ \fa_2$,
thus we have an equality  $\chi_{G,\vk}\circ \fa_1=\chi_{G,\vk}\circ \fa_2$ of morphisms $T_{\C{O}}\to c_{G,\C{O}}$
over $\C{O}$. Hence we have an equality of morphisms of special fibers
$\chi_{L_{\vk}}\circ \ov{\fa}_1=\chi_{L_{\vk}}\circ \ov{\fa}_2:\ov{T}\to c_{L_{\vk}}$
(use \rl{surj} (a)), which implies that
$\ov{\fa}_1,\ov{\fa}_2:\ov{T}\hra L_{\vk}$ are stably conjugate. Thus $\ov{\fa}_1$ and $\ov{\fa}_2$ are
$L_{\vk}(\fq)$-conjugate (by Lang's theorem), hence $\fa_1$ and $\fa_2$ are $G_{\vk}$-conjugate by (b).

To show the existence of $\fa^*$, we assume first that $T$ is split. In this case, the assertion follows from
the fact that all split maximal tori in $G^*$ are conjugate. To deduce the general case, we now argue  as in (b).
\end{proof}
\end{Emp}

\subsection{Lie algebra analog of the Main Theorem}

\begin{Emp} \label{E:srtreg}
{\bf Notation.}
(a) Denote by $\C{T}^{\sr}_{G,\C{O}}\subset \C{T}_{G,\C{O}}$ be the set of all $t$ which are not fixed by
any non-trivial $w\in W_G$, set $c^{\sr}_{G,\C{O}}:=
W_G\bs \C{T}^{\sr}_{\mcG,\C{O}}\subset c^{\sr}_{\mcG,\C{O}}$,
and denote its special fiber by  $\ov{c}^{\sr}_{G}$.

(b) We say that $x\in \C{L}_{\vk}$ is {\em $G$-strongly regular} and write
$x\in \C{L}^{G-\sr}_{\vk}$, if $\ov{\chi}_{\mcG,\vk}(x)\in \ov{c}^{\sr}_{\mcG}$.
\end{Emp}

\begin{Lem} \label{L:gstr}
In the notation of \ref{N:stconj} (b),  for each $\ov{t}\in\C{T}(\fq)$ we have
$d\ov{\fa}(\ov{t})\in \mcL_{\vk}^{G-\sr}(\fq)$
if and only if  $d\ov{\fa}^*(\ov{t})\in \ov{\mcG}^{\sr}(\fq)$.
\end{Lem}

\begin{proof}
Choose a lift $t\in\C{T}(\mcO)$ of $\ov{t}\in\C{T}(\fq)$ and an embedding $\fa^*:T\hra G$ corresponding
to $\ov{\fa}^*:\ov{T}\hra \ov{G}=L_{\vk^*}$ as in \rl{maxtor} (b).
 Then  $d\ov{\fa}(\ov{t})\in \mcL_{\vk}^{G-\sr}(\fq)$
if and only if $\chi_{\mcG}(d\fa(t))\in c^{\sr}_{\mcG}(\C{O})$, while
$d\ov{\fa}^*(\ov{t})\in \ov{\mcG}^{\sr}(\fq)$
if and only if $\chi_{\mcG}(d\fa^*(t))\in c^{\sr}_{\mcG}(\C{O})$. Hence the assertion follows from the fact that
$\fa$ and $\fa^*$ are stably conjugate.
\end{proof}

\begin{Emp} \label{E:notlie}
{\bf Set-up.} (a) Fix $\ov{t}\in\C{T}(\fq)$ such that $d\ov{\fa}(\ov{t})\in \mcL_{\vk}^{G-\sr}(\fq)$,
denote by $\ov{\Om}_{\ov{\fa},\ov{t}}\subset \mcL_{\vk}(\fq)$
the $\Ad (L_{\vk}(\fq))$-orbit of $d\ov{\fa}(\ov{t})$, by $\Om_{\ov{\fa},\ov{t}}\subset \C{G}_{\vk}\subset \C{G}(F)$
the preimage of $\ov{\Om}_{\ov{\fa},\ov{t}}$, and let  $\ov{\dt}_{\ov{\fa},\ov{t}}$ and $\dt^G_{\ov{\fa},\ov{t}}$  be the
characteristic functions of $\ov{\Om}_{\ov{\fa},\ov{t}}$ and $\Om_{\ov{\fa},\ov{t}}$, respectively.
We also fix an embedding $\fa:T\hra G$ as in \rl{maxtor} (b).

(b) In the situation of \re{unr}, we consider the function
$\dt^H_{\ov{\fa},\ov{t}}:=\sum_{\ov{\fa}_H}\inv_{\mcE}(\fa_0^{(3)},\fa^{(3)})
\dt^H_{\ov{\fa}_H,\ov{t}}$, where $\ov{\fa}_H$ runs over a set of representatives
of the set of conjugacy classes of embeddings $\ov{\fa}_H:\ov{T}\hra\ov{H}$, which are
$\mcE$-stably conjugate to $\ov{\fa}$, and $\dt^H_{\ov{\fa}_H,\ov{t}}$ is as in (a).
\end{Emp}

\begin{Lem} \label{L:Lie}
(a) For an element $x\in \mcG_{\vk}$, we have $x\in\Om_{\ov{\fa},\ov{t}}$ if and only
if there exists an $F$-isomorphism $\fa_x:T\isom G_x\subset G$ such that $\fa_x$ is $G_{\vk}$-conjugate to
$\fa:T\hra G$,  and $t:=d\fa_x^{-1}(x)\in\C{T}(\C{O})$ is a lift of $\ov{t}$.

(b) For each  $x\in\Om_{\ov{\fa},\ov{t}}$ and $g\in G(F^{sep})$ such that $\Ad g(x)\in \Om_{\ov{\fa},\ov{t}}$,
we have $g\in G_{\vk}G_x(F^{sep})$.

(c) In the situation of \re{unr}, an element $x_H\in\C{H}^{G-\sr}(F)$ is $\C{E}$-stably conjugate to
an element  $x\in \Om_{\ov{\fa},\ov{t}}$ if and only if there exists an embedding  $\ov{\fa}_H:\ov{T}\hra \ov{H}$,
$\mcE$-stably conjugate to $\ov{\fa}$, such that $x_H$ is stably conjugate to an element
of  $\Om_{\ov{\fa}_H,\ov{t}}$. Moreover, the conjugacy class of such an $\ov{\fa}_H$
is uniquely determined by $x_H$.
\end{Lem}
\begin{proof}
(a) The ``only if'' assertion is clear, so we can assume that $x\in \Om_{\ov{\fa},\ov{t}}$.
Replacing $x$ by its $G_{\vk}$-conjugate, we can assume that $\ov{x}=d\ov{\fa}(\ov{t})$.
Since $x\in \C{G}_{\vk}$, we have $\chi_{\mcG}(x)\in c_{\mcG}(\mcO)$ (by \rl{surj} (a)), and
$\ov{\chi_{\mcG}(x)}=\ov{\chi}_{\mcG,\vk}(\ov{x})\in  \ov{c}^{\sr}_{\mcG}(\fq)$, hence
$\chi_{\mcG}(x)\in c^{\sr}_{\mcG}(\mcO)$. Since the projection
$\chi_{\C{G}}:\C{T}^{\sr}_{G,\mcO}\to c_{\C{G},\mcO}^{\sr}$ is \'etale, the group scheme
$\un{G}_x$ becomes isomorphic to
the unramified torus $T_G$ over some \'etale covering of $\C{O}$. Thus $G_x$ is an unramified torus.

Since $\ov{\fa}$ gives an isomorphism between reductions $\ov{T}\isom \ov{G}_x$, it lifts to unique
isomorphism $\fa_x:T\isom G_x$ over $\C{O}$. It remains to show that $\fa_x$ is $G_{\vk^+}$-conjugate to $\fa$.
Equivalently, it will suffice to show that $x$ is  $G_{\vk^+}$-conjugate
to an element of $d\fa(\C{T})$ or that $\vk\in \mcB(G_x)$ (by \rl{maxtor} (b)).

To show that $\vk\in \mcB(G_x)$, we can extend scalars to $F^{nr}$. As in the proof of \rl{surj} (a),
one reduces to the case when $\vk$ is hyperspecial.

We claim that for any lift $t\in \mcT(\mcO)$ of $\ov{t}$, we have
$x\in\Ad(G_{\vk^+})(t+\frak{m}\mcT(\mcO))$. Consider the morphism $\mu:\mcT_{\mcO}\times \un{G}_{\vk}\to \C{G}_{\vk}$
given by the rule $\mu(t,g):=\Ad g(d\fa(t))$. By Hensel's lemma, it will suffice to show that $\mu$ is smooth at
the point $(t,1)$. But this follows from the smoothness of the corresponding morphism
$\ov{\mu}:\ov{\mcT}^{\sr}\times L_{\vk}\to \C{L}_{\vk}$ over $\fq$.

(b)  Set $y:=\Ad g(x)$. By (a), $x$ gives rise to an embedding $\fa_x:T\isom G_x\subset G$, hence
to an embedding $\fa'_x=g\fa_xg^{-1}:T\isom G_y\subset G$. We have to show that $\fa_x$ and $\fa'_x$ are
$G_{\vk}$-conjugate. By (a), $y$ gives rise to an embedding $\fa_y:T\isom G_y\subset G$, which is  $G_{\vk}$-conjugate of $\fa_x$.
It thus suffices to show that $\fa'_x=\fa_y$.

Notice first that the isomorphisms
 $\ov{\fa}'_x:\ov{T}\isom\ov{G}_{y}$ and $\ov{\fa}_y:\ov{T}\isom\ov{G}_{y}$ satisfy
$d\ov{\fa}'_x(\ov{t})=d\ov{\fa}_y(\ov{t})=\ov{y}$. On the other hand, $\fa'_x$ and $\fa_y$ are stably conjugate, hence there
exists $w\in W_G\subset\Aut(G_y)$ such that $\fa'_x=w\circ \fa_y$. It follows that
$w(\ov{y})=\ov{y}\in \mcL_{\vk}^{G-\sr}(\fq)$, which implies that $w=1$, thus  $\fa'_x=\fa_y$.

(c) Assume first that $x_H\in\C{H}^{G-\sr}(F)$ is $\C{E}$-stably conjugate to
an element  $x\in \Om_{\ov{\fa},\ov{t}}$. Then by (a), there exists an isomorphism
$\fa_{x_H}:T\isom H_{x_H}\subset H$ stably conjugate to $\fa:T\hra G$ such that $t:=d\fa_{x_H}^{-1}(x_H)\in\mcT(\mcO)$ is a
lift of $\ov{t}$. By \rl{maxtor} (c), there exists a stably conjugate embedding $\fa_H:T\hra H$ such that
$\vk^H\in \fa_H(\C{B}(T))$. Then the reduction $\ov{\fa}_H:\ov{T}\hra \ov{H}$ is an $\C{E}$-stable conjugate of $\ov{\fa}$,
and $d\fa(t)\in\Om_{\ov{\fa}_H,\ov{t}}$ is a stable conjugate of $x_H$.

Conversely, assume that  $x_H\in\C{H}^{G-\sr}(F)$ is stably conjugate to
an element of some $\Om_{\ov{\fa}_H,\ov{t}}$. Then  by (a) there exists unique isomorphism
$\fa_{x_H}:T\isom H_{x_H}\subset H$ stably conjugate to $\fa:T\hra G$ such that $t:=d\fa_{x_H}^{-1}(x_H)$ is a lift
of $\ov{t}$. Then $x_H$ is $\C{E}$-stably conjugate to an element $x:=d\fa(t)\in \mcG_{\vk}$, which by definition
belongs to   $\Om_{\ov{\fa},\ov{t}}$.

Finally, assume that  $x_H\in \Om_{\ov{\fa}_H,\ov{t}}$ is stably conjugate to  $y_H\in \Om_{\ov{\fb}_H,\ov{t}}$ for some
embeddings $\ov{\fa}_H,\ov{\fb}_H:\ov{T}\hra \ov{H}$, which are $\mcE$-stably conjugate to $\ov{\fa}$, and
we want to show that  $\ov{\fa}_H$ and $\ov{\fb}_H$ are  $\ov{H}(\fq)$-conjugate.

Let  $\fa_H,\fb_H:T\hra H$ be lifts of  $\ov{\fa}_H,\ov{\fb}_H$ from \rl{maxtor} (b), and let $h\in H(F^{sep})$ be such that
$y_H=\Ad h (x_H)$. By (a), there exist embeddings $\fa_{x_H}:T\isom H_{x_H}\subset H$ and $\fa_{y_H}:T\isom H_{y_H}\subset H$
which are $H_{\vk^H}$-conjugate to $\fa_H$ and $\fb_H$, respectively. Arguing as in (b), we conclude that
$\fa_{y_H}=h \fa_{x_H}h^{-1}$, thus $\fa_{x_H}$ and $\fa_{y_H}$ (and hence also $\fa_H$ and $\fb_H$) are stably conjugate.
It follows that $\fa_H$ and $\fb_H$ are $H_{\vk^H}$-conjugate
(by  \rl{maxtor} (c)), thus  $\ov{\fa}_H$ and $\ov{\fb}_H$ are  $\ov{H}(\fq)$-conjugate.
\end{proof}

\begin{Lem} \label{L:trfaclie}
In the notation of \rl{Lie} (c), assume that $x_H\in \Om_{\ov{\fa}_H,\ov{t}}$ is $\C{E}$-stably
conjugate to  $x\in \Om_{\ov{\fa},\ov{t}}$. Then we have an equality
$\Dt^{\wt{\mcE}}(x_H,x)=\inv_{\mcE}(\fa_0^{(3)},\fa^{(3)})$.
\end{Lem}

\begin{proof}
By \rl{Lie} (a), the centralizer $G_x\subset G$ is an unramified torus, thus
by \rl{maxtor} (c) we can assume that $\iota^*$ gives an embedding
$T_{\C{O}}\hra \un{G}^*_{\vk^*}$ defined over $\C{O}$. Hence we can take
the $a$-data to be defined over $\C{O}$, in which case $\inv(a)$ is a cohomology class in
$H^1(F^{nr}/F,G_{x}(\C{O}^{nr}))=0$ (by a version of Lang's theorem), thus
$\Dt_I(x_H,x)=1$ (compare \cite[Proof of Lem. 7.2]{Ha}). Next we note that
$\Dt_{II}(x_H,x)=\Dt_{IV}(x_H,x)=1$, because $\al(x)\in\ov{\C{O}}\m$ for all $\al$.
Finally, by \rl{Lie} (a), the triple $\iota^{(3)}$ is conjugate to $\fa^{(3)}$, therefore
$\Dt(x_H,x)=\Dt_{III_1}(x_H,x)=\inv_{\C{E}}(\fa_0^{(3)},\iota^{(3)})=
\inv_{\C{E}}(\fa_0^{(3)},\fa^{(3)})$.
\end{proof}

The following result is a Lie algebra analog of our Main Theorem
(compare \cite[Lem 2.2.9]{KV}).

\begin{Prop} \label{P:Lie}
(a) If $H$ is ramified, then  $\dt^G_{\ov{\fa},\ov{t}}$ is $\C{E}$-unstable.

(b) If $H$ is unramified then, in the notation of \re{unr}, function $\dt^H_{\ov{\fa},\ov{t}}$ is an
endoscopic transfer of  $\dt^G_{\ov{\fa},\ov{t}}$.
\end{Prop}

\begin{proof}
(a) If $H$ is ramified, then for every $\C{E}$-stably conjugate $x_H\in\mcH^{\sr}(F)$ and $x\in \mcG^{\sr}(F)$,
the torus $G_x\cong H_{x_H}$ is a ramified, hence $x$ is not stably conjugate to an element of $\Om_{\ov{\fa},\ov{t}}$
(by \rl{Lie} (a)), hence $O^{\ka_{x,x_H}}_{x}(\dt^G_{\ov{\fa},\ov{t}})=0$.

(b) We want to show that for every $x_H\in\C{H}^{G-\sr}(F)$ we have an equality
\begin{equation} \label{Eq:Lie}
O^{st}_{x_H}(\dt^H_{\ov{\fa},\ov{t}})=O^{\wt{\mcE}}_{x_H}(\dt^G_{\ov{\fa},\ov{t}}).
\end{equation}

It follows from \rl{Lie} (c) that both sides of
\form{Lie} vanish, unless $x_H$ is stably conjugate to $\Om^H_{\ov{\fa}_H,\ov{t}}$ for some embedding
 $\ov{\fa}_H:\ov{T}\hra\ov{H}$, which is $\C{E}$-stably conjugate to $\ov{\fa}$,
and  $x_H$ is $\C{E}$-stably conjugate to an element $x\in\Om_{\ov{\fa},\ov{t}}$.
From now on we will assume that the last conditions are satisfied.

By \rl{Lie} (a), (b), for each $\ka\in\wh{G_x}^{\Gm}$, we have
$O^{\ka}_{x}(\dt^G_{\ov{\fa},\ov{t}})=O_{x}(\dt^G_{\ov{\fa},\ov{t}})=1$. Hence the right
hand side of
\form{Lie} equals $\Dt(x_H,x)O^{\ka_{x,x_H}}_{x}(\dt^G_{\ov{\fa},\ov{t}})=
\Dt(x_H,x)$.

On the other hand,  $x_H$ is not stably conjugate to $\Om^H_{\ov{\fb}_H,\ov{t}}$ for each
$\ov{\fb}_H:\ov{T}\hra\ov{H}$ which is not conjugate to $\ov{\fa}_H$
(by the last assertion of \rl{Lie} (c)).
Thus it follows from  \rl{Lie} (a),(b) that
$O^{st}_{x_H}(\dt^H_{\ov{\fa},\ov{t}})=
\inv_{\mcE}(\fa^{(3)}_0,\fa^{(3)})O^{st}_{x_H}(\dt^H_{\ov{\fa}_H,\ov{t}})=
\inv_{\mcE}(\fa^{(3)}_0,\fa^{(3)})$.

Now \form{Lie} follows from the equality $\Dt(x_H,x)=\inv_{\mcE}(\fa^{(3)}_0,\fa^{(3)})$,
which was proven in \rl{trfaclie}.
\end{proof}

\section{Endoscopic transfer on Lie algebras and Fourier transform}

\subsection{A theorem of Waldspurger}

\begin{Emp} \label{E:qfna}
{\bf Preliminaries on quadratic forms.} We fix a non-trivial additive character $\psi:F\to\B{C}\m$.

(a) Let $Q:V\times V\to F$ be a non-degenerate quadratic form. Then for each Haar measure
$\mu$  on $V(F)$ the composition $\psi\circ Q$ gives rise to a Fourier transform
$\C{F}=\C{F}_{Q,\psi,\mu}:C_c^{\be}(V(F))\to C_c^{\be}(V(F))$. Moreover, there exists unique Haar measure
$\mu=\mu_{Q,\psi}$ such that$\C{F}_{Q,\psi,\mu}$ is a unitary operator, and we set
$\C{F}_{Q,\psi}:=\C{F}_{Q,\psi,\mu_{Q,\psi}}$.

(b) Recall that to every non-degenerate quadratic form $Q$ over $F$, Weil
(\cite{We}) associated an $8^{th}$ root of unity $\gm_{\psi}(Q)\in\B{C}\m$. Moreover, the map
$Q\mapsto \gm_{\psi}(Q)$ is "additive", that is, $\gm_{\psi}(Q_1\oplus Q_2)=\gm_{\psi}(Q_1)\gm_{\psi}(Q_2)$ and
$\gm_{\psi}(-Q)=\gm_{\psi}(Q)^{-1}$. In particular, $\gm_{\psi}(Q)=1$, if $Q$ is split (=hyperbolic).
\end{Emp}

\begin{Emp} \label{E:qfeg}
{\bf Endoscopic Lie algebras.}
Let $H$ be an endoscopic group for $G$, and let $Q_G:\C{G}\times \C{G}\to F$ be a  non-degenerate $G$-invariant
quadratic form on $\C{G}$.

(a) As it shown in \cite[VIII.6]{Wa1}, $Q_G$ gives rise a non-degenerate $H$-invariant
quadratic form $Q_H$ on $\C{H}$. Namely, $Q_G$ is naturally an element of $F[\mcG]^G=F[c_{\mcG}]$, and
$Q_H\in F[\C{H}]^H=F[c_{\mcH}]$ is just a pullback of  $Q_G\in F[c_{\mcG}]$ under
the projection $\nu:c_{\mcH}\to c_{\mcG}$.

(b) As in \re{qfna} (a), quadratic forms $Q_G$ and $Q_H$ give rise to Fourier transforms
$\C{F}_G=\C{F}_{Q_G,\psi}$ and $\C{F}_H=\C{F}_{Q_H,\psi}$.

(c) We set $e_{\psi}(H,G):=\gm_{\psi}(Q_H)/\gm_{\psi}(Q_G)$.
\end{Emp}

The following theorem of Waldspurger (completed by Ngo) is crucial
for our argument.

\begin{Thm} \label{T:Wa}
 If a function $\phi^H\in C_c^{\be}(\C{H}(F))$ is an endoscopic transfer of
$\phi\in C_c^{\be}(\C{G}(F))$, then the function
$e_{\psi}(H,G)\C{F}_H(\phi^H)$ is an endoscopic transfer of the Fourier
transform $\C{F}_G(\phi)$.
\end{Thm}

\begin{proof}
In \cite{Wa1} Waldspurger formulated this result as
Conjecture 1 (see \cite[p. 91]{Wa1}), and deduces it from another
Conjecture 2 (see \cite[p. 92--94]{Wa1}). Next the Main result of
\cite{Wa2} asserts that  Conjecture 2 follows from the fundamental
lemma for Lie algebras over $p$-adic fields of sufficiently large
residual characteristic. Then it was shown in \cite{Wa3} that it
will suffice to show the fundamental lemma for Lie algebras for
local fields of sufficiently large positive characteristic.
The latter result was recently shown by Ngo in \cite{Ngo}.
\end{proof}

\begin{Rem}
As indicated above, Waldspurger proved that \rt{Wa} follows from the fundamental lemma
for unit elements. In this work we prove the converse implication: \rt{Wa}
implies the fundamental lemma. Note also that Waldspurger's proof is global, while
our proof is purely local.
\end{Rem}

\subsection{Calculation of Weil constant}

\begin{Not}
We say that an additive character $\psi:F\to\B{C}\m$ is {\em non-degenerate over $\C{O}$},
if it induces a non-trivial additive character of $\fq$, that is, $\psi$ is
non-trivial on $\C{O}$, but is trivial on $\fm$.
\end{Not}

To apply \rt{Wa}, we need to calculate $e_{\psi}(H,G)$ explicitly.

\begin{Prop} \label{P:weil}
Assume that $p\neq 2$, that $G$ and $H$ are split over $F^{nr}$, and that $\psi$ and
$Q_G$ are non-degenerate over $\C{O}$.
Then $e_{\psi}(H,G)=(-1)^{\rk_F(H)-\rk_F(G)}$.
\end{Prop}

Our proof is based on the following three lemmas, all of which seem to be well known to specialists.

Note that if $Q$ and $Q'$ are two non-degenerate quadratic form over a field $F$ of the same rank, then
$Q'$  is ``form'' of $Q$, thus $Q'$ defines a cohomology class $c_{Q',Q}\in H^1(F,O(Q))$. We denote by
$\det c_{Q',Q}\in H^1(F,\{\pm 1\})=\Hom(\Gm_{F},\{\pm 1\})$ the image of $c_{Q',Q}$ under
the determinant map $\det:O(Q)\to\{\pm 1\}$.

\begin{Lem} \label{L:disc}
Let  $Q$ and $Q'$ be two non-degenerate  quadratic forms of the same rank over a field $F$ of characteristic
different from two, and let $d\in F\m/(F\m)^2$ be the quotient of discriminants $d(Q')/d(Q)$.

Then $\Gm_{F[\sqrt{d}]}\subset\Gm_F$ is the kernel of
$\det c_{Q',Q}:\Gm_F\to\{\pm 1\}$.
\end{Lem}

\begin{proof}
Note that if $Q'$ is a form of $Q$  corresponding to $c\in H^1(F,O(Q))$, then the determinant $\det Q'$ is the form of the quadratic form $\det Q$ corresponding to $\det(c)\in H^1(F,O(\det Q))$.

Hence we can replace $Q$ and $Q'$ by their determinants, thus assuming that
$Q$ is a rank one form $Q(x)=ax^2$ and that $Q'(x)=adx^2$. In this case,
$\det c_{Q',Q}:\Gm_F\to\{\pm 1\}$ is the homomorphism
$\si\mapsto\si(\sqrt{d})/(\sqrt{d})$, whose kernel  is equal to
$\Gm_{F[\sqrt{d}]}$.
\end{proof}

\begin{Lem} \label{L:formo}
Assume that $p\neq 2$, and let $R$ be a non-degenerate over $\C{O}$ quadratic form of
even rank. Then $\gm_{\psi}(R)=1$ if  $d(R)(-1)^{\rk R/2}=1\in F\m/(F\m)^2$, and  $\gm_{\psi}(R)=-1$
otherwise.
\end{Lem}

\begin{proof}
Consider the rank two quadratic form $R_0(x):=N_{F^{(2)}/F}(x)$, where $F^{(2)}$ is the unique unramified
quadratic extension of $F$. By direct calculation, one sees that $d(R_0)\neq -1\in  F\m/(F\m)^2$ and
$\gm_{\psi}(R_0)=-1$ (see, for example, \cite[p. 6]{JL}), proving the assertion in this case.

To prove the general case, recall that
each non-degenerate $\C{O}$ quadratic form $R$ has discriminant
$d(R)\in\C{O}\m/(\C{O}\m)^2$ and Hasse--Witt invariant
$\epsilon(R)=1$ (see \cite[p. 35]{Se}). Since a quadratic form $R$ is uniquely determined by its
invariants
$(\rk(R),d(R), \epsilon(R))$ (see, for example, \cite[Thm. 7, p. 39]{Se}), there are at most two non-isomorphic
quadratic forms of the same rank, which are non-degenerate over $\C{O}$.
In particular, each such even dimensional $R$ is either a split one or is isomorphic to a
direct sum of a split one and $R_0$. Since the assertion for the split $R$ is obvious,
the general case follows.
\end{proof}

\begin{Cor} \label{C:formo}
Let  $Q$ and $Q'$ be two non-degenerate over $\C{O}$ quadratic forms of the same rank.
Then $\gm_{\psi}(Q')/\gm_{\psi}(Q)=1$ if $d(Q')/d(Q)=1$, and
$\gm_{\psi}(Q')/\gm_{\psi}(Q)=-1$ otherwise.
\end{Cor}

\begin{proof}
By the additivity of $\gm_{\psi}$ we have $\gm_{\psi}(Q')/\gm_{\psi}(Q)= \gm_{\psi}(Q'\oplus(-Q))$.
Hence the assertion follows from \rl{formo} and equalities
$d(Q'\oplus(-Q))=(-1)^{\rk Q} d(Q')d(Q)$ and $\rk(Q'\oplus(-Q))=2\rk Q$.
\end{proof}

\begin{Lem} \label{L:det}
Let $W$ be a finite dimensional vector space over $\B{Q}$, and let
$g\in \Aut_{\B{Q}}(W)$ be an element of finite order. Then $\det g=(-1)^{\dim W-\dim W^g}$.
\end{Lem}

\begin{proof}
Let $C_n\subset\Aut_{\B{Q}}(W)$ be the finite cyclic group of order $n$ generated by $g$.
We want to show that for every finite dimensional representation $\rho:C_n\to\Aut_{\B{Q}}(V)$
over $\B{Q}$, we have an equality $\det g=(-1)^{\dim V-\dim V^{C_n}}$. The assertion is easy
for the regular representation $V_{reg}=\B{Q}[C_n]$. Indeed, in this case, $\dim V_{reg}=n,
\dim V_{reg}^{C_n}=1$, while $\det g=(-1)^{n-1}$.

By induction on $n$, we can assume that the
assertion holds for each not faithful representation $(\rho, V)$.
Thus we can assume that $V$ is faithful and irreducible. However,
there exists  unique (up to an isomorphism) faithful and irreducible representation
$V_n$ of $C_n$. Moreover, $V_{reg}$ decomposes
as a direct sum $V_{reg}\cong V_{n}\oplus V'$, where $V'$ is a sum of
irreducible not faithful representations of $C_n$.
Since the assertion holds for both $V_{reg}$ and $V'$, it also holds for $V_n$.
\end{proof}

\begin{Cor} \label{C:det}
Let $T$ be torus over $F$, which splits over $F^{nr}$. Then the Frobenius automorphism
$\Fr_F\in\Aut X_*(T)$ satisfies $\det(\Fr_F)=(-1)^{\dim T-\rk_F T}$.
\end{Cor}
\begin{proof}
Since $T$ splits over a finite unramified extension of $F$, the element
$\Fr_F\in\Aut X_*(T)\subset\Aut_{\B{Q}}(X_*(T)\otimes_{\B{Z}}\B{Q})$ is of finite
order. Since $X_*(T)^{\Fr_F}=X_*(T_{sp})$,
where $T_{sp}\subset T$ is the maximal split subtorus of $T$, we get that
$\rk_F(T)=\dim T_{sp}=\dim_{\B{Q}}(X_*(T)\otimes_{\B{Z}}\B{Q})^{\Fr_F}$, so the
assertion follows from the lemma.
\end{proof}

\begin{Emp} \label{E:pfweil}
\begin{proof}[Proof of \rp{weil}]
In the case when $H=G^*$, the assertion is shown in \cite[Remark after 2.7.5]{KP} (for an arbitrary $G$).
Since  $e_{\psi}(H,G)=e_{\psi}(H,G^*)e_{\psi}(G^*,G)$, we can assume that $G$ is quasi-split.

Let $T_G\subset B_G\subset G$ be a maximal torus and a Borel subgroup of $G$,
both defined over $F$. Since the quadratic form $Q_G$ is non-degenerate and $G$-invariant,
it decomposes as the orthogonal sum
$Q_G|_{\mcT_G}\oplus Q_G|_{\mcU^+\oplus\mcU^-}$, where $\mcU^+$ (resp.  $\mcU^-$)
is the Lie algebra of the unipotent radical of $B_G$ (the opposite Borel subgroup).
However, the quadratic form $Q_G|_{\mcU^+\oplus\mcU^-}$ is split. Indeed,
$\mcU^+$ is an isotropic subspace of $\mcU^+\oplus\mcU^-$ of half-dimension.

It follows that $\gm_{\psi}(Q_G)=\gm_{\psi}(Q_{T_G})$, where we write
$Q_{T_G}$ instead of $Q_G|_{\mcT_G}$.
Similarly,  $\gm_{\psi}(Q_H)=\gm_{\psi}(Q_{T_H})$, where $T_H$ is a maximal torus of
a Borel subgroup $B_H\subset H$ defined over $F$. Therefore
$\gm_{\psi}(Q_H)/\gm_{\psi}(Q_G)=\gm_{\psi}(Q_{T_H})/\gm_{\psi}(Q_{T_G})$.

Let $\iota:T_H\isom \wt{T}_H\subset G$ be an embedding, which is $\C{E}$-stable conjugate to the
inclusion $T_H\hra H$. By the definition of $Q_H$ (\re{qfeg} (a)) and the commutativity of the analog of
\form{nu} for Lie algebras, the restriction $Q_{T_H}=Q_H|_{\mcT_H}$ is isomorphic to
$Q_{\wt{T}_H}=Q_G|_{\wt{\mcT}_H}$. Since $\rk_F(G)=\rk_F(T_G)$ and
 $\rk_F(H)=\rk_F(T_H)=\rk_F(\wt{T}_H)$, it will therefore remains to show that
\begin{equation} \label{Eq:det}
\gm_{\psi}(Q_{\wt{T}_H})/\gm_{\psi}(Q_{T_G})=(-1)^{\rk_F(\wt{T}_H)-\rk_F(T_G)}.
\end{equation}

Recall that both $T_G$ and $\wt{T}_H$ are maximal tori in $G$, which are split
over $F^{nr}$. Hence there exists $g\in G(F^{nr})$ such that
$\wt{T}_{H}=g T_G g^{-1}$. Then $g^{-1}\Fr_F(g)\in N_G(T_G)$. We denote by
$\ov{g}\in W_G$ the image of $g^{-1}\Fr_F(g)$ and by  $\sgn\ov{g}\in\{\pm 1\}$ the image of $\ov{g}$ under the
sign homomorphism $W_G\to\{\pm 1\}$. We are going to show that both sides of \form{det} are equal to
$\sgn\ov{g}$.

Notice first that since $Q_G$ is non-degenerate over $\C{O}$, its restriction
$Q_{T_G}$ is non-degenerate over $\C{O}$ as well. Since
the quadratic form $Q_G$ is $G$-invariant, and $\Ad g$ induces an
 $F^{nr}$-isomorphism $\mcT_G\isom\wt{\mcT}_H$,  $\Ad g$ induces an $F^{nr}$-isomorphism
between $Q_{T_G}$ and $Q_{\wt{T}_H}$. Hence  $Q_{\wt{T}_H}$ is non-degenerate over $\C{O}$ as well.

Now it follows from \rco{formo} that  $\gm_{\psi}(Q_{\wt{T}_H})/\gm_{\psi}(Q_{T_G})=1$ if the discriminants satisfy
$d(Q_{\wt{T}_H})/d(Q_{T_G})=1\in F\m/(F\m)^2$, and  $\gm_{\psi}(Q_{\wt{T}_H})/\gm_{\psi}(Q_{T_G})=-1$
otherwise.

Observe that the cohomology class $c_{Q_{\wt{T}_H},Q_{T_G}}\in  H^1(\Gm_F,O(Q_{T_G}))$ is represented by a
cocycle
$(\wt{c}_{\si})_{\si}\in Z^1(\Gal(F^{nr}/F),O(Q_{T_G}))$ such that $\wt{c}_{\Fr_F}\in O(Q_{T_G})\subset \Aut(\C{T}_G)$
is the image of $\ov{g}$ under the natural embedding $W_G\hra\Aut(\C{T}_G)$.

Then
$\det \wt{c}_{\Fr_F}=\sgn\ov{g}$. Hence it follows from \rl{disc} that  $d(Q_{\wt{T}_H})/d(Q_{T_G})=1$ if
$\sgn\ov{g}=1$, and  $d(Q_{\wt{T}_H})/d(Q_{T_G})\neq 1$ if $\sgn\ov{g}=-1$.
These observations imply that   $\gm_{\psi}(Q_{\wt{T}_H})/\gm_{\psi}(Q_{T_G})=1$ if $\sgn\ov{g}=1$,
and  $\gm_{\psi}(Q_{\wt{T}_H})/\gm_{\psi}(Q_{T_G})=-1$ if $\sgn\ov{g}=-1$, which is equivalent to the fact that
$\sgn\ov{g}$ equals the left hand side of \form{det}.

To show that $\sgn\ov{g}$ is equal to the right hand side of \form{det},  we consider the natural embedding
$\iota:W_G\hra\Aut_{F^{sep}}(T)\isom\Aut_{\B{Z}} X_*(T)$ and notice that
$\sgn\ov{g}=\det\iota(\ov{g})$. By definition, $\iota(\ov{g})\in \Aut_{\B{Z}} X_*(T_G)$
decomposes as
\[
 X_*(T_G)\overset{\Fr_F^{-1}}{\lra} X_*(T_G)\overset{\Int g}{\lra} X_*(\wt{T}_H)
\overset{\Fr_F}{\lra} X_*(\wt{T}_H)\overset{\Int g^{-1}}{\lra} X_*(T_G).
\]
Hence  $\sgn\ov{g}=\det \iota(\ov{g})$ equals
\[
\det(\Fr_F^{-1})\det(g^{-1}\Fr_F g)=\det(\Fr_F, X_*(T_G))^{-1}\det(\Fr_F, X_*(\wt{T}_H)).
\]
By \rco{det}, the latter expression is equal to
\[
(-1)^{\rk_F(T_G)-\dim T_G}(-1)^{\dim \wt{T}_H-\rk_F(\wt{T}_H)}=(-1)^{\rk_F(\wt{T}_H)-\rk_F(T_G)},
\]
proving the assertion.
\end{proof}
\end{Emp}

\section{Quasi-logarithm maps}

In this section we study quasi-logarithm
maps, introduced in \cite[1.8]{KV}.

\subsection{Quasi-logarithms for endoscopic groups.}
\begin{Def} \label{D:qlog}
Let $G$ be a reductive group over a field $F$. A {\em quasi-logarithm} for $G$
is a $G^{\ad}$-equivariant morphism
$\Phi:G\to\mcG$ such that $\Phi(1)=0$ and the differential
$d\Phi_1:\mcG\to\mcG$ is the identity map.
\end{Def}

\begin{Rem} \label{R:qlog}
Let $\Phi$ be a quasi-logarithm for $G$.

(a) Since $\Phi$ is $G$-equivariant, it induces a morphism
$[\Phi]:c_{G}\to c_{\mcG}$.

(b) For each centralizer $M:=G_S\subset G$ of a torus $S\subset G$, its Lie algebra
$\C{M}$ is equal to the centralizer $\C{G}_{S}$, hence $\Phi$ induces a quasi-isomorphism
$\Phi|_M:M\to\mcM$ for $M$. In particular, this happens in the case when $M\subset G$ is a Levi subgroup or
a maximal torus of $G$.
\end{Rem}

\begin{Prop} \label{P:qlogend}
Let $H$ be an endoscopic group of $G$. Then for every quasi-logarithm
$\Phi:G\to\mcG$ for $G$ there exists unique quasi-logarithm
$\Phi^H:H\to\mcH$ for $H$ making the following diagram commutative:
\begin{equation} \label{Eq:qlog}
\CD
        c_H  @>[\Phi^H]>>                    c_{\mcH}\\
        @V{\nu}VV                        @V{\nu}VV\\
        c_G  @>[\Phi]>>                    c_{\mcG}
\endCD
\end{equation}
\end{Prop}

\begin{proof}
To make the proof more structural, we divide it into steps.

{\bf Step 1.} Notice first that we can extend scalars to $F^{sep}$. Indeed,
the existence and the uniqueness of $\Phi^H$ over  $F^{sep}$ implies  that
${}^{\si}\Phi^H=\Phi^H$ for each $\si\in \Gm_F$, thus  $\Phi^H$ is defined over $F$.
From now on we extend scalars to $F^{sep}$, thus assuming that $F$ is
separably closed.

{\bf Step 2. Uniqueness I.}
We claim that for every  maximal torus  $T_H\subset H$ there exists at most  unique
 quasi-logarithm $\Phi^{T_H}:T_H\to\mcT_H$ such that the diagram
\begin{equation} \label{Eq:qlog1}
\CD
        T_H  @>\Phi^{T_H}>>                    \mcT_H\\
        @V{\nu\circ\chi_H}VV                        @V{\nu\circ\chi_{\mcH}}VV\\
        c_G  @>[\Phi]>>                    c_{\mcG}.
\endCD
\end{equation}
is commutative.

Indeed, let $\Phi_1$ and $\Phi_2$ be two quasi-logarithms $T_H\to\mcT_H$ making the diagram \form{qlog1}
commutative. First we will show that there exists $w\in W_G$ such that  $\Phi_2=w\circ \Phi_1$.

By the commutativity of the Lie algebra version of \form{nu}, the composition
 $\nu\circ\chi_{\mcH}:\mcT_H\to  c_{\mcG}$ decomposes as $\mcT_H\isom\mcT\to W_G\bs\mcT= c_{\mcG}$. Hence
 the field of rational functions $F(\mcT_H)$ is a Galois extension of $F(c_{\mcG})$ with the Galois group $W_G$. Since
quasi-logarithms $\Phi_1$ and  $\Phi_2$ are dominant, they induce embeddings of fields
$\Phi_1^*,\Phi_2^*:F(\mcT_H)\hra F(T_H)$, extending $[\Phi]^*:F(c_{\mcG})\hra F(c_G)$.
Hence there exists $w\in W_G$ such that  $\Phi_2^*=\Phi_1^*\circ w^*$, thus  $\Phi_2=w\circ \Phi_1$.

Now since both $\Phi_1$ and $\Phi_2$ are quasi-logarithms, the differential  $dw\in\Aut(\C{G})$
has to be the identity morphism. It follows that $w=1$, thus $\Phi_1=\Phi_2$.

{\bf Step 3. Uniqueness II.}
Let $\Phi:H\to\mcH$ be any quasi-logarithm such that the diagram \form{qlog} is
commutative. Then for every  maximal torus  $T_H\subset H$ the quasi-logarithm
$\Phi^{T_H}:=\Phi^H|_{T_H}:T_H\to\mcT_H$ (use \rr{qlog} (b)) induces a commutative diagram
\form{qlog1}. By Step 2, the restriction  $\Phi^H|_{T_H}$ is unique, hence the uniqueness 
of $\Phi^H$ follows from the fact that the union $\cup T_H\subset H$ is Zariski dense.

{\bf Step 4. Existence: first reduction.}
For each maximal torus  $T_H\subset H$, we choose an embedding
$\iota:T_H\isom T\subset G$ which is  $\C{E}$-stably conjugate to the inclusion
$T_H\hra H$, and consider the composition
\[
\Phi^{T_H}_{\iota}:T_H\overset{\iota}{\lra}T\overset{\Phi|_T}{\lra}\mcT
\overset{d\iota^{-1}}{\lra}\mcT_H\subset\mcH
\]
(use \rr{qlog} (b)). Since $\iota$ is unique up to a $G$-conjugacy and $\Phi$ is $G$-equivariant, the constructed morphism $\Phi^{T_H}_{\iota}$ is independent of $\iota$, so we will denote it
simply by  $\Phi^{T_H}$.

Since $\Phi|_T:T\to \mcT$ is a quasi-logarithm for $T$, $\Phi^{T_H}$ is a
quasi-logarithm for $T_H$. By the construction of  $\Phi^{T_H}$, we have a commutative diagram
\begin{equation} \label{Eq:qlog2}
\CD
        T_H   @>\Phi^{T_H}>>                    \mcT_H\\
        @V{\Int h}VV                        @V{\Ad h}VV\\
h T_H h^{-1}  @>\Phi^{hT_H h^{-1}}>>                    \Ad h(\mcT_H).
\endCD
\end{equation}
 for each $h\in H(F)$. In particular, the morphism $\Phi^{T_H}:T_H\to\mcT_H$ is $N_H(T_H)$-equivariant.

We claim that to prove the existence assertion of the proposition
it will suffice to show the existence of a morphism $\Phi^H:H\to\mcH$
such that $\Phi^H|_{T_H}=\Phi^{T_H}$ for each maximal torus $T_H$.

Indeed, since $\cup T_H\subset H$ is Zariski dense, it follows from \form{qlog2} that such a $\Phi^H$  is automatically $H^{ad}$-equivariant.
Next, conditions $\Phi^H(1)=0$ and $d\Phi^H_1=\Id$ follow from the fact that
each $\Phi^{T_H}$ is a quasi-logarithm for $T_H$ and that the
$\cup \mcT_H$ spans (and is actually dense in) $\mcH$. Finally, to show the commutativity of
\form{qlog} it suffices to show the commutativity of \form{qlog1}. The latter
follows from the commutativity of \form{nu} and its Lie algebra analog.

{\bf Step 5. Rank one case.}
We claim that for every Levi subgroup $M_H\subset H$ of semisimple rank one,
there exists a morphism $\Phi^{M_H}:M_H\to\mcM_H\subset\mcH$ such that
$\Phi^{M_H}|_{T_H}=\Phi^{T_H}$ for each maximal torus $T_H\subset M_H$.

To construct $\Phi^{M_H}$, we choose a maximal torus $T_H\subset M_H$, an embedding
$\iota:T_H\isom T\subset G$ as before, and set $M$ to be the centralizer $M:=G_{\iota(Z(M_H)^0)}$.

Then $M$ is a Levi subgroup of $G$ of semisimple rank one, and $\iota$ induces an isomorphism
$\iota_M:M_H\isom M$, unique up to a conjugacy. To show it, we have to check that
$\iota$ induces a bijection
 $R\,{\check{}}(\iota):R\,\check{}(M_H,T_H)\hra R\,\check{}(M,T)$ on coroots.

Since  $M_H$ is an endoscopic group for $M$,
$\iota$ induces an injection on coroots $R\,\check{}(\iota):R\,\check{}(M_H,T_H)\hra R\,\check{}(M,T)$.
On the other hand, since  $M_H\subset H$ of semisimple rank one, we see that $R\,\check{}(M_H,T_H)$
consists of two elements and $Z(M_H)^0\subset T_H$  is of codimension one. Hence
$\iota(Z(M_H)^0)\subset T$ is of codimension one, therefore $R\,\check{}(M,T)$ consists
of at most two elements. Thus $R\,\check{}(\iota)$ is a bijection.

We claim that the composition
\[
\Phi^{M_H}:M_H\overset{\iota_M}{\lra}M\overset{\Phi|_M}{\lra}\mcM
\overset{d\iota_M^{-1}}{\lra}\mcM_H\subset\mcH
\]
satisfies $\Phi^{M_H}|_{T_H}=\Phi^{T_H}$ for each maximal torus $T'_H\subset M_H$.
For this we have to check that for every maximal torus $T'_H\subset M_H\subset H$, the embedding
$T'_H\hra M_H\overset{\iota_M}{\lra}M\hra G$ is $\C{E}$-stably conjugate to the inclusion
$T'_H\hra M_H\hra H$. To show the latter assertion, we can replace $T'_H$ by its $M_H$-conjugate $T_H$,
in which case the assertion follows from the definition of $\iota_M$.

{\bf Step 6. Existence: second reduction.}
To show the existence of $\Phi^H$ from Step 4, we will show first the existence of
a morphism  $\Phi^H_0:H^{\sr}\to\mcH$ such that
$\Phi^H_0|_{T^{\sr}_H}=\Phi^{T_H}|_{T^{\sr}_H}$ for each $T_H$, where we
write $T^{\sr}_H$ instead of $H^{\sr}\cap T_H$.

Note that the map $[h,t]\mapsto hth^{-1}$ induces an isomorphism
$(H\times T^{\sr}_H)/N_H(T_H)\isom H^{\sr}$. Since
$\Phi^{T_H}:T_H\to\mcT_H$ is $N_H(T_H)$-equivariant (by \form{qlog2}), the morphism
$(h,t)\mapsto \Ad(h)(\Phi(t)):H\times T^{\sr}_H\to\mcH$ descends to a morphism
 $\Phi^H_0:H^{\sr}\to\mcH$ which by \form{qlog2} extends each
$\Phi^{T_H}|_{T^{\sr}_H}$.

It now remains to show that  $\Phi^H_0$ extends to a morphism
$\Phi^H:H\to\mcH$. Indeed, such an extension necessarily satisfy that
$\Phi^H|_{T_H}$ and $\Phi^{T_H}$ have equal restrictions to $T^{\sr}_H$,
thus $\Phi^H|_{T_H}=\Phi^{T_H}$, because  $T^{\sr}_H$ is Zariski dense in $T_H$.

The existence of an extension of $\Phi_0^H$ to $H$ is obvious when $H$ is a torus,
and it was shown in Step 5 when $H$ is of semisimple rank one.
Hence we can assume that the semisimple rank of $H$ is at least two.

Moreover, since  $H$ is smooth, hence normal, it will suffice to show that
$\Phi^H_0$ extends to a morphism  $\Phi':H'\to\mcH$  for some open subset
$H'\supset H^{\sr}$ of $H$ such that  $H\sm H'\subset H$ has codimension two.

{\bf Step 7. Open subset.}
Denote by $T''_H\subset T_H$ the set of all $t\in T_H$ such that there exists at
most one root $\al:T_H\to\B{G}_m$  of $H$ such that  $\al(t)=1$. We set
$c''_H:=W_H\bs T''_H\subset W_H\bs T_H=c_H$, and
$H'':=\chi_H^{-1}(c''_H)\cap H^{reg}\subset H^{reg}$.

The complement $H\sm H''\subset H$ is of codimension two. Indeed, since
$H\sm H^{reg}\subset H$ is of codimension three, it remains to show that
$H^{reg}\sm H''\subset H^{reg}$ is of
codimension two. However $T_H\sm T''_H\subset T_H$ (and hence also $c_H\sm c''_H\subset c_H$)
is of codimension two, so the assertion follows from the fact that
$\chi_H|_{H^{reg}}:H^{reg}\to c_H$ is flat.

We claim that each $h\in H''$ lies in a Levi subgroup $M_H\subset H$ of semisimple rank
one. Indeed, let $h=su$ be the Jordan decomposition of $h$, and set $M_H:=H_s^0$.
Since $h\in H''$, we get that $\chi_H(s)=\chi_H(h)\in c''_H$, hence $M_H\subset H$ is a Levi subgroup of
$H$ of semisimple rank $\leq 1$. Also $h$ lies in $M_H$. Indeed,
$s\in H_s^0$ (because $s$ is semisimple, hence lies in a maximal torus), and
$u\in H_s^0$ (because $u\in  H_s$ and $u$ is unipotent).

{\bf Step 8. Completion of the proof.}
Consider the compactification $\B{P}(\mcH\oplus F)=\mcH\cup\B{P}(\mcH)$ of
$\mcH$. Since $H$ is smooth, the morphism $\Phi^H_0:H^{\sr}\to\mcH\subset \B{P}(\mcH\oplus F)$
extends (uniquely) to a morphism
$\Phi':H'\to\B{P}(\mcH\oplus F)$ for some open subset $H'\supset H^{\sr}$ of $H$ such that
$H\sm H'\subset H$ is of codimension at least two. Since the open $H''$ defined in Step 7 has a
complement of codimension two, we can replace $H'$ by $H'\cap H''$, thus assuming that $H'\subset H''$.

 By Step 6, it suffices to show that
$\Phi'(H')$ is contained in $\mcH\subset \B{P}(\mcH\oplus F)$. In other words, it suffices 
to show that for every $h\in H'\subset H''$, we have $\Phi'(h)\in\mcH$.

By Step 7, $h$ lies in some  Levi subgroup $M_H$ of $H$ of semisimple rank one.
By  Step 5, $\Phi^H_0|_{M_H\cap H^{\sr}}$ extends to a morphism
$\Phi^{M_H}:M_H\to\mcM_H\subset\mcH\subset\B{P}(\mcH\oplus F)$.
Since $M_H\cap H^{\sr}\subset M_H$ is Zariski dense, we conclude that the restrictions
$\Phi^{M_H}|_{M_H\cap H'}$ and $\Phi'|_{M_H\cap H'}$ coincide. Hence $\Phi'(h)=\Phi^{M_H}(h)$
belongs to $\mcH$.
\end{proof}

\subsection{Quasi-logarithms defined over $\mcO$.}

Recall that a quasi-logarithm map $\Phi:G\to\mcG$ is said to be defined over $\C{O}$
if it extends to a morphism
$\Phi_{\vk}:\un{G}_{\vk}\to\C{G}_{\vk}$ for some $\vk\in\mcB(G)$
(compare \cite[Lem. 1.8.7]{KV}). Recall also that $c_G$ and $c_{\mcG}$ have natural $\C{O}$-structures
(see \re{cg}).

\begin{Lem} \label{L:qlo}
A quasi-logarithm $\Phi:G\to\C{G}$ is defined over $\mcO$ if and only if the induced morphism
$[\Phi]:c_G\to c_{\C{G}}$ is defined over $\mcO$.
\end{Lem}
\begin{proof}
For the proof we can extend scalars to $F^{nr}$, choose a hyperspecial $\vk\in\mcB(G)$, and
a maximal $F^{nr}$-split torus $T\subset G$ such that $T(\mcO^{nr})\subset G_{\vk}(\mcO^{nr})$.
Then we have a commutative diagram
\begin{equation} \label{Eq:qlog3}
\CD
        T  @>\Phi|_T>>                    \mcT\\
        @V{\chi_G}VV                        @V{\chi_{\mcG}}VV\\
        c_G  @>[\Phi]>>                    c_{\mcG},
\endCD
\end{equation}
whose vertical arrows are finite morphisms over $\C{O}$. It follows that $[\Phi]$ is defined over $\C{O}$ if and only if
$\Phi|_T$ is defined over $\C{O}$. Thus we have to check that $\Phi$ extends to a morphism
$\Phi_{\vk}:\un{G}_{\vk}\to\C{G}_{\vk}$ if and only if $\Phi|_T$ is defined over $\C{O}$.

As in the proof of \rl{surj} (a) we set $S:=\cup_{g\in G_{\vk}(\mcO^{nr})}gT(\mcO^{nr})g^{-1}
\subset G_{\vk}(\mcO^{nr})$. Thus we have to show that $\Phi(S)\subset \mcG_{\vk}(\mcO^{nr})$
if and only if  $\Phi(T(\mcO^{nr}))\subset \mcT(\mcO^{nr})$ (compare \rl{defo}). But this follows
from the fact that $\Phi$ is $G^{ad}$-equivariant.
\end{proof}

\begin{Cor} \label{C:qlo}
If a quasi-logarithm $\Phi:G\to\C{G}$ is defined over $\mcO$ and $H$ is unramified, then the quasi-logarithm
$\Phi^H$ is defined over $\C{O}$ as well.
\end{Cor}
\begin{proof}
Since  $\Phi:G\to\C{G}$ is defined over $\mcO$, we get that $[\Phi]:c_G\to c_{\C{G}}$ is defined over $\mcO$. Next both $\nu:c_H\to c_G$ and $\nu:c_{\mcH}\to c_{\mcG}$ are finite $\mcO$-morphisms, therefore we deduce from the commutativity of \form{qlog} that  $[\Phi^H]:c_H\to c_{\C{H}}$ is defined over $\mcO$. Hence
$\Phi^H:H\to\C{H}$ is defined over $\mcO$ by \rl{qlo}.
\end{proof}


\begin{Emp} \label{E:tu}
{\bf Notation.} Recall that an element $u\in G(F)$ is called {\em topologically unipotent}, if the
sequence $\{u^{p^n}\}_n$ converges to $1$. We denote by $G(F)_{tu}\subset G(F)$ the set of
topologically unipotent elements of $G(F)$.
\end{Emp}

\begin{Rem} \label{R:tu}
When $p$ does not divide the order of $Z(G^{sc})$, our notion of $G(F)_{tu}$ coincides with that
of \cite[1.8.14]{KV} (use \rl{tu} (b) below). In particular, this happens when $p$ does not divide the
order of $W_G$.
\end{Rem}

\begin{Prop} \label{P:qlogst}
Assume that a quasi-logarithm $\Phi:G\to\C{G}$ is defined over $\mcO$.

(a) For each maximal torus $T\subset G$, each root $\al\in X^*(T)$ of $G$ (over $\ov{F}$),
and each $t\in T(F)\cap G^{\sr}(F)_{tu} $, we have
$\frac{d\al(\Phi(t))}{\al(t)-1}\in 1+\frak{m}_{\ov{F}}$.

(b) Assume that $p$ does not divide the order of $Z(G^{sc})$. Then  for each
$\C{E}$-stably conjugate elements $\gm_H\in H^{G-\sr}(F)_{tu}$ and $\gm\in G^{\sr}(F)_{tu}$,
we have  $\Phi^H(\gm_H)\in \mcH^{G-\sr}(F)$ and  $\Phi(\gm)\in \mcG^{\sr}(F)$. Moreover,
elements $\Phi^H(\gm_H)$ and $\Phi(\gm)$ are $\C{E}$-stably conjugate.

(c) If, in addition, $p$ does not divide the order of $W_G$, then we have an equality
$\Dt(\gm_H,\gm)=\Dt(\Phi^H(\gm_H),\Phi(\gm))$.
\end{Prop}

\begin{proof}
(a) Extending scalars, we can assume that $T$ is split.
To show the assertion, we will show that the rational function
$f(x):=d\al\circ\Phi(x)/(\al(x)-1)$ on $T$ belongs to $\mcO[T]$ and
satisfies $f(1)=1$. If so, then $f(t)\in  1+\frak{m}$.
Indeed, since $t\in T(\mcO)$ and has reduction $\ov{t}=1$, we get that
$f(t)\in\mcO$ and that $\ov{f(t)}=\ov{f}(\ov{t})=\ov{f}(1)=\ov{f(1)}=1$.

We set $S:=(\Ker\al)^0$ and $M:=G_S$.
Since the restriction $\Phi|_M$ is a quasi-logarithm for $M$ (see \rr{qlog} (b)),
we can replace $G$ by $M$, thus assuming that $G$ is of semisimple rank one and that
$\Ker\al=Z(G)$.

Consider regular functions $g:=d\al\circ\Phi|_T$ and $h(t):=\al(t)-1$ on $T$.
Note that $h$ vanishes on
$\Ker\al=Z(G)\subset T$ (since $\Phi(Z(G))\subset Z(\mcG)\in\Ker d\al$) and that
$\Ker\al\subset T$ is a reduced subscheme (since the characteristic of $F$ is different from
two).  Hence the rational function $f=h/g$ is regular. Moreover, since  $d\Phi_1=\Id$, we get
$dh|_1=dg|_1(\neq 0)$, hence  $f(1)=1$.

Finally, since  $g,h\in \C{O}[T]$, $f=g/h\in F[T]$ and $\ov{h}\in\fq[\ov{T}]$ is not identically zero,
we get that  $f\in\mcO[T]$.

(b) Recall that $\Phi$ induces a $G^{ad}$-equivariant homeomorphism between $G(F)_{tu}$
and the set of topologically nilpotent elements of
$\mcG(F)$ (see \cite[Prop 1.8.16]{KV} and \rr{tu}). In particular, for every $\gm\in G^{\sr}(F)_{tu}$
we get that $G_{\Phi(\gm)}=G_{\gm}$ is a maximal torus, hence
$\Phi(\gm)\in\mcG^{\sr}(F)$.

If $\gm_H\in H^{G-\sr}(F)$ is
$\C{E}$-stably conjugate to $\gm$, then $\nu\circ\chi_{H}(\gm_H)=\chi_{G}(\gm)$. On the other hand,
$\chi_{\mcG}(\Phi(\gm))=[\Phi]\circ\chi_{G}(\gm)$ by the definition of $[\Phi]$, while
\[
\nu\circ \chi_{\mcH}(\Phi^H(\gm_H))=\nu\circ[\Phi^H]\circ\chi_{H}(\gm_H)=
[\Phi]\circ\nu\circ\chi_{H}(\gm_H)
\]
by the definition of $[\Phi^H]$ and the commutativity of \form{qlog}. This shows that
$\chi_{\mcG}(\Phi(\gm))=\nu\circ \chi_{\mcH}(\Phi^H(\gm_H))$, that is,
$\Phi^H(\gm_H)$ and $\Phi(\gm)$ are $\C{E}$-stably conjugate.

(c) will be proven in \re{qlogtr} below.
\end{proof}

\section{Reduction formula.}

\subsection{Centralizers of semisimple elements.}

\begin{Emp} \label{E:ssel}
{\bf Stable conjugacy of semisimple elements.}
(a) In the notation of \re{stconj}, two semisimple elements
$s\in G(F)$ and
 $s'\in G'(F)$ are called {\em stably conjugate}, if there exists $g\in G'(F^{sep})$
such that $g\varphi(s)g^{-1}=s'$ and the map $h\mapsto g\varphi(h)g^{-1}$ is an inner twisting
$\varphi_{s}:G_s^0\to {G'}^0_{s'}$. In this case we say that an inner twisting $\varphi_{s}$
is {\em compatible} with $\varphi$ (compare \cite[$\S$3]{Ko1}).

(b) Note that  semisimple elements $s\in G(F)$ and $s'\in G'(F)$ are stably conjugate, if there
exist stably conjugate embeddings of maximal tori $\fa:T\hra G$ and $\fa':T\hra G'$ and an element 
$t\in T(F)$ such that $s=\fa(t)$ and $s'=\fa'(t)$. Indeed, any  $g\in G'(F^{sep})$ such that
$g\varphi(\fa)g^{-1}=\fa'$ satisfies the property of (a).

 In particular, every semisimple $s\in G(F)$ has a
stable conjugate $s^*\in G^*(F)$ (by \re{stconj} (c)).

(c) Every semisimple $s\in G(F)$ has a stable conjugate $s^*\in G^*(F)$ such that
$(G^*)^0_{s^*}$ is quasi-split (by (b) and \cite[Lem 3.3]{Ko1}). In this case, $(G^*)^0_{s^*}$
is a quasi-split inner form of $G_s^0$. In such a situation we say that an element $s^*$ is
{\em quasi-split}.

\end{Emp}
\begin{Emp} \label{E:endcent}
{\bf Endoscopic triples for centralizers.}
(a) In the notation of \re{estconj}, two semisimple elements $s_H\in H(F)$ and
$s\in G(F)$ are called {\em $\C{E}$-stably conjugate}, if there exists a stable conjugate
$s^*\in G^*(F)$ of $s$, a pair of $\C{E}$-stably conjugate embedding of maximal tori
$\fa_H:T\hra H$ and $\fa^*:T\hra G^*$, and an element $t\in T(F)$ such that $s_H=\fa_H(t)$ and $s^*=\fa^*(t)$
(compare \cite[1.2]{LS2}).

(b) In the situation of (a), one can form an endoscopic triple
$\C{E}_s=((H_{s_H}^0)^*,\ka_{s},\eta_{s})$ for $G_s^0$ (depending on the inner twisting
$\varphi_s:G_s^0\to (G^*)^0_{s^*}$), where

$\bullet$  $\ka_s\in Z(\wh{H^0_{s_H}})^{\Gm}$ is the image of $\ka_{\fa_H}\in \wh{T}^{\Gm}$ under the
embedding $\wh{\fa_H}:\wh{T}\hra\wh{H^0_{s_H}}$ (defined up to a conjugacy), dual to $\fa_H$;

$\bullet$ $\eta_{s}$ is an embedding
$\wh{H^0_{s_H}}\hra \wh{G^0_s}$  such that the composition
$\wh{T}\overset{\wh{\fa_H}}{\lra}\wh{H^0_{s_H}}\overset{\eta_s}{\lra}\wh{G^0_s}
\overset{\wh{\varphi_s}^{-1}}{\lra}\wh{(G^*)^0_{s^*}}$ is conjugate to
$\wh{\fa}^*:\wh{T}\hra \wh{(G^*)^0_{s^*}}$.

(c) Assume that $s_H$ is quasi-split. Then embeddings of maximal tori
$\fa_H:T\hra H_{s_H}^0$ and $\fa:T\hra G_s^0$ are $\C{E}_s$-stably conjugate if and only
if $\fa_H(\fa^{-1}(s))=s_H$ and the corresponding embeddings of maximal tori
$\fa_H:T\hra H_{s_H}^0\hra H$ and $\fa:T\hra G_s^0\hra G$ are $\C{E}$-stably conjugate
(compare \cite[1.4]{LS2}).

(d) Assume that in the situation of (a) there exists an embedding of a maximal torus $\fa:T\hra G$
such that $\fa(t)=s$. Then $s$ is an $\C{E}$-stable conjugate of $s_H$ (by \re{ssel} (b)), and we can
take $\eta_s$ in (b) such that  the composition
$\wh{T}\overset{\wh{\fa_H}}{\lra}\wh{H^0_{s_H}}\overset{\eta_s}{\lra}\wh{G^0_s}$ is conjugate to
$\wh{\fa}:\wh{T}\hra \wh{G^0_{s}}$. In particular, the equivalence class of an
endoscopic triple $\C{E}_s$ from (b) depends only on $\C{E}$, $\fa_H$ and $\fa$.
\end{Emp}

\begin{Rem}
The existence of $\eta_{s}$ in \re{endcent} (b) follows from the fact that embedding $\eta_0$ identifies the set of roots
$R(\wh{H^0_{s_H}},\wh{T})=\{\al\in R(\wh{H},\wh{T})\,|\,\wh{\al}(t)=1\}$ with a subset
of $R(\wh{G^0_{s}},\wh{T})=\{\al\in R(\wh{G},\wh{T})\,|\,\wh{\al}(t)=1\}$.
\end{Rem}

\begin{Lem} \label{L:sscent}
Let $s\in G(F)$ be a semisimple element, and let $u\in G_s(F)$ be such that $\gm:=su\in G^{\sr}(F)$.
Then

(a) $u\in (G^0_s)^{\sr}(F)$ and $G_u\cap G_s=G_{\gm}=G_u\cap G^0_s$.

(b) Assume that $s'\in G'(F)$ is a stable conjugate of $s$, and
$u'\in (G'^0_{s'})^{\sr}(F)$ is a stable conjugate of $u$ (for some inner twisting
$G_s^0\to {G'}^0_{s'}$ compatible with $\varphi$). Then $\gm':=s'u'$ is a stable conjugate of $\gm$.

(c) Assume that $s_H\in H(F)$ is a quasi-split
$\C{E}$-stable conjugate of $s$, and
$u_H\in (H^0_{s_H})^{\sr}(F)$ is an $\C{E}_s$-stable conjugate of $u$ (for some endoscopic triple
$\C{E}_s$ as in \re{endcent} (b)). Then $\gm_H:=s_H u_H$ is an $\C{E}$- stable conjugate of $\gm$.
\end{Lem}

\begin{proof}
(a) Since $G_u\cap G_s\subset G_{\gm}$ and $G_{\gm}\subset G$ is a maximal torus of dimension
$\rk G=\rk G^0_s$, we get that
$G_u\cap G_s\subset G_s^0$ is a maximal torus. Hence $u\in (G^0_s)^{\sr}(F)$, and
$G_u\cap G_s=G_{\gm}=G_u\cap G^0_s$.

(b) By definition, there exists  $g\in G'(F^{sep})$ such that $g\varphi(s)g^{-1}=s'$ and
$g\varphi(u)g^{-1}=u'$. Hence  $g\varphi(\gm)g^{-1}=\gm'$.

(c) By our assumptions, there exists an embedding $\iota_H:G_{\gm}\hra H^0_{s_H}$  which is
$\mcE_s$-stably conjugate to the inclusion $G_{\gm}\hra G^0_s$ such that $\iota_H(u)=u_H$.
Then  $\iota_H(s)=s_H$, and the composition $\iota_H:G_{\gm}\hra H^0_{s_H}\hra H$ is $\mcE$-stably
conjugate to the inclusion $G_{\gm}\hra G^0_s\hra G$ (by \re{endcent} (c)).
Hence $\gm_H=\iota_H(\gm)$ is $\C{E}$-stably conjugate to $\gm$.
\end{proof}
\subsection{ Topological Jordan decomposition}
\begin{Emp} \label{E:tjd}
{\bf Notation.}
An element $\gm\in G(F)$ is called {\em compact}, if the closure of the cyclic subgroup
$\ov{\lan\gm\ran}\subset G(F)$ is compact. Recall that for every compact $\gm\in G(F)$
there exists unique decomposition $\gm =su=us$ such that $s=\gm_s$ is of finite order prime to $p$,
and $u=\gm_u$ is topologically unipotent (see, for example, \cite{Ka2}, \cite{Wa0} or \cite{Ha}).
Moreover, both $s$ and $u$ lie in $\ov{\lan\gm\ran}$. In particular,  $s, u\in G_{\vk}$ for each $\gm\in  G_{\vk}$.

This decomposition is called the {\em topological Jordan decomposition} and will be sometimes abbreviated as $TJD$.
\end{Emp}

\begin{Lem} \label{L:cent}
Let $s\in G_{\vk}$ be a semisimple element of finite order prime to $p$. Then

(a) the connected  centralizer $G_s^0$ splits over $F^{nr}$, $\vk$ lies in $\C{B}(G_s^0)=\C{B}(G)^s$,
and the parahoric subgroup $(G_s^0)_{\vk}$ equal $G_{\vk}\cap G_s^0(F)$.

(b) If $\vk\in \C{B}(G)$ is a hyperspecial vertex, then $G_s^0$ is quasi-split, and $\vk$ is a
hyperspecial vertex of $\C{B}(G_s^0)$. Moreover, every stable conjugate $s'\in G_{\vk}$ of $s$
is  $G_{\vk}$-conjugate to $s$.
\end{Lem}

\begin{proof}

(a) was proven in \cite[Lem 2.3.8]{KV}, (b) was proven in \cite[7.1]{Ko3}.
\end{proof}

\begin{Lem} \label{L:tjd}
Let $\gm\in G^{sr}(F)$ be a compact element with TJD  $\gm=su$.

(a) We have $u\in (G^0_s)^{\sr}(F)$, and $G_u\cap G_s=G_{\gm}=G_u\cap G^0_s$.

(b) If $\gm'\in G'^{\sr}(F)$ (resp.  $\gm_H\in H^{\sr}(F)$) is a stable (resp. an $\C{E}$-stable)
conjugate of $\gm$, then $\gm'$ (resp. $\gm_H$) is compact.

(c) Let $\gm'\in G'(F)$ be a stable conjugate of $\gm$ with TJD  $\gm'=s'u'$.
Then $s'$ is a stable conjugate of $s$, $\varphi$ gives rise to an
inner twisting $G^0_s\to G^0_{s'}$ (defined up to an equivalence), and
$u'\in (G'^0_{s'})^{\sr}(F)$ is a stable conjugate of $u$.

(d) Let $\gm_H\in H(F)$ be an $\C{E}$-stable conjugate of $\gm$ with TJD $\gm_H=s_Hu_H$. Then
$s_H$ is an $\C{E}$-stable conjugate of $s$, $\C{E}$ gives rise to an endoscopic triple
$\C{E}_s$ for $G^0_s$ (defined up to an equivalence), and
$u_H\in (H^0_{s_H})^{\sr}(F)$ is an $\C{E}_s$-stable conjugate of
$u\in (G^0_{s})^{\sr}(F)$.

(e) Let  $\vk^*\in\mcB(G^*)$ be a hyperspecial vertex. Then there exists an embedding
$\iota^*:G_{\gm}\hra G^*$ stably conjugate to the inclusion $G_{\gm}\hra G$ such that
$\iota(s)\in G^*_{\vk^*}$.
\end{Lem}

\begin{proof}
(a) follows from \rl{sscent} (a).

(b) By definition, there exists an isomorphism $\iota:G_{\gm}\isom G'_{\gm'}$
(resp. $\iota_H:G_{\gm}\isom H_{\gm_H}$) such that  $\iota(\gm)=\gm'$ (resp.  $\iota_H(\gm)=\gm_H$).
Hence $\gm'$ (resp. $\gm_H$) is compact.

(c) Let $g\in G'(F^{sep})$ be such that $\gm'=g\varphi(\gm)g^{-1}$. Then by the uniqueness of
the TJD, we have $s'=g\varphi(s)g^{-1}$ and  $u'=g\varphi(u)g^{-1}$. Hence the map
$h\mapsto g\varphi(h)g^{-1}$ defines an inner twisting $\varphi_s:G_s^0\to G'^0_{s'}$ and maps
$u$ to $u'$. Moreover, the equivalence class of $\varphi_s$ is independent of $g$.

(d) By our assumption, there exists an embedding $\iota_H:G_{\gm}\hra H$ which is
$\C{E}$-stably conjugate to the inclusion $G_{\gm}\hra G$ and $\iota_H(\gm)=\gm_H$. By the uniqueness of the TJD, we get $\iota_H(s)=s_H$ and $\iota_H(u)=u_H$. Now the assertion follows from remarks \re{endcent} (c),(d).

(e) By \rl{cent} (a), $G_s^0$ splits over $F^{nr}$. Thus there exists an embedding of
a maximal torus $\fa:T\hra G_s^0$ such that $T$ is unramified (use for example \rl{maxtor} (b)).
By \rl{maxtor} (c), there exists an embedding $\fa^*:T\hra G^*$ stably conjugate to
$\fa:T\hra G^0_s\hra G$ such that $\vk^*\in \fa^*(\mcB(T))$.
Set $s^*:=\fa^*(\fa^{-1}(s))\in G^*_{\vk^*}$. Then $(G^*_{s^*})^0$ is a quasi-split inner form
of $G^0_{s}$ by \rl{cent} (b), hence there exists an embedding $\iota^*:G_{\gm}\hra  (G^*_{s^*})^0$
which is stably conjugate to the inclusion $G_{\gm}\hra G^0_{s}$. Then the composition
$\iota^*:G_{\gm}\hra (G^*_{s^*})^0\hra G^*$ is stably conjugate to the inclusion $G_{\gm}\hra G$,
and $\iota^*(s)=s^*$.
\end{proof}

\begin{Emp} \label{E:eetc}
{\bf Extended endoscopic triples for $G^0_s$.}
Let $\vk$ be an element of $\C{B}(G)$, $\ov{\fa}:\ov{T}\hra L_{\vk}$ an embedding of
a maximal torus, $s\in G_{\vk}$ a semisimple element of finite order prime to $p$,
$\ov{\fb}:\ov{T}\hra (L_{\vk})^0_{\ov{s}}$ an embedding of a maximal torus such that  the embedding
$\ov{\fb}:\ov{T}\hra (L_{\vk})^0_{\ov{s}}\hra L_{\vk}$ is conjugate to $\ov{\fa}$. As in \rn{stconj} (b), 
we fix a hyperspecial vertex $\vk^*\in\C{B}(G^*)$.

(a) Let $\fb:T\hra G_s^0$ be an embedding of a maximal torus which corresponds to $\ov{\fb}$ by
\rl{maxtor} (b), and let $\fb^*:T\hra G^*$ be an embedding which is stably conjugate to the composition
$\fb:T\hra G_s^0\hra G$ such that $\vk^*\in \fb^*(\C{B}(T))$ (use  \rl{maxtor} (c)). Set
$s^*:=\fb^*(\fb^{-1}(s))\in G^*(F)$. Then $s^*\in G^*_{\vk^*}$ is a stable conjugate of $s$,
 hence $(G^*_{s^*})^0$ is a quasi-split inner form of $G_s^0$, and  $\vk^*$ is a hyperspecial
vertex of $\C{B}((G^*_{s^*})^0\subset \C{B}(G)$ (use \re{ssel} (b) and \rl{cent} (b)).

(b) Assume that we are in the situation of \re{unr}. Let $s_H\in H_{\vk^H}$ be a stable conjugate
of $s\in G_{\vk}$, and fix a corresponding endoscopic triple $\C{E}_s$ for $G_s^0$ with
unramified endoscopic group $H_{s_H}$ (see \re{endcent} (b) and use \rl{cent} (b)).

When the endoscopic triple $\C{E}_s$ is consistent, we fix an extended endoscopic
triple $\wt{\C{E}}_s=(\C{E}_s,\eta_s,Spl_s,\fa_s^{(3)})$  for $G_s^0$ satisfying assumptions of \re{unr},
and denote by  $\phi^{H_{s_H}^0}_{\ov{\fb},\theta}$ the corresponding linear combination of
Deligne--Lusztig functions.
We also note that the embedding $\fb^*:T\hra (G^*_{s^*})^0$,
constructed in (a), coincides with the corresponding embedding from \re{unr} (c).

\end{Emp}

\begin{Rem} \label{R:ass}
If the group $G$ and the embedding $\ov{\fa}$ satisfy the assumptions of \rt{main}, then the group  $G^0_s$
and the embedding $\ov{\fb}$ as in \re{eetc} satisfy them as well. Indeed,  since $W_{G^0_s}\subset W_{G}$, we get
$p$ does not divide $|W_{G^0_s}|$. Also if $d\ov{\fa}^*(\ov{t})\in\ov{\mcG}^{\sr}(\fq)$, then
$d\ov{\fb}^*(\ov{t})\in\ov{\mcG}^{\sr}(\fq)\cap \ov{\mcG}_s(\fq)\subset\ov{\mcG}^{\sr}_s(\fq)$.
\end{Rem}

The following assertion will be proven in \re{pftrfac}.
\begin{Prop} \label{P:trfac}
In the notation of \re{eetc} (b), let $\gm_H\in H^{G-\sr}(F)$ and $\gm\in G^{\sr}(F)$ be two compact $\C{E}$-stable
conjugate elements with topological Jordan decompositions $\gm_H=s_H u_H$ and $\gm=s u$. Then
$\Dt^{\wt{\C{E}}}(\gm_H,\gm)=\inv_{\C{E}}(\fa_0^{(3)},\fa_{s}^{(3)})\Dt^{\wt{\C{E}}_s}(u_H,u)$.
\end{Prop}

\subsection{Application to orbital integrals}

\begin{Lem} \label{L:redkap}
For each $\gm\in G^{\sr}(F)\cap G_{\vk}$ with the topological Jordan decomposition $\gm=s u$ and every
$\ka\in\wh{G_{\gm}}^{\Gm}$ we have an equality
\begin{equation} \label{Eq:redka}
O^{\ka}_{\gm}(\phi^G_{\ov{\fa},\theta})=\sum_{s'\in G_{\vk}}\lan\inv(\gm,\gm'),\ka\ran
\sum_{\ov{\fb}:\ov{T}\hra (L_{\vk})_{\ov{s}'}^0}
\theta(\ov{\fb}^{-1}(\ov{s}'))O^{\ka}_{u'}(\phi^{G_{s'}^0}_{\ov{\fb},\theta}),
\end{equation}
where

$\bullet$ $s'$ runs over a set of representatives of the $G_{\vk}$-equivalence classes, stably conjugate
to $s$, such that there exists a stable conjugate $\gm'\in G_{\vk}$ of $\gm$ with the
topological Jordan decomposition $\gm'=s'u'$;

$\bullet$ $\ov{\fb}$ runs over a set of representatives of the conjugacy classes of embeddings
$\ov{T}\hra (L_{\vk})_{\ov{s}'}^0$ such that the composition
$\ov{\fb}:\ov{T}\hra (L_{\vk})_{\ov{s}'}^0\hra L_{\vk}$ is $L_{\vk}(\fq)$-conjugate to $\ov{\fa}$.
\end{Lem}

\begin{Rem}
It follows from the proof that the  contribution of each $s'$ to \form{redka} is independent of a choice of
$\gm'$.
\end{Rem}

\begin{proof}
Since function $\phi^G_{\ov{\fa},\theta}$ is supported on $G_{\vk}$ and is
$G_{\vk}$-conjugation invariant, the left hand side of \form{redka} equals
\begin{equation} \label{Eq:redka1}
\sum_{\gm'\in G_{\vk}}\lan\inv(\gm,\gm'),\ka\ran \frac{1}{dg_{\gm'}(G_{\vk}\cap G_{\gm'}(F))}
\phi^G_{\ov{\fa},\theta}(\gm'),
\end{equation}
where $\gm'$ runs over a set of representatives of the $G_{\vk}$-conjugacy classes which are
stably conjugate to $\gm$.

Each such $\gm'$ is a compact element (by \rl{tjd} (b)), and $\gm'_s$
is stably conjugate to $s$ (by \rl{tjd} (c)). In other words, the summation in \form{redka1} can be
written in the form $\sum_{s'}\sum_{\gm'}$, where $s'$ is as in \form{redka}, and
$\gm'\in G_{\vk}$ runs over a set of representatives of the $G_{\vk}$-conjugacy
classes which are stably conjugate to $\gm$ such that $\gm'_s=s'$.

Fix $s'$ as in \form{redka}, and choose a stable conjugate $\gm'\in G_{\vk}$ of $\gm$ with TJD $\gm'=s'u'$.
Then the contribution of $s'$ to \form{redka1} equals
\begin{equation} \label{Eq:redka1'}
\lan\inv(\gm,\gm'),\ka\ran\sum_{u''}\lan\inv(s'u',s'u''),\ka\ran
\frac{1}{dg_{s'u''}(G_{\vk}\cap G_{s'u''}(F))}
\phi^G_{\ov{\fa},\theta}(s'u''),
\end{equation}
where $u''\in (G_{s'}^0)_{\vk}\subset G_{s'}^0(F)$ runs over a set of representatives of the
$(G_{\vk}\cap G_{s'}(F))$-conjugacy classes which are stably conjugate to $u'\in G_{s'}^0(F)$ (use
\rl{tjd} (c) and \rl{sscent} (b)).

Next we observe that if $u'''$ is a stably conjugate of $u''$ and $u'''=gu''g^{-1}$ for some
$g\in G_{s'}(F)$, then $g\in  G_{s'}^0(F)$, Indeed, by the assumption, there exists
$h\in  G_{s'}^0(F^{sep})$ such that $u'''=hu''h^{-1}=gu''g^{-1}$. Then $h^{-1}g$ belongs to
$G_{s'}\cap G_{u''}\subset G_{s'}^0$ (by \rl{tjd} (a)), thus
$h\in  G_{s'}^0(F^{sep})\cap  G_{s'}(F)=G_{s'}^0(F)$.

Using  \rl{cent} (a) we can therefore assume that
in \form{redka1'} $u''$ runs over a set of representatives of the
$(G_{s'}^0)_{\vk}$-conjugacy classes,  which are stably conjugate to $u'\in G_{s'}^0(F)$.

On the other hand, $O^{\ka}_{u'}(\phi^{G_{s'}^0}_{\ov{\fb},\theta})$ from the
right hand side of \form{redka} equals
\begin{equation} \label{Eq:redka2}
\sum_{u''\in (G_{s'}^0)_{\vk}}\lan\inv(u',u''),\ka\ran\sum_{\ov{\fb}:\ov{T}\hra (L_{\vk})_{\ov{s}'}^0}
\frac{1}{dg_{u''}((G_s^0)_{\vk}\cap G_{u''}(F))}
\phi^{G_s^0}_{\ov{\fb},\theta}(u''),
\end{equation}
where $u''$  runs over a set of representatives of the $(G_{s'}^0)_{\vk}$-conjugacy classes
which are $G_{s'}^0$-stably conjugate to $u'$.

For such a $u''$, we have $\lan\inv(u',u''),\ka\ran=\lan\inv(s'u',s'u''),\ka\ran$ and also
\[
(G_{s'}^0)_{\vk}\cap G_{u''}(F)=G_{\vk}\cap G_{s'}^0(F)\cap G_{u''}(F)=G_{\vk}\cap G_{s'u''}(F)
\] (by \rl{cent} (a) and \rl{tjd} (a)). Thus to finish the proof it suffices to show that
for every $s'$ and $u''$ as above, we have an equality
\begin{equation} \label{Eq:redka4}
\phi^G_{\ov{\fa},\theta}(s'u'')=\sum_{\ov{\fb}}\theta(\ov{\fb}^{-1}(\ov{s}'))
\phi^{G_s^0}_{\ov{\fb},\theta}(u'').
\end{equation}
Finally, formula  \form{redka4} is equivalent to the formula
$\ov{\phi}_{\ov{\fa},\theta}(\ov{s}'\ov{u}'')=\sum_{\ov{\fb}}
\theta(\ov{\fb}^{-1}(\ov{s}'))\ov{\phi}_{\ov{\fb},\theta}(\ov{u}'')$,
proven by Deligne--Lusztig (\cite[Thm 4.2]{DL}).
\end{proof}


\begin{Cor} \label{C:trfac}
In the notation of \re{eetc} (b), let  $\gm_H$ be an element of $H^{G-\sr}(F)\cap H_{\vk^H}$. Then
we have equalities
\begin{equation} \label{Eq:red1}
O^{\wt{\C{E}}}_{\gm_H}(\phi^G_{\ov{\fa},\theta})=\sum_{s'\in G_{\vk}}\inv_{\C{E}}(\fa_0^{(3)},\fa_{s'}^{(3)})
\sum_{\ov{\fb}:\ov{T}\hra
(L_{\vk})_{\ov{s}'}^0}\theta(\ov{\fb}^{-1}(\ov{s}'))O^{\wt{\C{E}}_{s'}}_{u_H}(\phi^{H_{s_H}^0}_{\ov{\fb},\theta})
\end{equation}
and
\begin{equation} \label{Eq:red2}
O^{st}_{\gm_H}(\phi^H_{\ov{\fa},\theta})=\sum_{s'\in G_{\vk}}\inv_{\C{E}}(\fa_0^{(3)},\fa_{s'}^{(3)})
\sum_{\ov{\fb}:\ov{T}\hra
(L_{\vk})_{\ov{s}'}^0}\theta(\ov{\fb}^{-1}(\ov{s}'))O^{st}_{u_H}(\phi^{H_{s_H}^0}_{\ov{\fb},\theta}),
\end{equation}
where

$\bullet$ $s'$  runs over a set of representatives of the $G_{\vk}$-conjugacy classes which are
$\C{E}$-stably conjugate to $s_H$;

$\bullet$ $\ov{\fb}$ runs over a set of representatives of the set of conjugacy classes of embeddings
$\ov{T}\hra (L_{\vk})_{\ov{s}'}^0$ such that the composition
$\ov{\fb}:\ov{T}\hra (L_{\vk})_{\ov{s}'}^0\hra L_{\vk}$ is $L_{\vk}(\fq)$-conjugate to $\ov{\fa}$;

$\bullet$ extended endoscopic triples $\wt{\C{E}}_{s'}$ are chosen in \re{eetc} (b).
\end{Cor}

\begin{proof}
{\bf Proof of equality \form{red1}.}
%
First we claim that both sides of \form{red1} vanish, if $\gm_H$ is not $\C{E}$-stably conjugate to an
element of $G_{\vk}$. The vanishing of the left hand side follows from the fact that
$\phi^G_{\ov{\fa},\theta}$ is supported on $G_{\vk}$. Assume that a contribution of some $s'$ to
the right hand side of \form{red1} is non-zero. Then there exists an $\C{E}_{s'}$-stably conjugate
$u'\in (G_{s'})_{\vk}$ of $u_H$. Hence
$\gm'=s'u'\in G_{\vk}$ is an $\C{E}$-stable conjugate of $\gm_H$ (by \rl{sscent} (c)), contradicting our assumption.

Thus we can assume that  $\gm_H$ is $\C{E}$-stably conjugate to an element of $\gm\in G_{\vk}$ with TJD $\gm=su$.
Notice that an element $\gm'\in G_{\vk}$ is stably conjugate to $\gm$ if and only if it is $\C{E}$-stably conjugate
to $\gm_H$. Hence, by the shown above, only $s'$'s which appear in  \form{redka} might have a non-trivial contribution
to the right hand side of  \form{red1}.

On the other hand, it follows from \rl{redkap} and  \re{transf} (b) that the left hand side of
\form{red1} equals
\[
\sum_{s'}\Dt^{\wt{\C{E}}}(\gm_H,\gm)\lan\inv(\gm,\gm'),\ka\ran\sum_{\ov{\fb}}
\theta(\ov{\fb}^{-1}(\ov{s}'))O^{\ka}_{u'}(\phi^{G_{s'}^0}_{\ov{\fb},\theta}),
\]
where $\ka=\ka_{\gm,\gm_H}$. Thus we have to show the equality
\[
\Dt^{\wt{\C{E}}}(\gm_H,\gm)\lan\inv(\gm,\gm'),\ka\ran O^{\ka}_{u'}(\phi^{G_{s'}^0}_{\ov{\fb},\theta})=
\inv_{\C{E}}(\fa_0^{(3)},\fa_{s'}^{(3)})O^{\wt{\C{E}}_{s'}}_{u_H}(\phi^{G_{s'}^0}_{\ov{\fb},\theta}),
\]
which in turn is equivalent to the equality
\[
\Dt^{\wt{\C{E}}}(\gm_H,\gm)\lan\inv(\gm,\gm'),\ka\ran=\inv_{\C{E}}(\fa_0^{(3)},\fa_{s'}^{(3)})\Dt^{\wt{\C{E}}_{s'}}(u_H,u').
\]
The latter follows from a combination of \re{proptf} (b) and \rp{trfac}.

{\bf Proof of equality \form{red2}.}
By the definition of $\phi^H_{\ov{\fa},\theta}$, the left hand side of \form{red2} equals
$\sum_{\ov{\fa}_H}\inv_{\C{E}}(\fa_0^{(3)},\fa^{(3)}) O^{st}_{\gm_H}(\phi^H_{\ov{\fa}_H,\theta})$, where $\ov{\fa}_H$ runs over a set
of representatives of the set of conjugacy classes of
embeddings  $\ov{\fa}_H:\ov{T}\hra \ov{H}$, which are $\C{E}$-stable conjugate to $\ov{\fa}$.
Using \rl{cent} (b) and \rl{redkap}, it is equal to
\begin{equation} \label{Eq:red3}
\sum_{\ov{\fb}_H}\inv_{\C{E}}(\fa_0^{(3)},\fa^{(3)})\theta(\ov{\fb}_H^{-1}(\ov{s}_H))
O^{st}_{u_H}(\phi^{H_{s_H}^0}_{\ov{\fb}_H,\theta}),
\end{equation}
where $\ov{\fb}_H$  runs over a set of representatives of the set of conjugacy classes of
embeddings $\ov{T}\hra \ov{H}_{\ov{s}_H}^0$ such that the composition
$\ov{\fb}_H:\ov{T}\hra \ov{H}_{\ov{s}_H}^0\hra \ov{H}$ is
$\C{E}$-stably conjugate to $\ov{\fa}$.

On the other hand, the right hand side of \form{red2} equals
\begin{equation} \label{Eq:red4}
\sum_{s'}\inv_{\C{E}}(\fa_0^{(3)},\fa_{s'}^{(3)})
\sum_{\ov{\fb}}\sum_{\ov{\fb}_H}\theta(\ov{\fb}^{-1}(\ov{s}'))\inv_{\C{E}_{s'}}(\fa_{s'}^{(3)},\fb^{(3)})
O^{st}_{u_H}(\phi^{H_{s_H}^0}_{\ov{\fb}_H,\theta}),
\end{equation}
where  $\ov{\fb}_H$ runs over conjugacy classes of embeddings
$\ov{T}\hra \ov{H}_{\ov{s}'}^0$, which are $\C{E}_{s'}$-stably conjugate to $\ov{\fb}$.
Since $\ov{\fb}^{-1}(\ov{s})=\ov{\fb}_H^{-1}(\ov{s}_H)$, and
\[
\inv_{\C{E}}(\fa_0^{(3)},\fa^{(3)})=\inv_{\C{E}}(\fa_0^{(3)},\fa_{s'}^{(3)})\inv_{\C{E}}(\fa_{s'}^{(3)},\fa^{(3)})=
\inv_{\C{E}}(\fa_0^{(3)},\fa_{s'}^{(3)})\inv_{\C{E}_{s'}}(\fa_{s'}^{(3)},\fb^{(3)})
\] (use \re{inv} (iv),(vi)),
it remains to show that in \form{red3} and \form{red4} the summation is taken over the same set of conjugacy classes.

By \rl{maxtor}, one gets that both these sets coincide with the set of conjugacy classes of embeddings
$\fb_H:T\hra H_{s_H}^0$ such that $\vk^H\in \fb_H(\C{B}(T))$ and the composition
$\fb_H:T\hra H_{s_H}^0\hra H$ is $\C{E}$-stably conjugate to $\fa$.
%
%
\end{proof}

\section{Proof of the Main Theorem}
\subsection{Preparations}

\begin{Lem} \label{L:tu}

(a) Let $\pi:G_0\to G$ be an isogeny of degree prime to $p$. Then $\pi$ induces a homeomorphism
$\pi(F)_{tu}:G_0(F)_{tu}\isom G(F)_{tu}$.

(b) Assume that $p$ does not divide the order of $Z(G^{sc})$. Then for each
$u\in G(F)_{tu}$ there exists $\vk\in\C{B}(G)$ such that $u\in G_{\vk}$ and the image
$\ov{u}\in L_{\vk}(\fq)$ is unipotent.
\end{Lem}

\begin{proof}
(a) Since $\pi(F)_{tu}$ open and continuous, it remains to show that it is a bijection.
First we will show that $\pi(F)_{tu}$ is injective. Assume that $u,u'\in G_0(F)_{tu}$ satisfy $\pi(u)=\pi(u')$.
Then the quotient $u^{-1}u'$ is a topologically unipotent element of $\Ker\pi$, hence $u^{-1}u'=1$, because
$\Ker\pi$ is of order prime to $p$.

Now choose any $u\in  G(F)_{tu}$. By the injectivity, it will suffice to show the   existence of
$u_0\in G_0(F^{sep})_{tu}$ such that $\pi(u_0)=u$. Since $\pi$ is \'etale over $F$, there exists
$\gm_0\in G_0(F^{sep})$ such that $\pi(\gm_0)=u$. Since $u$ is compact and
$\Ker\pi$ is finite, we get that $\gm_0$ is compact. Let $\gm_0=s_0u_0$ be the TJD of $\gm_0$. Then $\pi(u_0)=u$,
by the uniqueness of the TJD, completing the proof.

(b) Consider the canonical isogeny  $\pi:G_0:=G^{sc}\times Z(G)^0\to G$. Since $\Ker\pi$ is a subgroup of $Z(G^{sc})$, its degree is prime to $p$. Hence by (a) we reduce to the case of $G_0$. Thus we can assume that the derived
group of $G$ is simply connected. In this case the assertion is shown in \cite[Cor. 2.3.3 (b)]{KV}.
\end{proof}

\begin{Lem} \label{L:ql}
Let  $G$ be semisimple and simply connected such that $p$ does not divide
the order of $W_G$.  Then there exists a quasi-logarithm $\Phi:G\to\C{G}$ defined
over $\C{O}$, and $\C{G}$ admits an invariant pairing $\lan\cdot,\cdot\ran$ non-degenerate over $\C{O}$.
\end{Lem}

\begin{proof}
Since  $p$ does not divide the order of $W_G$, the assumptions of
\cite[Lemma 1.8.12]{KV} are satisfied, hence the assertion follows.
\end{proof}

The proof of following two lemmas will be given in \re{pfisog} and \re{pfcenter} respectively.

\begin{Lem} \label{L:isog}
Let $\pi:G_0\to G$ be a quasi-isogeny (that is, $\pi$ induces an isomorphism of adjoint groups
$\pi^{ad}:(G_0)^{ad}\isom  G^{ad}$), $\wt{\C{E}}_0$
the extended endoscopic triple for $G_0$ induced by $\wt{\C{E}}$
(compare \cite[Lem 1.3.10]{KV}), and  $\pi_H:H_0\to H$ the corresponding
quasi-isogeny of $H$.

Then for each pair of $\wt{\C{E}}_0$-stably conjugate elements $\gm_H\in H_0^{G_0-\sr}(F)$ and  
$\gm\in G_0^{\sr}(F)$, we have an equality
$\Dt^{\wt{\C{E}}_0}(\gm_H,\gm)=\Dt^{\wt{\C{E}}}(\pi_H(\gm_H),\pi(\gm))$.
\end{Lem}

\begin{Lem} \label{L:center}
In the situation of \re{unr}, let $\gm_H\in H^{G-\sr}(F)$ and  $\gm\in G^{\sr}(F)$ be $\wt{\C{E}}$-stably conjugate
elements. Then for every $z\in Z(G)(\C{O})\subset  Z(H)(\C{O})$, we have an equality
$\Dt^{\wt{\mcE}}(z\gm_H,z\gm)=\Dt^{\wt{\mcE}}(\gm_H,\gm)$.
\end{Lem}

\subsection{The topologically unipotent case}

\rt{main} asserts that we have an equality
\begin{equation} \label{Eq:mainr}
O^{\ka_{\gm,\gm_H}}_{\gm}(\phi^G_{\ov{\fa},\theta})=0
\end{equation}
for each pair of $\C{E}$-stably conjugate elements $\gm_H\in H^{G-\sr}(F)$ and $\gm\in G^{\sr}(F)$, if $H$ is ramified, and that
we have an equality
\begin{equation} \label{Eq:main}
O^{st}_{\gm_H}(\phi^H_{\ov{\fa},\theta})=O^{\wt{\mcE}}_{\gm_H}(\phi^G_{\ov{\fa},\theta}).
\end{equation}
for each $\gm_H\in H^{G-\sr}(F)$, if $H$ is unramified and the assumptions of \re{unr} are satisfied.

First we will show these equalities in the case when $\gm_H$ is topologically unipotent.

\begin{Thm} \label{T:topun}
Equalities \form{mainr} and \form{main} hold when $\gm_{H}$ is topologically unipotent.
\end{Thm}

\begin{Lem} \label{L:redtu}
It will suffice to show the assertion of \rt{topun} in the case when $G$ is semisimple and
simply connected.
\end{Lem}

\begin{proof}
Again consider the canonical isogeny $\pi:G_0:=G^{sc}\times Z(G)^0\to G$.
First we claim that \rt{topun} for $G$ follows from that for $G_0$.

Indeed, the extended endoscopic triple $\wt{\C{E}}$ for $G$ gives rise to the extended endoscopic triple
$\wt{\C{E}}_0$ for $G_0$ (compare \cite[Lem 1.3.10]{KV}). Moreover, $\wt{\C{E}}_0$ satisfies the assumptions of \re{unr}, if $\wt{\mcE}$ satisfies them.  Also $\pi$ induces an homeomorphism $\mcB(G_0)\isom \mcB(G)$, hence $\vk\in\mcB(G)$ gives rise to a point $\vk_0\in\mcB(G_0)$. Furthermore, $\pi$ gives rise to 
 an isogeny $\pi_{\vk}:L_{\vk_0}\to L_{\vk}$; embedding $\ov{\fa}:\ov{T}\hra L_{\vk}$ gives rise to an embedding $\ov{\fa}_0:\ov{T}_0\hra L_{\vk_0}$, while character $\theta:\ov{T}(\fq)\to\B{C}\m$ defines a  character $\theta_0:\ov{T}_0(\fq)\overset{\ov{\pi}}{\lra}\ov{T}(\fq)\overset{\theta}{\lra}\B{C}\m$.

Next the kernel of $\pi$ and hence the kernel of the corresponding isogeny
$\pi_H:H_0\to H$ is  a subgroup of $Z(G^{sc})$. Since the order of $Z(G^{sc})$ divides
the order of $W_G$, the orders of $\pi$ and $\pi_H$ are  prime to $p$.
Therefore the maps $\pi(F):G_0(F)_{tu}\to G(F)_{tu}$ and $\pi_H(F):H_0(F)_{tu}\to H(F)_{tu}$
are homeomorphisms (by \rl{tu} (a)).

Notice that for every $\C{E}_0$-stably conjugate elements $u_H\in H_0^{G_0-\sr}(F)_{tu}$ and
$u\in G_0^{\sr}(F)_{tu}$,
we have equalities $\phi^{G_0}_{\ov{\fa}_0,\theta_0}(u)=\phi^{G}_{\ov{\fa},\theta}(\pi(u))$,
$\phi^{H_0}_{\ov{\fa}_0,\theta_0}(u_H)=\phi^{H}_{\ov{\fa},\theta}(\pi_H(u_H))$ and 
also $\Dt^{\wt{\C{E}}_0}(u_H,u)=\Dt^{\wt{\C{E}}}(\pi_H(u_H),\pi(u))$ (by \rl{isog}).
Therefore we have equality $O^{\ka_{u,u_H}}_{u_H}(\phi^{G_0}_{\ov{\fa},\theta})=
O^{\ka_{\pi(u),\pi_H(u_H)}}_{\pi_H(u_H)}(\phi^G_{\ov{\fa},\theta})$ in the ramified case, and equalities
$O^{st}_{u_H}(\phi^{H_0}_{\ov{\fa}_0,\theta_0})=O^{st}_{\pi_H(u_H)}(\phi^H_{\ov{\fa},\theta})$, 
$O^{\wt{\mcE}_0}_{u_H}(\phi^{G_0}_{\ov{\fa},\theta})=O^{\wt{\mcE}}_{\pi_H(u_H)}(\phi^G_{\ov{\fa},\theta})$
in the unramified case.

We see that \rt{topun} for $\wt{\C{E}}$ follows from that for $\wt{\C{E}}_0$.
Hence we can assume  that $G=G^{sc}\times S$ for some torus $S$. Then using similar (but easier)
arguments and \rl{center}, we reduce the assertion to the corresponding endoscopic tuple
$\wt{\C{E}}^{sc}$ over $G^{sc}$.
\end{proof}

Our proof is based on the following assertion conjectured by Springer,
which was deduced in \cite[Thm A.1]{KV} from results of Lusztig (\cite{Lu}) and Springer (\cite{Sp}).

\begin{Thm} \label{T:Spr}
Let  $L$ be a reductive group over a finite field $\fq$, $\Phi:L\to \C{L}$
a quasi-logarithm,  $\lan\cdot,\cdot\ran$ a non-degenerate invariant
pairing on $\C{L}$, $\ov{T}\subset L$ a maximal torus, $\theta$ a character of $\ov{T}(\fq)$,
$\ov{\psi}$ a character of $\fq$, and $\ov{t}$ an element of $\C{T}(\fq)\cap \C{L}^{\sr}(\fq)$.

Denote by $\dt_t$ the characteristic function of the $\Ad L(\fq)$-orbit of $t$,
by $\C{F}(\dt_t)$ its Fourier transform. Then  for every unipotent $u\in L(\fq)$, we have an equality 
\[
\Tr R^{\theta}_{T}(u)=(-1)^{\rk_{\fq}(L)-\rk_{\fq}(T)}q^{-\frac{1}{2}\dim(L/T)}\C{F}(\dt_t)(\Phi(u)),
\]

\end{Thm}

\begin{Emp}
\begin{proof}[Proof of \rt{topun}]

By \rl{redtu}, we can assume that  $G$ is semisimple and simply connected.
Then by \rl{ql}, there exists an invariant pairing $\lan\cdot,\cdot\ran$ on $\C{G}$, which is
non-degenerate over $\C{O}$,
and a quasi-logarithm $\Phi:G\to\C{G}$ defined over $\C{O}$. Let $\psi$ be an additive character $F\to\B{C}\m$
which is non-degenerate over $\C{O}$.

This data define an invariant pairing on $\mcL_{\vk}$ which is non-degenerate over $\fq$, a quasi-logarithm
$\ov{\Phi}_{\vk}:L_{\vk}\to \mcL_{\vk}$ (see \cite[Lem 1.8.7 (b)]{KV}), and
a non-trivial additive character $\ov{\psi}:\fq\to\B{C}\m$. By our assumption, there exists
$\ov{t}\in\C{T}(\fq)$ such that $d\ov{\fa}(\ov{t})\in \mcL_{\vk}^{G-\sr}(\fq)$ (use \rl{gstr}).
Then using the compatibility of the Fourier transforms
over $F$ and over $\fq$ one deduces from \rt{Spr} (copying the arguments of \cite[Lem 2.2.11]{KV} word-by-word) that for each $u\in G(F)_{tu}$, we have an equality
\begin{equation} \label{Eq:tu0}
\phi^G_{\ov{\fa},\theta}(u)=(-1)^{\rk_{\fq}(L_{\vk})-\rk_{\fq}(\ov{T})}
\C{F}(\dt^G_{\ov{\fa},\ov{t}}(\Phi(u)).
\end{equation}
Also let $\Phi^H:H\to\C{H}$ be the quasi-logarithm  induced by $\Phi$
(see \rp{qlogend}).

Assume first that $H$ is ramified. Then for every pair of $\C{E}$-stably conjugate elements 
$u_H\in H^{\sr}(F)_{\tu}$ and $u\in G^{\sr}(F)_{tu}$, the elements
$\Phi^H(u_H)\in \mcH^{\sr}(F)$ and $\Phi(u)\in G^{\sr}(F)$ are $\C{E}$-stably conjugate (by \rp{qlogst} (b))
and we have an equality

\begin{equation} \label{Eq:katu}
O^{\ka_{u,u_H}}_{u}(\phi^G_{\ov{\fa},\theta})=(-1)^{\rk_{\fq}(L_{\vk})-\rk_{\fq}(\ov{T})}
O^{\ka_{\Phi(u),\Phi^H(u_H)}}_{\Phi(u)}(\C{F}(\dt^G_{\ov{\fa},\ov{t}})).
\end{equation}
By \rp{Lie} (a), function $\C{F}(\dt^G_{\ov{\fa},\ov{t}})$ is $\C{E}$-unstable. Hence by \cite{Wa5} the Fourier
transform $\C{F}(\dt^G_{\ov{\fa},\ov{t}})$ is  $\C{E}$-unstable, thus the right hand side of \form{katu} vanishes.
Therefore the left hand side of \form{katu} vanishes as well.

Assume now that $H$ is unramified and $\wt{\C{E}}$ satisfies the assumption of \re{unr}.
Then it follows from \rp{qlogst} (b),(c)
that for every $u_H\in H(F)_{tu}$ we have an equality
\begin{equation} \label{Eq:tu1}
O^{\wt{\mcE}}_{u_H}(\phi^G_{\ov{\fa},\theta})=(-1)^{\rk_{\fq}(L_{\vk})-\rk_{\fq}(\ov{T})}
O^{\wt{\mcE}}_{\Phi^H(u_H)}(\C{F}(\dt^G_{\ov{\fa},\ov{t}})).
\end{equation}
By \rco{qlo}, $\Phi^H$ is defined over $\C{O}$.
Then equality \form{tu0} for $H$ implies that
\begin{equation} \label{Eq:tu2}
O^{st}_{u_H}(\phi^H_{\ov{\fa},\theta})=(-1)^{\rk_{\fq}(\ov{H})-\rk_{\fq}(\ov{T})}O^{st}_{\Phi^H(u_H)}
(\C{F}(\dt^H_{\ov{\fa},\ov{t}})).
\end{equation}

From this the unramified case of \rt{topun} follows. Indeed, by \rp{Lie} (b), $\dt^H_{\ov{\fa},\ov{t}}$
is an endoscopic transfer of  $\dt^G_{\ov{\fa},\ov{t}}$. Hence by \rt{Wa} and \rp{weil}
we conclude that $(-1)^{\rk_{F}(H)-\rk_{F}(G)}\C{F}(\dt^H_{\ov{\fa},\ov{t}})$ is an endoscopic
transfer of $\C{F}(\dt^G_{\ov{\fa},\ov{t}})$. Since $\rk_{\fq}(L_{\vk})=\rk_{F}(G)$ and
$\rk_{\fq}(\ov{H})=\rk_{F}(H)$, our assertion follows from a combination of equalities
\form{tu1} and  \form{tu2}.
\end{proof}
\end{Emp}

\subsection{The general case}

Now we are ready to prove \rt{main} in general.

\begin{Emp} \label{E:pfofa}
{\bf Proof of (a).}
Recall that we have to show equality \form{mainr} for each pair of  stably conjugate elements 
$\gm_H\in H^{G-\sr}(F)$ and $\gm\in G^{\sr}(F)$. Since $\phi^G_{\ov{\fa},\theta}$ is
supported on  $G_{\vk}$, we may assume that $\gm\in G_{\vk}$.

Let $\gm=su$ be the TJD of $\gm$. Then by \rl{redkap}, we have to show that
$O^{\ka_{\gm,\gm_H}}_{u'}(\phi^{G_{s'}^0}_{\ov{\fb},\theta})=0$ for each stable conjugate
$\gm'\in G_{\vk}$ of $\gm$ with TJD $\gm'=s'u'$.

Since $\gm$ is compact, the element $\gm_H$ is compact (see \rl{tjd} (b)), and we denote by $\gm_H=s_Hu_H$ the
TJD of $\gm_H$. Then $s'$ and $s_H$ are $\C{E}$-stable conjugate (by \rl{tjd} (d)),
so we can form an endoscopic triple $\C{E}_{s'}$ for $G^0_{s'}$ with endoscopic group $(H_{s_H}^0)^*$
 (see \re{endcent} (b)) .
Choose a stable conjugate $u_H^*\in (H_{s_H}^0)^*(F)$ of $u_H\in (H_{s_H}^0)^{\sr}(F)$.
Then $\ka_{\gm,\gm_H}=\ka_{u',u_H^*}$, so we have to show that

Since $H$ is ramified, we get that $H_{s_H}^0$ is ramified. Since $(H_{s_H}^0)^*$ is isomorphic to $H_{s_H}^0$
over $F^{nr}$ (see for example \cite[Prop. 10.1]{Land}), we get that  $(H_{s_H}^0)^*$ is ramified as well.
Since $u'$ is topologically unipotent, the equality $O^{\ka_{u',u^*_H}}_{u'}(\phi^{G_{s'}^0}_{\ov{\fb},\theta})=0$ follows from \rr{ass} and the ramified case of \rt{topun} for the group $G_{s'}$.
\end{Emp}
\begin{Emp} \label{E:pfofb}
{\bf Proof of (b).}
Recall that we have to show equality \form{main} for each $\gm_H\in H^{G-\sr}(F)$.
Assume first that $\gm_H$ is not stably conjugate to an element of $H_{\vk^H}$. Then
$\chi_{H}(\gm_H)\notin c_{H}(\mcO)$ (by \rl{surj} (b)), hence
$\nu\circ \chi_H(\gm_H)\notin c_{G}(\mcO)$ (because $\nu:c_H\to c_G$ is a finite morphism over $\mcO$).
Since $\chi_G(G_{\vk})\subset c_{G}(\mcO)$ (by  \rl{surj} (a)), we conclude that $\gm_H$ is not $\C{E}$-stably
conjugate to an element of $G_{\vk}$. Since  $\phi^G_{\ov{\fa},\theta}$ is supported on $G_{\vk}$, while
$\phi^H_{\ov{\fa},\theta}$ is supported on $H_{\vk^H}$, we conclude that both sides of \form{main} vanish.

By the proven above, we can assume that  $\gm_H$ is stably conjugate to an element of $H_{\vk^H}$.
Since both sides of \form{main} only depend on the stable conjugacy class of $\gm_H$,
we can assume that $\gm_H\in H_{\vk^H}$. Let $\gm_H=s_Hu_H$ be the TJD of $\gm_H$.
Then by \rco{trfac}, it will suffice to show that
$O^{st}_{u_H}(\phi^{H_{s_H}^0}_{\ov{\fb},\theta})= O^{\wt{\C{E}}_{s'}}_{u_H}(\phi^{H_{s_H}^0}_{\ov{\fb},\theta})$
for each $\C{E}$-stable conjugate element $s'\in G_{\vk}$ of $s_H$ and every embedding
$\ov{\fb}:\ov{T}\hra (L_{\vk})_{\ov{s}'}^0$, whose composition
$\ov{\fb}:\ov{T}\hra (L_{\vk})_{\ov{s}'}^0\hra L_{\vk}$ is $L_{\vk}(\fq)$-conjugate to $\ov{\fa}$.

Since $u_H$ is topologically unipotent, the assertion follows from
\re{eetc} (b), \rr{ass} and the unramified case of \rt{topun} for $G_{s'}$.
\end{Emp}

\section{Properties of the transfer factors}

In this section we will prove properties of the transfer factors, which were used in the previous sections.

\subsection{Local Langlands correspondence for tori} \label{SS:abllc}
\begin{Emp} \label{E:rec}
{\bf Langlands pairing.}
It was shown by Langlands \cite{La2} (compare also \cite{Lab}) that for every torus $T$ over $F$ there exists a
functorial reciprocity isomorphism
$\varphi_{T,F}:H_c^1(W_F,\wh{T})\isom\Hom_c(T(F),\B{C}\m)$, where on both sides ${\cdot}_c$ means
{\em continuous}. In particular, there is a canonical pairing
\[
\lan\cdot,\cdot\ran:H_c^1(W_F,\wh{T})\times T(F)\to \B{C}\m.
\]
 \end{Emp}
The goal of this subsection is to show the following result.

\begin{Prop} \label{P:abllc}
Assume that $T$  is  a tamely ramified (resp. unramified) torus over $F$, and that
$c\in  H_c^1(W_F,\wh{T})$ is a tamely ramified (resp. unramified) cohomology class, that is,
the restriction of   $c$ to the wild inertia subgroup $I^{wild}_F\subset W_F$ (resp. inertia subgroup
$I_F\subset W_F$) is trivial.

Then for each $t\in T(F)_{tu}$ (resp. $t\in T(\C{O})$) we have $\lan c,t\ran=1$.
\end{Prop}

We start with a preliminary result of independent interest which
seem to be known to specialists.

\begin{Lem} \label{L:abllc}
For every finite extension $F'/F$ the following diagram is commutative.
\begin{equation} \label{Eq:abllc}
\CD
        H^1_c(W_{F},\wh{T})  @>Res_{W_F/W_{F'}}>>   H^1_c(W_{F'},\wh{T})\\
        @V\varphi_{T,F}VV                        @V\varphi_{T,F'}VV\\
         \Hom_c(T(F),\B{C}\m)  @>N_{F'/F}>>      \Hom_c(T(F'),\B{C}\m)
\endCD
\end{equation}
\end{Lem}
\begin{proof}
The proof is based on the following well-known assertion

\begin{Lem} \label{L:equiv}
Let $A$ be a group, $B\subset A$ a normal subgroup of finite index, and $M$ an $A$-module.
Then for every $i\in\B{N}$,

(a) the homology group $H_i(B,M)$ is naturally an $A/B$-module;

(b) the composition $H_i(B,M)\overset{Cor_{A/B}}{\lra}H_i(A,M)\overset{Res_{A/B}}{\lra}H_i(B,M)$ is the norm
map $N_{A/B}:x\mapsto\sum_{a\in A/B} a(x)$.
\end{Lem}

\begin{proof}
The assertion for the general $i$ follows from that for $i=0$, in which case the assertion follows
from definitions.
\end{proof}

Now we come back to the proof of \rl{abllc}. We start from recalling Langlands construction of
the isomorphism $\varphi_{T,F}$ (compare \cite[p.126]{KS}).

For every splitting field $K\supset F$ of $T$, we denote by $\phi_{T,K}$ the composition
\[
H_1(W_K,X_*(T))\cong H_1(W_K,\B{Z})\otimes_{\B{Z}} X_*(T)\cong
(W_K)^{ab}\otimes _{\B{Z}} X_*(T)\isom K\m\otimes _{\B{Z}} X_*(T)=T(K),
\]
induced by the isomorphism $(W_K)^{ab}\isom K\m$ of the class field theory.

Langlands showed that there exists unique isomorphism of topological groups
$\phi_{T,F}:H_1(W_F,X_*(T))\isom T(F)$ such that for each splitting field $K\supset F$ of $T$,
the composition $H_1(W_F,X_*(T))\overset{\phi_{T,F}}{\lra}T(F)\hra T(K)$
decomposes as
\begin{equation} \label{Eq:CFT}
H_1(W_F,X_*(T))\overset{Res_{W_F/W_K}}{\lra}H_1(W_K,X_*(T))\overset{\phi_{T,K}}{\lra} T(K).
\end{equation}

Since $\B{C}\m$ is an injective $\B{Z}$-module, the topological isomorphism $\phi_{T,F}$ induces an
isomorphism $\varphi_{T,F}$ between
\[
 H_c^1(W_F,\wh{T})= H_c^1(W_F,\Hom(X_*(T),\B{C}\m))\cong\Hom_c(H_1(W_F,X_*(T)),\B{C}\m)
\]
and $\Hom_c(T(F),\B{C}\m)$.

By the definition of $\varphi_{T,F}$, the commutativity of \form{abllc}
is implied by the commutativity of the right square of the diagram
\begin{equation} \label{Eq:abllc2}
\CD
       H_1(W_{F'},X_*(T))   @<Res_{W_F/W_{F'}}<<  H_1(W_F,X_*(T))  @<Cor_{W_F/W_{F'}}<< H_1(W_{F'},X_*(T))\\
      @V\phi_{T,F'}VV                              @V\phi_{T,F}VV                       @V\phi_{T,F'}VV\\
           T(F')                  @<\text{inclusion}<<                 T(F)           @<N_{F'/F}<<           T(F').
\endCD
\end{equation}

We claim that the left square of \form{abllc2} is commutative. Indeed, for each splitting field
$K\supset F'$ of $T$, both compositions $H_1(W_F,X_*(T))\to T(F')\hra  T(K)$ coincide with
\form{CFT}. It follows that the commutativity of the right square of \form{abllc2} is equivalent to
the commutativity of the exterior square of \form{abllc2}.

Note that the map $\phi_{T,K}$ is $\Gal(K/F)$-equivariant (compare \rl{equiv} (a)),
hence the map $\phi_{T,F'}$ is  $\Gal(F'/F)$-equivariant. Therefore the  commutativity of the exterior
square of \form{abllc2} follows from \rl{equiv} (b).
\end{proof}

\begin{Lem} \label{L:norm}
Let $T$ be a tamely ramified (resp. unramified) torus over $F$, and let $F'/F$ be a finite tamely ramified
(resp. unramified) extension. Then the norm map  $N_{F'/F}:T(F')_{tu}\to T(F)_{tu}$
(resp. $N_{F'/F}:T(\C{O}_{F'})\to T(\C{O})$) is surjective.
\end{Lem}
\begin{proof}
Assume first that $T$ and $F'/F$ are unramified, and let $S:=R_{F'/F}T$ be the Weil
restriction of scalars of $T$.
Then  $T$ and  $S$  are smooth tori over $\C{O}$, the norm map $N_{F'/F}$ is a
surjective homomorphism $S\to T$ over $\C{O}$, whose kernel $\Ker N_{F'/F}$ is again a smooth torus over
$\C{O}$. In particular, $N_{F'/F}$ is smooth. Then the surjectivity of
$N_{F'/F}:S(F)_{tu}\to T(F)_{tu}$ follows from Hensel's lemma, while the surjectivity of
$N_{F'/F}:S(\C{O})\to T(\C{O})$ follows
from  a combination of Hensel's lemma and the surjectivity of
$N_{F'/F}:S(\fq)\to T(\fq)$ (Lang's theorem). This completes the proof in the unramified case.

Notice next that for each integer $n$ prime to $p$ the map $t\mapsto t^n$ induces a homeomorphism
$m_n(F):T(F)_{tu}\isom T(F)_{tu}$. Indeed, this follows from classical Hensel's lemma if
$T$ splits over $F$, and it follows from the assertion for $m_n(K)$,
where $K\supset F$ is the splitting field of $T$, in the general case. In particular,
we get the surjectivity of $N_{F'/F}:T(F')_{tu}\to T(F)_{tu}$,
if the degree $[F':F]$ is prime to $p$.

Assume now that  $T$ and $F'/F$ are tamely ramified. Enlarging $F'$, we may assume that
$T$ splits over $F'$. Choose the intermediate field $F\subset F''\subset F'$ such that
$F''/F$ is unramified, and $F'/F''$ is totally ramified. Then
the degree $[F':F'']$ is prime to $p$. Hence, by the proven above,
both $N_{F''/F}:T(F'')_{tu}\to T(F)_{tu}$ and $N_{F'/F''}:T(F')_{tu}\to T(F'')_{tu}$ are
surjective. Therefore their composition  $N_{F'/F}:T(F')_{tu}\to T(F)_{tu}$ is surjective as
well.
\end{proof}

\begin{Emp}
\begin{proof}[Proof of \rp{abllc}]
Let $F'\supset F$ be the splitting field of $T$. By \rl{norm}, there exists
$t'\in  T(F')_{tu}$ (resp. $t'\in T(\C{O}_{F'})$) such that $t=N_{F'/F}(t')$.
Then it follows from \rl{abllc} that  $\lan c,t\ran=\lan Res_{W_F/W_{F'}}(c),t'\ran$.
Extending scalars to $F'$ and  replacing $c$ by $Res_{W_F/W_{F'}}(c)$ and $t$ by $t'$,
we can assume that $T$ splits over $F$. Hence we reduce to the case $T=\B{G}_m$.

In this case, $\varphi_{T,F}: H_c^1(W_F,\B{C}\m)=\Hom_c((W_F)^{ab},\B{C}\m)\isom
\Hom_c(F\m,\B{C}\m)$ is induced by the isomorphism of the class field
theory, so our assertion is well known.
\end{proof}
\end{Emp}

\subsection{Transfer factors for tamely ramified topologically unipotent elements}

The following simple observation is crucial for our work.

\begin{Lem} \label{L:tame1}
If $\eta$, $T$ and $\chi$-data are tamely ramified (resp. unramified), then the cohomology
class  $\inv(\chi)\in H_c^1(W_F,\wh{T})$ is tamely ramified (resp. unramified) as well.
 \end{Lem}

\begin{proof}
In this proof we will freely use notations of \cite{LS}.

Recall (see \cite[(2.6)]{LS}) that the $\chi$-data
$\{\chi_{\al}\}$ give rise to embeddings $\eta_{T,\chi}:{}^L T\hra {}^L G$ and
$\eta^H_{T,\chi}:{}^L T\hra {}^L H$ over $W_F$ and that
$\inv(\chi)\in H_c^1(W_F,\wh{T})$ is the cohomology class of the continuous cocycle
$a_{w}\in  Z_c^1(W_F,\wh{T})$ such that $\eta\circ\eta^H_{T,\chi}(w)=a_{w}\eta_{T,\chi}(w)$
for all $w\in W_F\subset {}^L T$ (compare \cite[(3.5)]{LS}).

Thus it suffices to show that for each $w\in I^{wild}_F$ (resp.  $w\in I_F$) we have
$\eta_{T,\chi}(w)=w\in W_F\subset {}^L G$ and $\eta^H_{T,\chi}(w)=w\in W_F\subset {}^L H$. In other words, it
will suffice to show that each element $r_p(w)n(\om_T(\si))\in\wh{T}$ from \cite[(2.6)]{LS}
equals one.
Since $T$ is tamely ramified (resp. unramified), we get that $\si=1$, hence $\om_T(\si)=1$,
thus $n(\om_T(\si))=1$. Hence it remains to show that $r_p(w)=1$.

In the notation of \cite[(2.5)]{LS}, it will suffice to show that each $\chi_{\al}(v_0(u_i(w)))=1$
(except that in \cite{LS} notation $\la$ is used instead of $\al$). Since $\chi_{\al}$ is
tamely ramified (resp. unramified), it will suffice to check that each $v_0(u_i(w))$ belongs to
$I^{wild}_F$ (resp.  $I_F$). However, $I^{wild}_F$
(resp.  $I_F$) is contained in  $W_{F_{\al}}$ and is a normal subgroup of $W_F$.
Therefore $v_0(u_i(w))$ is equal to $v_0^{-1}(u_i^{-1} w u_i) v_0$, hence it lies in
$I^{wild}_F$ (resp.  $I_F$).
\end{proof}


\begin{Cor} \label{C:tame}
Assume that $p$ does not divide the order of $W_G$ and that the group $G$ and the $\chi$-data
are tamely ramified. Then for every pair of $\mcE$-stably conjugate topologically unipotent 
elements $\gm_H\in H^{\sr}(F)$ and $\gm\in G^{\sr}(F)$, we have an equality $\Dt_{III_2}(\gm_H,\gm)=1$.
\end{Cor}

\begin{proof}
The assumptions imply that the torus $T$ is tamely ramified. Hence the cohomology class
$\inv(\chi)$ is tamely ramified (by \rl{tame1}), so the assertion follows from \rp{abllc}.
\end{proof}

\subsection{Proofs}
\begin{Emp} \label{E:pfisog}
\begin{proof}[Proof of \rl{isog}]
Each embedding $\iota^*$, $a$-data and $\chi$-data for $G_{\pi(\gm)}$ naturally define the corresponding data
for  $(G_0)_{\gm}$, and we claim that we have an equality
$\Dt^{\wt{\C{E}}_0}_{\star}(\gm_H,\gm)=\Dt^{\wt{\C{E}}}_{\star}(\pi_H(\gm_H),\pi(\gm))$ for each  $\star\in\{I,II,III_1,III_2,IV\}$.
The assertion for $\Dt_I, \Dt_{II}$ and  $\Dt_{IV}$ obvious, because $\pi^{ad}$ is an isomorphism,
the assertion for  $\Dt_{III_1}$ follows from \re{inv} (v),
while the assertion for $\Dt_{III_2}$ follows from the functoriality of the reciprocity
isomorphism $\varphi_{T,F}$.
\end{proof}
\end{Emp}

\begin{Emp} \label{E:pfcenter}
\begin{proof}[Proof of \rl{center}]
By \cite[Lem 3.5.A]{LS2}, there exists a character
$\la:Z(G)(F)\to\B{C}\m$ such that an equality $\Dt(z\gm_H,z\gm)=\la(z)\Dt(\gm_H,\gm)$ holds
for all pairs of $\C{E}$-stably conjugate elements $\gm_H\in H^{G-\sr}(F)$ and  $\gm\in G^{\sr}(F)$
and all $z\in Z(G)(F)\subset  Z(H)(F)$.
We have to show that $\la(z)=1$.

Since $H$ is unramified, we can choose $\gm_H\in H^{\sr}(F)$ and $\gm\in G^{\sr}(F)$ such that $H_{\gm_H}\cong G_{\gm}$ is an
unramified torus, and take the $\chi$-data to be unramified.
We claim that $\Dt_{\star}(z\gm_H,z\gm)=\Dt_{\star}(\gm_H,\gm)$ for
all $\star\in\{I,II,III_1,III_2,IV\}$. The assertion for $\Dt_I, \Dt_{II},\Dt_{III_1}$ and  $\Dt_{IV}$
is obvious (for every $\gm_H$), because $z$ is central. Next
$\Dt_{III_2}(z\gm_H,\gm\dt)/\Dt_{III_2}(\dt_H,\dt)=\lan\inv(\chi),z\ran$, which is one by
\rp{abllc}.
\end{proof}
\end{Emp}

\begin{Emp} \label{E:qlogtr}
\begin{proof}[Proof of \rp{qlogst} (c)]
Since $G_{\Phi(\gm)}=G_{\gm}$, we can use the same embedding $\iota^*$ and the same $a$-data and $\chi$-data
in the definition of the transfer factors. Furthermore, since $p\neq 2$, we can assume that the $\chi$-data
are tamely ramified. Since $\gm$ is topologically unipotent, we conclude by
\rco{tame} that $\Dt_{III_2}(\gm_H,\gm)=1$.
Thus it will suffice to show the equality
$\Dt_{\star}(\gm_H,\gm)=\Dt_{\star}(\Phi^H(\gm_H),\Phi(\gm))$
for each $\star\in\{I,II,III_1,IV\}$.

Since $G_{\Phi(\gm)}=G_{\gm}$ and $H_{\Phi^H(\gm_H)}=H_{\gm_H}$, the equality for $\Dt_I$'s and $\Dt_{III_1}$'s
follow. Next we observe that
\[
\Dt_{IV}(\Phi^H(\gm_H),\Phi(\gm))/\Dt_{IV}(\gm_H,\gm)=\prod_{\al}
\left|\frac{d\al(\Phi(\gm))}{\al(\gm)-1}\right|^{1/2},
\]
and
\[
\Dt_{II}(\Phi^H(\gm_H),\Phi(\gm))/\Dt_{II}(\gm_H,\gm)=\prod_{\al}
\chi_{\al}\left(\frac{d\al(\Phi(\gm))}{\al(\gm)-1}\right),
\]
where the products are taken over certain roots of $G$ which are not roots of $H$.

However we have seen in \rp{qlogst} (a)  that each
$\frac{d\al(\Phi(\gm))}{\al(\gm)-1}$ belongs to $1+\frak{m}_{F_{\al}}$. It follows
that for each $\al$ we have $|\frac{d\al(\Phi(\gm))}{\al(\gm)-1}|=1$ and
$\chi_{\al}(\frac{d\al(\Phi(t))}{\al(t)-1})=1$ (since
$\chi_{\al}$ is tamely ramified).
\end{proof}
\end{Emp}

\begin{Emp} \label{E:pftrfac}
\begin{proof}[Proof of \rp{trfac}]
By \rl{tjd} (e), one can assume that the embedding $\iota^*:T\hra G^*$ satisfies $s^*:=\iota^*(s)\in G^*_{\vk^*}$.
Then $(G^*_{s^*})^0$ is quasi-split (by \rl{cent} (b)), so one can assume that $(G_s^0)^*=(G^*_{s^*})^0\subset G^*$.

By the observation of \re{proptf} (e), we have $\Dt^{\wt{\C{E}}}(\gm_H,\gm)=
\inv_{\mcE}(\fa_0^{(3)},\iota^{(3)})\Dt^{\wt{\C{E}}^*}(\gm_H,\gm^*)$ and similarly
$\Dt^{\wt{\C{E}}_s}(u_H,u)=\inv_{\mcE_s}(\fa_s^{(3)},\iota^{(3)})
\Dt^{\wt{\C{E}}^*_s}(u_H,u^*)$. Hence by \re{inv} (iv),(vi), it is enough to show that
$\Dt^{\wt{\C{E}}^*}(\gm_H,\gm^*)=\Dt^{\wt{\C{E}}^*_s}(u_H,u^*)$.

Replacing $G$ by $G^*$, $\wt{\mcE}$ by $\wt{\mcE}^*$ and $\gm$ by $\gm^*$, we can assume that
$G$ is quasi-split and that $\vk$ is hyperspecial. Then using \rl{center}, it remains to
show the equality
\begin{equation} \label{Eq:eqtf}
\Dt^{\wt{\C{E}}}(\gm_H,\gm)=\Dt^{\wt{\C{E}}_s}(\gm_H,\gm).
\end{equation}

Using \cite[Thm 1.6.A]{LS2} and arguing as in \cite[Lem 8.1]{Ha}, we see that the
equality $\Dt^{\wt{\C{E}}}(\gm'_H,\gm')=\Dt^{\wt{\C{E}}_s}(\gm'_H,\gm')$ holds for each
$\gm'\in T(F)$ sufficiently close to $s$ (and $\gm'_H=\iota_H(\gm')$).
In particular, it holds for some $\gm'\in T(F)$ such that $\gm'/s$ is
topologically unipotent and $\gm/\gm'\in G^{sr}(F)$. It remains to show the equality
\begin{equation} \label{Eq:trfac}
\Dt^{\wt{\C{E}}}(\gm_H,\gm)/\Dt^{\wt{\C{E}}}(\gm'_H,\gm')=\Dt^{\wt{\C{E}}_s}(\gm_H,\gm)
/\Dt^{\wt{\C{E}}_s}(\gm'_H,\gm').
\end{equation}

We choose the $\chi$-data for
$G_{\gm}\subset G$ to be tamely ramified and the $\chi$-data for $G_{\gm}\subset G_s^0$
to be the restrictions of those for $G_{\gm}\subset G$. It will suffice to show that
\form{trfac} holds when $\Dt$ is replaces by each $\Dt_{\star}$, where
$\star\in\{I,II,III_2,IV\}$ (see \re{proptf} (d)).

First we claim that for factors $\Dt_I$ and  $\Dt_{III_2}$ both sides of
\form{trfac} equal one. Indeed, the assertion for the $\Dt_I$'s follows from equalities
$G_{\gm}=G_{\gm'}=T$ and $H_{\gm_H}=H_{\gm'_H}=T_H$. Next for the $\Dt_{III_2}$'s \form{trfac}
has the form  $\Dt^{\wt{\C{E}}}_{III_2}(\gm_H/\gm'_H,\gm/\gm')=
\Dt^{\wt{\C{E}}_s}_{III_2}(\gm_H/\gm'_H,\gm/\gm')$. Since  $\gm/\gm'=u/(\gm/s)$ is topologically unipotent,
both sides equal one by \rco{tame}.

Next the quotients LHS/RHS of \form{trfac} for $\Dt_{II}$ and  $\Dt_{IV}$ are equal to
\[
\prod_{\al}\chi_{\al}\left(\frac{\al(\gm)-1}{\al(\gm')-1}\right)\text{  and }
\prod_{\al}\left|\frac{\al(\gm)-1}{\al(\gm')-1}\right|^{1/2},
\]
respectively, where the products are taken over certain roots of $R(G,T)\sm R(G_s,T)$.
However for each  $\al\in R(G,T)\sm R(G_s,T)$, we have
$\frac{\al(\gm)-1}{\al(s)-1}, \frac{\al(\gm')-1}{\al(s)-1}\in 1+\frak{m}_{F_{\al}}$, hence
$\frac{\al(\gm)-1}{\al(\gm')-1}\in 1+\frak{m}_{F_{\al}}$. Since each $\chi_{\al}$
is tamely ramified, we get that $\chi_{\al}\left(\frac{\al(\gm)-1}{\al(\gm')-1}\right)=1$ and  $|\frac{\al(\gm)-1}{\al(\gm')-1}|=1$.
\end{proof}
\end{Emp}


\begin{thebibliography}{99}

\bibitem[BT]{BT}
F. Bruhat and J. Tits, {\em Groupes r\'eductifs sur un corps local. II. Sch\'emas en groupes. Existence d'une
donn\'ee radicielle valu\'ee}, Inst. Hautes \'Etudes Sci. Publ. Math. {\bf 60} (1984), 197--376.

\bibitem[DB]{DB}
S. DeBacker, {\em Parameterizing conjugacy classes of maximal unramified tori via Bruhat--Tits theory},
Michigan Math. J. {\bf 54} (2006), no. 1, 157--178.


\bibitem[DL]{DL}
P. Deligne and  G. Lusztig, {\em Representations of reductive groups over finite fields},
Ann. of Math. (2), {\bf  103} (1976), 103--161.

\bibitem[Ha]{Ha}
T. C. Hales, {\em A simple definition of transfer factors for unramified groups}, in
{\em Representation theory of groups and algebras}, 109--134, Contemp. Math. {\bf 145},
Amer. Math. Soc., Providence, RI, 1993.


\bibitem[JL]{JL}
H. Jacquet and R. P. Langlands, {\em Automorphic forms on $GL(2)$}, Lecture Notes in Mathematics {\bf 114},


\bibitem[Ka]{Ka2}
D. Kazhdan, {\em On lifting}, in {\em  Lie group representations}, II (College Park, Md., 1982/1983), 209--249,
Lecture Notes in Mathematics {\bf 1041}, Springer, Berlin, 1984.


\bibitem[KP]{KP}
D. Kazhdan and A. Polishchuk, {\em Generalization of a theorem of Waldspurger to nice representations}, in
{\em The orbit method in geometry and physics}
(Marseille, 2000), 197--242, Progr. Math. {\bf 213}, Birkh\"auser, Boston, 2003.

\bibitem[KV]{KV}
D. Kazhdan and Y. Varshavsky, {\em Endoscopic decomposition of certain depth zero representations}, in
{\em Studies in Lie theory}, 223--301, Progr. Math. {\bf 243}, Birkh\"auser,
Boston, 2006.


\bibitem[Ko1]{Ko1}
 R. E. Kottwitz, {\em Rational conjugacy classes in reductive groups}, Duke Math. J. {\bf 49} (1982), no. 4, 785--806.

\bibitem[Ko2]{Ko2}
 R. E. Kottwitz, {\em Stable trace formula: cuspidal tempered terms}, Duke Math. J. {\bf 51} (1984), 611--650.

\bibitem[Ko3]{Ko3}
 R. E. Kottwitz, {\em Stable trace formula: elliptic singular terms}, Math. Ann. {\bf 275} (1986), 365--399.

\bibitem[Ko4]{Ko4}
 R. E. Kottwitz, {\em Endoscopy for Hecke Algebras}, Lecture notes of the
Proceeding of Seattle conference, available at www.math.umd.edu/jda/seattle\_proceedings/.

\bibitem[KS]{KS}
 R. E. Kottwitz and D. Shelstad, {\em Foundations of twisted endoscopy},
Ast\'erisque {\bf 255} (1999).


\bibitem[Lab]{Lab}
 J.-P. Labesse, {\em Cohomologie, $L$-groupes et fonctorialit\'e}, Comp. Math. {\bf 55}
(1985),no. 2, 163--184.

\bibitem[Land]{Land}
E. Landvogt, {\em A compactification of the Bruhat-Tits building},
Lecture Notes in Mathematics {\bf 1619},  Springer--Verlag, Berlin, 1996.

\bibitem[La1]{La1}
R. P. Langlands, {\em Stable conjugacy: definitions and lemmas}, Canad. J. Math. {\bf 31} (1979),
no. 4, 700--725.

\bibitem[La2]{La2}
R. P. Langlands, {\em Representations of abelian algebraic groups},  Pacific J. Math. (1997), Special Issue, 231--250.

\bibitem[LS1]{LS}
R. P. Langlands and D. Shelstad, {\em On the definition of transfer factors},
Math. Ann. {\bf 278} (1987) 219-271.

\bibitem[LS2]{LS2}
R. P. Langlands and D. Shelstad, {\em Descent for transfer factors}, in
{\em The Grothendieck Festschrift}, Vol. {\bf II}, 485--563, Progr. Math. {\bf 87},
Birkh\"auser Boston, Boston, MA, 1990.

\bibitem[Lu]{Lu}
 G. Lusztig, {\em Green functions and character sheaves}, Ann. of Math. (2) {\bf 131}
(1990), no. 2, 355--408.

\bibitem[MP]{MP1}
A. Moy and G. Prasad, {\em Unrefined minimal $K$-types for $p$-adic groups}, Invent. Math. {\bf 116}
(1994), no. 1-3, 393--408.

\bibitem[Ngo]{Ngo}
B. C. Ngo, {\em Le lemme fondamental pour les algebres de Lie}, preprint,
arXiv:0801.0446.

\bibitem[Se]{Se}
J.-P. Serre, {\em A course in arithmetic}, Graduate Texts in Mathematics {\bf 7}, Springer-Verlag,
New York-Heidelberg, 1973.



\bibitem[Sp]{Sp}
T. A. Springer, {\em Trigonometric sums, Green functions of finite groups and representations of Weyl groups},
Invent. Math. {\bf 36} (1976), 173--207.

\bibitem[Wa1]{Wa0}
J.-L. Waldspurger, {\em Sur les germes de Shalika pour les groupes
linéaires}, Math. Ann. {\bf 284} (1989), 199--221.

\bibitem[Wa2]{Wa1}
J.-L. Waldspurger, {\em Une formula des traces locale pour les algebres de Lie $p$-adiques}, J. Reine Agew. Math. {\bf 465} (1995)
41-99.

\bibitem[Wa3]{Wa2}
J.-L. Waldspurger, {\em Le lemme fondamental implique le transfert}, Compositio Math. {\bf 105} (1997), 153--236.

\bibitem[Wa4]{Wa5}
J.-L., Waldspurger, {\em Transformation de Fourier et endoscopie}, J. Lie Theory {\bf 10} (2000), 195--206.

\bibitem[Wa5]{Wa3}
J.-L. Waldspurger, {\em Endoscopie et changement de caract\'eristique}, J. Inst. Math.
Jussieu {\bf 5} (2006), no. 3, 423--525.

\bibitem[We]{We}
 A. Weil, {\em Sur certains groupes d'op\'erateurs unitaires}, Acta Math. {\bf 111} (1964)
143--211.




































\end{thebibliography}
\end{document}